\documentclass[a4paper,12pt]{amsart}

\usepackage{framed}
\usepackage{accents}
\usepackage{syntonly}
\usepackage{amsmath,amsfonts,amssymb,amsthm,euscript}
\usepackage{enumerate}
\newcommand{\kappal}{\mbox{\large{$\kappa$}}}
\def\g{g_+}
\def\M{M_+}
\def\Rico{\operatorname{Ric}^{g_+}}
\usepackage{mathrsfs}
\def\scrI{\mathscr{I}}
\usepackage{enumitem}
\usepackage{url}
\usepackage{tikz}
\usetikzlibrary{arrows}
\usetikzlibrary{decorations.pathreplacing}
\def\bbX{{\mathbb{X}}}

\newcommand{\tnabla}{\underline{\nabla}}
\newcommand{\TN}{\mathrm{N}}

\newcommand{\bbL}{\mathbb{L}}
\newcommand{\cZ}{\mathcal{Z}}
 
\newcommand{\og}{\overline{g}}
\newcommand{\occ}{\overline{\cc}}
\providecommand{\pramb}[1]{\accentset{\mathsf{v}}{#1}} % projected ambient
\newcommand{\chnabla}{\pramb{\nabla}}
\newcommand{\bnabla}{\overline{\nabla}}
\newcommand{\TS}{\mathrm{S}}

\newcommand{\bX}{\bar{X}}

\newcommand{\bZ}{\bar{Z}}
\newcommand{\Fialkow}{\mathcal{F}} % Fialkow tensor
\def\Cal{\mathcal}

\newcommand{\up}{\Upsilon}
\newcommand{\Up}{\up}
\newcommand{\ocT}{\overline{\cT}} 

\newcommand{\V}{\mbox{\sf P}}
\newcommand{\J}{\mbox{\sf J}}
\newcommand{\cc}{\boldsymbol{c}}

\newcommand{\ce}{{\Cal E}}
\newcommand{\ct}{{\Cal T}}
\newcommand{\cT}{{\mathcal T}}

\newcommand{\rpl}                         % +) or <+
{\mbox{$
\begin{picture}(12.7,8)(-.5,-1)
\put(0,0.2){$+$}
\put(4.2,2.8){\oval(8,8)[r]}
\end{picture}$}}

\def\cQ{\mathcal{Q}}

\def\cV{\mathcal{V}}
\def\cB{\mathcal{B}}
\def\cF{\mathcal{F}}
\def\cH{\mathcal{H}}

\def\cN{\mathcal{N}}
\def\cU{\mathcal{U}}
\def\cL{\mathcal{L}}

\def\al{\alpha}

\def\ph{\phi}

\def\bg{\mbox{\boldmath $g$}}

\newcommand{\bV}{\mathbb{V}}

\newcommand{\lto}{\longrightarrow}

\newcommand{\om}{\omega}

\newcommand{\si}{\sigma}

\newcommand{\End}{\operatorname{End}}

\newcommand{\lpl}{
  \mbox{$
  \begin{picture}(12.7,8)(-.5,-1)
  \put(2,0.2){$+$}
  \put(6.2,2.8){\oval(8,8)[l]}
  \end{picture}$}}

\newtheorem{theorem}{Theorem}[section]
\newtheorem{lemma}[theorem]{Lemma}
\newtheorem{proposition}[theorem]{Proposition}
\newtheorem{corollary}[theorem]{Corollary}

\theoremstyle{remark}
\newtheorem{remark}[theorem]{\rm\bf Remark}
\newtheorem{problem}[theorem]{Problem}  

\newcommand{\nn}[1]{(\ref{#1})}

\newcommand{\Ric}{\operatorname{Ric}}
\newcommand{\Sc}{\operatorname{Sc}}
\newcommand{\nd}{\nabla}

\def\endrk{\hfill \begin{tikzpicture}
\clip (0,-0.15) rectangle (0.17, 0.15);
\draw [black, line width=90] (0,0) -- (0.17,0);
\end{tikzpicture}}

\def\sideremark#1{\ifvmode\leavevmode\fi\vadjust{\vbox to0pt{\vss% the remark
 \hbox to 0pt{\hskip\hsize\hskip1em%                          will appear only
 \vbox{\hsize3cm\tiny\raggedright\pretolerance10000%          on the side
  \noindent #1\hfill}\hss}\vbox to8pt{\vfil}\vss}}}%
\renewenvironment{leftbar}[1][\hsize]
{%
\MakeFramed{\hsize#1\advance\hsize-\width\FrameRestore}%
}
{\endMakeFramed}
\newenvironment{HM-adds}{\begin{leftbar}}{\end{leftbar}}

\author{Sean Curry \& A.\ Rod Gover}

\title{An introduction to conformal geometry and tractor calculus, with a view to applications in general relativity}
   \thanks{ARG gratefully acknowledges
    support from the Royal Society of New Zealand via Marsden Grant
  13-UOA-018.}
\address{SC \& ARG: Department of Mathematics\\
  The University of Auckland\\
  Private Bag 92019\\
  Auckland 1142\\
 New Zealand} \email{r.gover@auckland.ac.nz}
\email{sean.curry@auckland.ac.nz}


 \newtheorem{exe}{Exercise}
 \newtheorem{prob}{Problem}
\newtheorem{definition}{Definition}
\begin{document}

\begin{abstract}
The following are expanded lecture notes for the course of eight one
hour lectures given by the second author at the 2014 summer school
Asymptotic Analysis in General Relativity held in Grenoble by the
Institut Fourier. The first four lectures deal with conformal geometry
and the conformal tractor calculus, taking as primary motivation the
search for conformally invariant tensors and diffrerential
operators. The final four lectures apply the conformal tractor
calculus to the study of conformally compactified geometries,
motivated by the conformal treatment of infinity in general
relativity.
\end{abstract}

\maketitle

\pagestyle{myheadings} \markboth{Curry \& Gover}{Conformal geometry and GR}

%\maketitle

\tableofcontents

\thanks{ARG gratefully acknowledges support from the Royal
  Society of New Zealand via Marsden Grant 13-UOA-018 }

\section{Introduction}

\begin{definition}\label{cman}
\noindent A conformal $n$-manifold ($n\geq 3$) is the structure $(M,\cc)$ where
\begin{itemize}
\item $M$ is an $n$-manifold,
\item $\cc$ is a  conformal equivalence class of signature $(p,q)$ metrics,\\ i.e. $g,\widehat{g}\in \cc\stackrel{{\rm def.}}{\Longleftrightarrow} \widehat{g}=\Omega^2 g$ and $C^\infty(M)\ni\Omega>0$.
\end{itemize}
\end{definition}

To any pseudo-Riemannian $n$-manifold $(M,g)$ with $n\geq3$ there is an associated conformal manifold $(M,[g])$ where $[g]$ is the set of all metrics $\hat{g}$ which are smooth positive multiples of the metric $g$. In Riemannian signature $(p,q)$ $=$ $(n,0)$ passing to the conformal manifold means geometrically that we are forgetting the notion of lengths (of tangent vectors and of curves) and retaining only the notion of angles (between tangent vectors and curves) and of ratios of lengths (of tangent vectors at a fixed point) associated to the metric $g$. In Lorentzian signature $(n-1,1)$ passing to the conformal manifold means forgetting the ``spacetime interval'' (analogous to length in Riemannian signature) and retaining only the light cone structure of the Lorentzian manifold. On a conformal Lorentzian manifold one also has the notion of angles between intersecting spacelike curves and the notion of orthogonality of tangent vectors at a point, but the conformal structure itself is determined by the light cone structure, justifying the use of the word \emph{only} in the previous sentence. 

The significance of conformal geometry for general relativity largely stems from the fact that the light cone structure determines the \emph{causal structure} of spacetime (and, under some mild assumptions, the causal structure determines the light cone structure). On top of this we shall see in the lectures that the Einstein field equations admit a very nice interpretation in terms of conformal geometry.

In these notes we will develop the natural invariant calculus on conformal manifolds, the (conformal) tractor calculus, and apply it to the study of conformal invariants and of conformally compactified geometries. The course is divided into two parts consisting of four lectures each. The first four lectures deal with conformal geometry and the conformal tractor calculus, taking as primary motivation the problem of constructing conformally invariant tensors and differential operators. The tools developed for this problem however allow us to tackle much more than our original problem of invariants and invariant operators. In the final four lectures they will be applied in particular to the study of conformally compactified geometries, motivated by the conformal treatment of infinity in general relativity. Along the way we establish the connection between the conformal tractor calculus and Helmut Friedrich's conformal field equations. We also digress for one lecture, discussing conformal hypersurface geometry, in order to facilitate the study of the relationship of the geometry of conformal infinity to that of the interior. Finally we show how the tractor calculus may be applied to treat aspects of the asymptotic analysis of boundary problems on conformally compact manifolds. For completeness an appendix has been added which covers further aspects of the conformal tractor calculus as well as discussing briefly the canonical conformal Cartan bundle and connection.

The broad philosophy behind our ensuing discussion is that conformal geometry is important not only for understanding conformal manifolds, or conformally invariant aspects of pseudo-Riemannian geometry (such as conformally invariant field equations), but that it is highly profitable to think of a pseudo-Riemannian manifold as a kind of symmetry breaking (or holonomy reduction) of a conformal manifold whenever there are any (even remotely) conformal geometry related aspects of the problem being considered. Our discussion of the conformal tractor calculus will lead us naturally to the notions of almost Einstein and almost pseudo-Riemannian geometries, which include Einstein and pseudo-Riemannian manifolds (respectively) as well as their respective conformal compactifications (should they admit one). The general theory of Cartan holonomy reductions then enables us to put constraints on the smooth structure and the geometry of conformal infinities of Einstein manifolds, and the tractor calculus enables us to partially generalise these results to pseudo-Riemannian manifolds.

We deal exclusively with conformal manifolds of at least three
dimensions in these notes. That is not to say that two dimensional
conformal manifolds cannot be fitted in to the framework which we
describe. However, in order to have a canonical tractor calculus on
two dimensional conformal manifolds the conformal manifold needs to be
equipped with an extra structure, weaker than a Riemannian structure
but stronger than a conformal structure, called a \emph{M\"obius
  structure}. In higher dimensions a conformal structure determines a
canonical M\"obius structure via the construction of the canonical
conformal Cartan bundle and connection (outlined in the appendix,
Section \ref{CartanConn}). In two dimensions there is no canonical
Cartan bundle and connection associated to a conformal manifold, this
(M\"obius) structure must be imposed as an additional assumption if we
wish to work with the tractor calculus. We note that to any Riemannian
$2$-manifold, or to any (nondegenerate) $2$-dimensional submanifold of
a higher dimensional conformal manifold, there is associated a
canonical M\"obius structure and corresponding tractor calculus.

Also left out in these notes is any discussion of conformal spin geometry. In this case there is again a canonical tractor calculus, known as \emph{spin tractor calculus} or \emph{local twistor calculus}, which is a refinement of the usual conformal tractor calculus in the same way that spinor calculus is a refinement of the usual tensor calculus on pseudo-Riemannian spin manifolds. The interested reader is referred to \cite{Branson-ConfSpin,PenroseMacCallum}.

\subsection{Notation and conventions} We may use abstract indices, or no indices, or frame indices according to convenience. However, we will make particularly heavy use of the abstract index notation. For instance if $L$ is a linear endomorphism of a finite dimensional vector space $V$ then we may choose to write $L$ using    abstract indices as $L^a{_b}$ (or $L^b{_c}$, or $L^{a'}{_{b'}}$, it makes no difference as the indices are just place holders meant to indicate tensor type, and contractions). In this case we would write a vector $v\in V$ as $v^a$ (or $v^b$, or $v^{a'}$, ...) and we would write the action of $L$ on $v$ as $L^a{_b}v^b$ (repeated indices denote tensor contraction so $L^a{_b}v^b$ simply means $L(v)$). Similarly if $w\in V^*$ then using abstract indices we would write $w$ as $w_a$ (or $w_b$, or $w_{a'}$, ...) and $w(v)$ as $w_a v^a$, whereas the outer product $v\otimes w \in \mathrm{End}(V)$ would be written as $v^a w_b$ (or $v^b w_c$, or $w_b v^a$, ...). A covariant 2-tensor $T\in \otimes^2 V^*$ may be written using abstract indices as $T_{ab}$, the symmetric part of $T$ is then denoted by $T_{(ab)}$ and the antisymmetric part by $T_{[ab]}$, that is
$$
T_{(ab)}=\frac{1}{2} \left(T_{ab}+T_{ba}\right) \quad \mbox{and} \quad T_{[ab]}=\frac{1}{2} \left(T_{ab}-T_{ba}\right).
$$
Note that swapping the indices $a$ and $b$ in $T_{ab}$ amounts to swapping the ``slots'' of the covariant two tensor (so that $b$ becomes the label for the first slot and $a$ for the second), this gives in general a different covariant two tensor from $T_{ab}$ whose matrix with respect to any basis for $V$ would be the transpose of that of $T_{ab}$. We can similarly define the symmetric or antisymmetric part of any covariant tensor $T_{ab\cdots e}$ and these are denoted by $T_{(ab\cdots e)}$ and $T_{[ab\cdots e]}$ respectively. We use the same bracket notation for the symmetric and antisymmetric parts of contravariant tensors. Note that we do not have to symmetrise or skew-symmetrise over all indices, for instance $T^{a[bc]}{_d}$ denotes
$$
\frac{1}{2} \left(T^{abc}{_d} - T^{acb}{_d}\right). 
$$

The abstract index notation carries over in the obvious way to vector and tensor fields on a manifold. The virtue of using abstract index notation on manifolds is that it makes immediately apparent the type of tensorial object one is dealing with and its symmetries without having to bring in extraneous vector fields or 1-forms. We will commonly denote the tangent bundle of $M$ by $\ce^a$, and the cotangent bundle of $M$ by $\ce_a$. We then denote the bundle of covariant 2-tensors by $\ce_{ab}$, its subbundle of symmetric 2-tensors by $\ce_{(ab)}$, and so on. In order to avoid confusion, when working on the tangent bundle of a manifold $M$ we will always use lower case Latin abstract indices taken from the beginning of the alphabet ($a$, $b$, etc.) whereas we will take our frame indices from a later part of the alphabet (starting from $i$, $j$, etc.). 

It is common when working with abstract index notation to use the same notation $\ce^a$ for the tangent bundle and its space of smooth sections. Here however we have used the notation $\Gamma(\cV)$ for the space of smooth sections of a vector bundle $\cV$ consistently throughout, with the one exception that for a differential operator $D$ taking sections of a vector bundle $\mathcal{U}$ to sections of $\cV$ we have written
$$
D:\mathcal{U}\rightarrow \cV
$$
in order to simplify the notation.

Consistent with our use of $\ce^a$ for $TM$ we will often denote by $\ce$ the trivial $\mathbb{R}$-bundle over our manifold $M$, so that $\Gamma(\ce)=C^\infty (M)$. When using index free notation we denote the space of vector fields on $M$ by $\mathfrak{X} (M)$, and we use the shorthand $\Lambda^k$ for $\Lambda^k T^*M$ (when the underlying manifold $M$ is understood). Unless otherwise indicated $[\,\cdot\, , \,\cdot\,]$ is the commutator bracket acting on pairs of endomorphisms. Note that the Lie bracket arises in this way when we consider vector fields as derivations of the algebra of smooth functions. In all of the following all structures (manifolds, bundles, tensor fields, etc.) will be assumed smooth, meaning $C^\infty$.
 
\subsubsection{Coupled connections} We assume that the reader is familiar with the notion of a linear connection on a vector bundle $\cV \rightarrow M$ and the special case of an affine connection (being a connection on the tangent bundle of a manifold). Given a pair of vector bundles $\cV$ and $\cV'$ over the same manifold $M$ and linear connections $\nabla$ and $\nabla'$ defined on $\cV$ and $\cV'$ respectively, then there is a natural way to define a connection on the tensor product bundle $\cV\otimes \cV'\rightarrow M$. The \emph{coupled connection} $\nabla^\otimes$ on $\cV\otimes \cV'$ is given on a simple section $v\otimes v'$ of $\cV\otimes \cV'$ by the Leibniz formula
$$
\nabla^{\otimes}_X (v\otimes v') = (\nabla_X v)\otimes v' + v\otimes (\nabla'_X v')
$$
for any $X\in \mathfrak{X}(M)$. Since $\Gamma(\cV\otimes \cV')$ is (locally) generated by simple sections, this formula determines the connection $\nabla^{\otimes}$ uniquely. In order to avoid clumsy notation we will often simply write all of our connections as $\nabla$ when it is clear from the context which (possibly coupled) connection is being used.

\subsubsection{Associated bundles} We assume also that the reader is familiar with the notion of a vector bundle on a smooth manifold $M$ as well as that of an $H$-principal bundle over $M$, where $H$ is a Lie group. If $\pi:\mathcal{G}\rightarrow M$ is a (right) $H$-principal bundle and $\mathbb{V}$ is a (finite dimensional) representation of $H$ then the \emph{associated vector bundle} $\mathcal{G}\times_H \mathbb{V} \rightarrow M$ is the vector bundle with total space defined by
$$
\mathcal{G}\times_H \mathbb{V} = \mathcal{G}\times \mathbb{V} /\sim
$$
where $\sim$ is the equivalence relation 
$$
(u,v) \sim (u\cdot h,h^{-1}\cdot v), \quad h\in H
$$
on $\mathcal{G}\times \mathbb{V}$; the projection of $\mathcal{G}\times_H \mathbb{V}$ to $M$ is simply defined by taking $[(u,v)]$ to $\pi(u)$. For example if $\mathcal{F}$ is the linear frame bundle of a smooth $n$-manifold $M$ then
$$
TM = \mathcal{F}\times_{\mathrm{GL}(n)} \mathbb{R}^n \quad \mbox{and} \quad T^*M = \mathcal{F}\times_{\mathrm{GL}(n)} (\mathbb{R}^n)^*.
$$

Similarly if $H$ is contained in a larger Lie group $G$ then one can extend any principal $H$-bundle $\mathcal{G}\rightarrow M$ to a principal $G$-bundle $\tilde{\mathcal{G}}\rightarrow M$ with total space
$$
\tilde{\mathcal{G}}= \mathcal{G}\times_H G = \mathcal{G}\times G/\sim \quad \mbox{where} \quad (u,g) \sim (u\cdot h,h^{-1}\cdot g).
$$
For example if $(M,g)$ is a Riemannian $n$-manifold and $\mathcal{O}$ denotes its orthonormal frame bundle then
$$
\mathcal{F} = \mathcal{O}\times_{\mathrm{O}(n)}\mathrm{GL}(n)
$$
is the Linear frame bundle of $M$.

\section{Lecture 1: Riemannian invariants and invariant operators}\label{L1}

Recall that if $\nabla$ is an affine connection then its {\em torsion} is the tensor field 
$T\in \Gamma (TM\otimes \Lambda^2 T^*M) $ defined by  
$$
T^{\nabla}(u,v) = \nabla_u v-\nabla_v u -[u,v] \quad \mbox{for all} \quad u,v\in \mathfrak{X}(M).
$$
It is interesting that this is a
tensor; by its construction one might expect a differential operator.
Importantly torsion is an {\em invariant} of connections. On a smooth manifold the map
$$
\nabla\mapsto T^\nabla
$$ 
taking connections to their torsion depends only the smooth
structure. By its construction here it is clear that it is independent
of any choice of coordinates. 

For any connection $\nabla$ on a vector bundle $\cV$ its
curvature is defined by
$$
R^{\nabla}(u,v)W:= [\nabla_u,\nabla_v] W-\nabla_{[u,v]}W \quad \mbox{for all} \quad u,v\in \mathfrak{X}(M), \quad \mbox{and} \quad W\in \Gamma(\cV),
$$
and $R^{\nabla}\in \Gamma(\Lambda^2 T^*M\otimes \End(\cV))$.
Again the map taking connections to their curvatures 
$$
\nabla \mapsto R^{\nabla} 
$$ clearly depends only the (smooth) vector bundle structure of $\cV$
over the smooth manifold $M$. The curvature is an invariant of linear
connections. In particular this applies to affine connections,
i.e. when $\cV=TM$.

\smallskip

\subsection{Ricci calculus and Weyl's invariant theory} 
The most familiar setting for these objects is pseudo-Riemannian
geometry. In this case we obtain a beautiful local calculus that is
sometimes called the Ricci calculus. Let us briefly recall how this
works. Recall that a pseudo-Riemannian manifold consists of an
$n$-manifold $M$ equipped with a metric $g$ of signature $(p,q)$, that is a section 
$g\in \Gamma(S^2T^* M)$  such that pointwise $g$ is non-degenerate and of signature $(p,q)$.
Then $g$ canonically determines a distinguished affine connection called 
the Levi-Civita connection. This is the unique 
connection $\nabla$ satisfying:
\begin{itemize}[itemsep=4pt]
\item $\nabla g =0$ \;\;\;\;\mbox{(metric compatibility), and}
\item  $T^{\nabla}=0$ \;\;\;\;\mbox{(torsion freeness).} 
\end{itemize}

\noindent Thus on a smooth manifold we have a canonical map from each metric to
its Levi-Civita connection
$$
g\mapsto \nabla^g
$$ and, as above, a canonical map which takes each Levi-Civita
connection to its curvature $ \nabla^g\mapsto R^{\nabla^g}, $ called
the {\em Riemannian curvature}. Composing these we get a canonical map that
takes each metric to its curvature
$$
g\longmapsto R^g, 
$$ 
and this map depends only on the smooth structure of $M$. So we say
that $R^g$ is an invariant of the pseudo-Riemannian manifold $(M,g)$. How
can we construct more such invariants? Or ``all'' invariants, in some
perhaps restricted sense?

The first, and perhaps most important, observation is that using the Levi-Civita connection and Riemannian curvature one can proliferate Riemannian invariants. To simplify the explanation let's fix a pseudo-Riemannian manifold $(M,g)$ and use abstract index notation
(when convenient). Then the metric is written $g_{ab}$ and we write
$R_{ab}{}^c{}_d$ for the Riemannian curvature. So if $u,v,w$ are
tangent vector fields then so is $R(u,v)w$ and this is written
$$
u^av^b R_{ab}{}^c{}_d w^d.
$$ 
From curvature we can form the {\em Ricci} and {\em Scalar}
curvatures, respectively:
$$
\Ric_{ab}:=R_{ca}{}^c{}_b \quad \mbox{and}\quad \Sc:=g^{ab}\Ric_{ab},
$$
and these are invariants of $(M,g)$. As also are:
$$
\nabla_b R_{cdef}, \,  \nabla_a\nabla_b R_{cdef}, 
$$
which are {\em tensor valued invariants} and 
$$
\Ric^{ab}\Ric_{ab}\Sc= \vert \Ric \vert^2\Sc, \, \, R_{abcd}R^{abcd}=\vert R\vert^2, \, \, (\nabla_a R_{bcde})\nabla^a R^{bcde}=\vert \nabla R\vert^2,
$$ 
which are some {\em scalar valued invariants}. Here we have used
the metric (and its inverse) to raise and lower indices and
contractions are indicated by repeated indices.

Since this is a practical and efficient way to construct invariants,
it would be useful to know: Do \underline{all local Riemannian
  invariants} arise in this way? That is from partial or complete
contractions of expressions made using $g$, $ R$ and its covariant
derivatives $\nabla\cdots \nabla R$ (and the metric volume form
$vol^g$, if $M$ is oriented). We shall term invariants constructed
this way {\em Weyl invariants}.

Before answering this one first needs to be careful about what is
meant by a local invariant. For example, the following is a reasonable
definition for scalar invariants.
\begin{definition}\label{R-invtdef}
A {\em scalar Riemannian invariant} 
$P$ is a function which assigns to each pseudo-Riemannian $n$-manifold
$(M,g)$  a function
$P(g)$ such that:
\begin{enumerate}
\item[(i)] $P(g)$ is natural, in the sense that
for any diffeomorphism $\phi:M\to M$ we have $P(\phi^* g ) =
\phi^* P(g) $.
\item[(ii)] $P$ is given by a universal polynomial expression of the
  nature that, given a local coordinate system $(x^i)$ on $(M,g)$,
  $P(g)$ is given by a polynomial in the variables $g_{mn}$,
  $\partial_{i_1}g_{mn}$, $\cdots$,
  $\partial_{i_1}\partial_{i_2}\cdots \partial_{i_k} g_{mn}$, $(\det
  g)^{-1}$,  for some positive
  integer $k$.
\end{enumerate}
%\vspace{1cm}
\end{definition}
Then with this definition, and a corresponding definition for
tensor-valued invariants, it is true that all local invariants arise
as {\em Weyl invariants}, and the result goes by the name of {\em
Weyl's classical invariant theory}, see e.g. \cite{AtiyahBP,Stredder}. Given this result, in the following when we mention pseudo-Riemannian invariants we will mean Weyl invariants.

\subsection{Invariant operators, and analysis} \label{iops} 
In a similar way we can use the Ricci calculus to construct invariant
differential operators on pseudo-Riemannian manifolds. For example the 
 (Bochner) Laplacian is given by the formula
\begin{equation}\nonumber
\nabla^a\nabla_a=\Delta: \mathcal{E}\longrightarrow\mathcal{E},
\end{equation}
in terms of the Levi-Civita connection $\nabla$.
 There are also obvious ways to make operators with curvature in coefficients, e.g.
\begin{equation}\nonumber
R_{a\phantom{c}b}^{\phantom{a}c\phantom{b}d}\nabla_c\nabla_d:\mathcal{E}\longrightarrow\mathcal{E}_{(ab)}.
\end{equation}

With suitable restrictions imposed, in analogy with the case of
invariants, one can make the statement that all local natural
invariant differential operators arise in this way. It is beyond our
current scope to make this precise, suffice to say that when we
discuss invariant differential operators on pseudo-Riemannian
manifolds we will again mean operators constructed in this
way.

\begin{remark}
If a manifold has a spin structure, then essentially the above is
still true but there is a further ingredient involved, namely the
Clifford product. This allows the construction of important operators
such as the Dirac operator. 
\end{remark}

The main point here is that the ``Ricci calculus'' provides an
effective and geometrically transparent route to the construction of
invariants and invariant operators. 

Invariants and invariant operators are the basic objects underlying
the first steps (and often significantly more than just the first steps) of
treating problems in general relativity and, more generally, in: 
\begin{itemize}
\item the global analysis of manifolds;
\item the study and application of geometric PDE;
\item Riemannian spectral theory;
\item physics and mathematical physics.
\end{itemize}
Furthermore from a purely theoretical point of view, we cannot claim
to understand a geometry if we do not have a good theory of local
invariants and invariant operators.

\section{Lecture 2: Conformal transformations and conformal covariance}

A good theory of conformal geometry should provide some hope of
treating the following closely related problems:
\begin{prob} Describe a practical way to generate/construct 
(possibly all)
local natural invariants of a conformal structure.
\end{prob}
\begin{prob} Describe a practical way to generate/construct 
(possibly all) natural linear differential operators that are canonical and well-defined on (i.e. are invariants of) a conformal structure.
\end{prob}
\noindent We have not attempted to be precise in these statements, since here they are mainly for the purpose of motivation. Let us first approach these na\"{\i}vely.

\subsection{Conformal Transformations} \label{n-one}

Recall that for any metric $g$ we can associate its Levi-Civita
connection $\nabla$. Let $e_i=\frac{\partial}{\partial
  x^i}=\partial_i$ be a local coordinate frame and $E^i$ its dual. Locally, any
connection is determined by how it acts on a frame field. For the Levi-Civita connection the resulting {\em connection coefficients}
$$
\Gamma^i_{\phantom{i}jk}:= E^i(\nabla_k e_j), \quad \mbox{where}\quad \nabla_k:=\nabla_{e_k} ,
$$ 
are often called the {\em Christoffel symbols}, and are given by the Koszul formula:
\begin{equation}\nonumber
\Gamma^i_{\phantom{i}jk}:= \frac{1}{2}g^{il}\left(g_{lj,k}+g_{lk,j}-g_{jk,l}\right)
\end{equation}
where $g_{ij}=g(e_i,e_j)$ and $g_{lj,k}=\partial_k g_{lj}$.

Using this formula for the Christoffel symbols we can easily compute
the transformation formula for $\nabla$ under a {\em conformal
  transformation} $g\mapsto \widehat{g}=\Omega^2 g$.  Let
$\Upsilon_a:=\Omega^{-1}\nabla_a\Omega$, $v^a\in\Gamma(\mathcal{E}^a)$
and $\om_b\in\Gamma(\mathcal{E}_b)$. Then we have:
\begin{equation}\label{eq:one}
\nabla^{\widehat{g}}_a v^b=\nabla_a v^b+\Upsilon_a v^b-\Upsilon^b v_a+\Upsilon^c v_c \delta_a^b,
\end{equation}
\begin{equation}\label{eq:two}
\nabla^{\widehat{g}}_a \om_b=\nabla_a \om_b-\Upsilon_a \om_b-\Upsilon_b \om_a+\Upsilon^c \om_c g_{a b} .
\end{equation}

For $\om_b\in \Gamma(\ce_b)$ we 
have from \nn{eq:two}
that 
$$
\om_b\mapsto \nabla_a \om_b- \nabla_b \om_a
$$
is conformally invariant. But this is just the exterior derivative $\om \mapsto d\om $, its conformal invariance is better seen from the fact that it
 is defined on a smooth manifold without further structure: for $u,v\in \mathfrak{X}(M)$
$$
\mathrm{d}\om (u,v) = u \om(v)-v\om (u)-\om ( [u,v]). 
$$ Inspecting \nn{eq:two} it is evident that this is the only first
order conformally invariant linear differential operator on $T^*M$
that takes values in an irreducible bundle.

The Levi-Civita connection acts also on other tensor bundles. We
can use the formulae \nn{eq:one} and \nn{eq:two} along with the product
rule to compute the conformal transformation of the result. For
example for a simple covariant 2-tensor
$$
u_b\otimes w_c \quad \mbox{we have} \quad \widehat{\nabla}_a(u_b\otimes w_c)=
(\widehat{\nabla}_a u_b)\otimes w_c+ u_b\otimes (\widehat{\nabla}_a w_c). 
$$ 
Thus we can compute $\widehat{\nabla}_a(u_b\otimes w_c) $ by
using \nn{eq:two} for each term on the right-hand-side.  But locally
any covariant 2-tensor is a linear combination of simple 2-tensors and
so we conclude that for a covariant 2-tensor $F_{bc}$
\begin{equation}\label{Ft}
\widehat{\nabla}_a F_{bc}= \nabla_a F_{bc}- 2\Upsilon_a F_{bc}
-\Upsilon_b F_{ac}-\Upsilon_c F_{ba} + \Upsilon^d F_{d c}g_{ab} +
\Upsilon^d F_{b d}g_{ac} .
\end{equation}
By the obvious extension of this idea one quickly calculates the
formula for the conformal transformation for the Levi-Civita covariant
derivative of an $(r,s)$-tensor.

{From} the formula \nn{Ft} we see that the completely skew part
$\nabla_{[a}F_{bc]}$ is conformally invariant. In the case $F$ is
skew, in that $F_{bc}=-F_{cb}$, this recovers that $dF$ is conformally
invariant. A more interesting observation arises with 
 the
divergence $\nabla^b F_{bc}$.  We have
$$
\widehat{\nabla}^b F_{bc}= \widehat{g}^{ab} \widehat{\nabla}_a F_{bc}, 
$$
and $\widehat{g}^{ab}= \Omega^{-2}g^{ab}$. Thus we obtain 
\begin{equation}\nonumber
\begin{split} \widehat{\nabla}^b F_{bc}=&\, 
    \Omega^{-2}g^{ab}\big( \nabla_a F_{bc}-
2\Upsilon_a F_{bc} -\Upsilon_b F_{ac}-\Upsilon_c F_{ba} + \Upsilon^d
F_{d c}g_{ab} + \Upsilon^d F_{b d}g_{ac} \big)\\
=&\,\Omega^{-2}\big( \nabla^b F_{bc} + (n- 3)\Upsilon^d
F_{d c}+   \Upsilon^d
F_{cd}-\Upsilon_c 
F_{b}{}^b \big).
\end{split}
\end{equation} 
In particular then, if $F$ is skew then
$F_{b}{}^b=0$ and we have simply
\begin{equation}\label{Divt}
 \widehat{\nabla}^b F_{bc}=\Omega^{-2}\big( \nabla^b F_{bc} + (n- 4)\Upsilon^d
F_{d c}\big).
\end{equation}
So we see that something special happens in dimension 4. Combining with 
our earlier observation we have the following result.
\begin{proposition}\label{EM}
In dimension 4 the differential operators 
$$
\operatorname{Div}:\Lambda^2\to \Lambda^1 \quad \mbox{and} \quad 
\operatorname{Max}:\Lambda^1\to \Lambda^1
$$
given by 
$$
F_{ bc}\mapsto \nabla^b F_{ bc}\quad \mbox{and} \quad u_c \mapsto 
\nabla^b\nabla_{[b}u_{c]}
$$
respectively, are conformally covariant, in that 
\begin{equation}\label{Maxw}
 \widehat{\nabla}^b F_{bc}=\Omega^{-2} \nabla^b F_{bc}\quad
\mbox{and} \quad \widehat{\nabla}^b \widehat{\nabla}_{[b} u_{c]}=
\Omega^{-2} \nabla^b \nabla_{[b}u_{c]}.
\end{equation}
\end{proposition}
The non-zero powers of $\Omega$ (precisely $\Omega^{-2}$)
appearing in \nn{Maxw} mean that these objects are only {\em
  covariant} rather than invariant. Conformal covariance is still a
strong symmetry property however, as we shall see. Before we discuss
that in more detail note that, for the equations, these factors of
$\Omega$ make little difference: 
$$
\nabla^b F_{bc}=0 \quad \Leftrightarrow \quad \widehat{\nabla}^b F_{bc}=0
$$
and 
$$ \nabla^b \nabla_{[b}u_{c]}=0 \quad \Leftrightarrow \quad
\widehat{\nabla}^b \widehat{\nabla}_{[b} u_{c]}=0.
$$ In this sense these equations are conformally invariant.  
\begin{remark}\label{Maxeq}
In fact these
equations are rather important. If we add the condition that $F$ is
closed then on Lorentzian signature 4-manifolds the system
$$\mathrm{d}F=0  \quad
\mbox{and} \quad \operatorname{Div}(F)=0
$$ 
is the field strength formulation of the source-free {\em Maxwell
equations} of electromagnetism. The locally equivalent equations
$\operatorname{Div} (d u)=0$ give the potential \label{RS: check
  language} formulation of the (source-free) Maxwell equations. The conformal
invariance of these has been important in Physics. 
 \endrk
 \end{remark}

Returning to our search for conformally covariant operators and
equations, our preliminary investigation suggests that such things
might be rather rare. From \nn{Ft} we see that the divergence of a
2-form is not conformally invariant except in dimension 4. In fact, in
contrast to what this might suggest, there is a rich theory of
conformally covariant operators. However there are some subtleties
involved. Before we come to this it will be useful
to examine how conformal rescaling affects the curvature.

\subsection{Conformal rescaling and curvature}\label{ccurv}

Using \nn{eq:one}, \nn{eq:two} and the observations following these  we can
compute, for example, the conformal transformation formulae for the
Riemannian curvature and its covariant derivatives and so forth. At
the very lowest orders this provides a tractable approach to finding
conformal invariants.

Let us fix a metric $g$. With respect to metric traces, we can
decompose the curvature tensor of $g$ into a trace-free part and a
trace part in the following way:
\begin{equation}\nonumber
R_{abcd}=\underbrace{W_{abcd}}_\textrm{trace-free}+\underbrace{2g_{c[a}P_{b]d}+2g_{d[b}P_{a]c}}_\textrm{trace part}.
\end{equation}
Here $P_{ab}$, so defined, is called the Schouten tensor while the tensor $W_{ab}{}^c{}_d$ is called the Weyl tensor.
In dimensions $n\geq 3$ we have 
$\Ric_{ab}=(n-2) P_{ab}+J g_{ab}$, where  $J:=g^{ab}P_{ab}$. So the Schouten 
tensor $P_{ab}$ is a trace modification of the  Ricci tensor.
\begin{exe}
Prove using \nn{eq:one} that under a conformal transformation 
$g\mapsto \widehat{g}=\Omega^2 g$, as above, the Weyl and Schouten tensors transform as follows:
\begin{equation}\nonumber
W^{\widehat{g}\phantom{ab}c}_{\phantom{\widehat{g}}ab\phantom{c}d}=
W^{g\phantom{ab}c}_{\phantom{g}ab\phantom{c}d}
\end{equation}
and
\begin{equation}\label{Ptrans}
P^{\widehat{g}}_{\phantom{\widehat{g}}ab}=P_{ab}-\nabla_a\Upsilon_b
+\Upsilon_a\Upsilon_b-\frac{1}{2}g_{ab}\Upsilon_c\Upsilon^c .
\end{equation}
\end{exe}
Thus the Weyl curvature is a conformal invariant, while objects such
as $\vert W \vert^2:=W_{abcd}W^{abcd}$ may be called
\textit{conformal covariants} 
because under the conformal change they
pick up a power of the conformal factor $\vert
\widehat{W}\vert^2=\Omega^{-4}\vert W\vert^2$ (where for simplicity we
are hatting the symbol for the object rather than the
metric). We will see shortly that such objects correspond to invariants.

Here we are defining {\em conformal invariants} to be Riemannian
invariants, that have the additional property of being unchanged under
conformal transformation.

\begin{exe}
Consider computing the conformal transformation of the derivatives
of the curvature: $\nabla R$, $\nabla \nabla R$, etcetera. Then
possibly using this and an undetermined coefficient approach to
finding conformal invariants or conformal covariants. This rapidly
gets intractable.
\end{exe}

\subsection{Conformally invariant linear differential operators} \label{cildos}

We may try the corresponding approach for constructing further conformally
invariant linear differential operators:
\begin{equation}\label{cildo}
D^g:\mathcal{U}\rightarrow\mathcal{V}
\end{equation}
such that $D^{\widehat{g}}=D^g$.  Namely, first consider the possible
Riemannian invariant linear differential operators between the bundles
concerned and satisfying some order constraint. Next compute their
transformation under a conformal change. Finally seek a linear
combination of these that forms a conformally invariant operator. For
our purposes this also defines what we mean by conformally invariant
linear differential operator. For many applications we require that
the domain and target bundles $\cU$ and $\cV$ are irreducible. 

It turns out that irreducible tensor (or even spinor) bundles are not
sufficient to deal with conformal operators. Let us see a first
glimpse of this by recalling the construction of one of the most well
known conformally invariant differential operators.

\subsubsection{The conformal Wave operator} \label{Y}

For analysis on pseudo-Riemannian manifolds $(M,g)$ the Laplacian is
an extremely important operator.  The {\em Laplacian} $\Delta$
on functions, which is also called the {\em Laplace-Beltrami operator}, 
is given by
$$
\Delta:= \nabla^a\nabla_a :\ce\to \ce,
$$ 
where $\nabla$ is the Levi-Civita connection for $g$. 

Let us see how this behaves under conformal rescaling. For a function $f$,
$\nabla_a f$ is simply the exterior derivative of $f$ and so is
conformally invariant. So to compute the Laplacian for
$\widehat{g}=\Omega^2 g$ we need only use \nn{eq:two} with $u:=df$:
\begin{equation}\nonumber
\begin{split}
 \Delta^{\widehat{g}}f=&\,\widehat{\nabla}^a\nabla_a f\\
=&\,\widehat{g}^{ab}\widehat{\nabla}_b\nabla_a f\\
=&\,\Omega^{-2}g^{ab}\left(\nabla_b\nabla_a f-\Upsilon_b\nabla_a f 
-\Upsilon_a\nabla_b f+g_{ab}\Upsilon^c\nabla_c f\right).
\end{split}
\end{equation}
So 
\begin{equation}\label{Bt}
\Delta^{\hat{g}} f = 
\Omega^{-2}\left(\Delta f +(n-2)\Upsilon^c\nabla_c f\right).
\end{equation}
By inspecting this formula we learn two things. Let us summarise with a
proposition.
\begin{proposition}\label{lap2}
The Laplacian on functions is conformally covariant in
dimension 2, but not in other dimensions.
\end{proposition}
\noindent This is reminiscent of our observations surrounding the
expression \nn{Divt} and the Maxwell system (cf.\ Proposition \ref{EM}). 

We need to modify the Laplacian to have any hope of obtaining an
invariant operator in other dimensions.  A key idea will be to
introduce a curvature into the formula for a new Laplacian. However
inspecting the formulae for curvature transformation in Section
\ref{ccurv} it is easily seen that this manoeuvre alone will not deal
with the term $\Upsilon^c\nabla_c f$ in \nn{Bt}.

To eliminate the term $(n-2)\Upsilon^c\nabla_c f$ we will make what
seems like a strange move (and we will explain later the mathematics
behind this).  We will allow the domain ``function'' to depend on the choice
of metric in the following way: We decree that the function $f$ on
$(M,g)$ corresponds to $\widehat{f}:=\Omega^{1-\frac{n}{2}}f$ on
$(M,\widehat{g})$, where $\widehat{g}=\Omega^2 g$ as above. Now let us
calculate $\widehat{\Delta}\widehat{f} $. First we have:
\begin{equation}\nonumber
\begin{split}
 \nabla_a (\Omega^{1-\frac{n}{2}}f)=&
 \,\left(1-\frac{n}{2}\right)\Omega^{1-\frac{n}{2}}\Upsilon_a f+
 \Omega^{1-\frac{n}{2}} df\\
=&\,\Omega^{1-\frac{n}{2}}\left(\nabla_a f+\left(1-\frac{n}{2}\right)
\Upsilon_a f\right).
\end{split}
\end{equation}
We use this in the next step:
\begin{equation}\nonumber
\begin{split}
\Delta^{\widehat{g}}\widehat{f}
%% &\,\widehat{\nabla}^a 
%% \left[\Omega^{1-\frac{n}{2}} \left(\nabla_a f+\left(1-\frac{n}{2}\right)
%% \Upsilon_a f \right)\right]\\
=&\,\Omega^{-2} g^{ab} \widehat{\nabla}_b\left[\Omega^{1-\frac{n}{2}}
\left(\nabla_a f+\left(1-\frac{n}{2}\right)\Upsilon_a f \right)\right].\\
%% =&\,\Omega^{-2} g^{ab}\left\lbrace\left(1-\frac{n}{2}\right)
%% \Omega^{1-\frac{n}{2}}\Upsilon_b\left(\nabla_a f+\left(1-\frac{n}{2}\right)
%% \Upsilon_a f\right)\right.\\
%% &\,\hspace{14mm}+\Omega^{1-\frac{n}{2}}\left[\nabla_b\left(\nabla_a f+
%% \left(1-\frac{n}{2}\right)\Upsilon_a f \right)
%% -\Upsilon_b\left(\nabla_a f+\left(1-\frac{n}{2}\right)\Upsilon_a f \right)
%% \right.\\
%% &\,\left.\left.\hspace{20mm}-\Upsilon_a\left(\nabla_b f+
%% \left(1-\frac{n}{2}\right)\Upsilon_b f \right)
%% +g_{ab}\Upsilon^c\left(\nabla_c f+
%% \left(1-\frac{n}{2}\right)\Upsilon_c f \right)\right]\right\rbrace\\
 =&\,\Omega^{-1-\frac{n}{2}}g^{ab}\left[\nabla_b\nabla_a f +
 \left(1-\frac{n}{2}\right)\Upsilon_a\nabla_b f +
 \left(1-\frac{n}{2}\right)\Upsilon_b\nabla_a f+
 \left(1-\frac{n}{2}\right)^2\Upsilon_a\Upsilon_b f\right.\\
 &\,\hspace{18mm}+\left(1-\frac{n}{2}\right)f\nabla_b\Upsilon_a -
 \Upsilon_b\nabla_a f -\left(1-\frac{n}{2}\right)\Upsilon_b\Upsilon_a f
 -\Upsilon_a\nabla_b f \\
 &\,\left.\hspace{18mm}-\left(1-\frac{n}{2}\right)\Upsilon_a\Upsilon_b f
 +g_{ab}\left(\Upsilon^c\nabla_c f+\left(1-\frac{n}{2}\right)\Upsilon^2f\right)
 \right]\\
=&\,\Omega^{-1-\frac{n}{2}}\left[\Delta f+ \left(1-\frac{n}{2}\right)
\left(\nabla^a\Upsilon_a+\Upsilon^2\left(\frac{n}{2}-1\right)\right)f \right],
%% =&\,\Omega^{-1-\frac{n}{2}}\left[\Delta f+ f\left(1-\frac{n}{2}\right)
%% \left(\nabla^a\Upsilon_a+\Upsilon^2\left(n-2+\left(1-\frac{n}{2}\right)\right)
%% \right)\right] ,
\end{split}
\end{equation}
where we used \nn{eq:two} once again and have written $\Upsilon^2$ as a shorthand for $\Upsilon^a\Upsilon_a$.

Now in the last formula for $\Delta^{\widehat{g}}\widehat{f}$  the terms involving $\Upsilon$ there are no terms differentiating $f$.
Thus there is hope of matching this with a curvature transformation.
Indeed contracting \nn{Ptrans} with $\widehat{g}^{-1}$ gives
\begin{equation}\label{Jtrans}
\J^{\widehat{g}}=\Omega^{-2}\left(\J-\nabla^a\Upsilon_a+
(1-\frac{n}{2})\Upsilon^2\right),
\end{equation}
and so
\begin{equation}\label{eq:ast}
\left(\Delta^{\widehat{g}}+\left(1-\frac{n}{2}\right)\J^{\widehat{g}}\right)
\widehat{f}=
\Omega^{-1-\frac{n}{2}}\left(\Delta+\left(1-\frac{n}{2}\right)\J\right)f,
\end{equation}
and we have found a Laplacian operator with a symmetry under conformal
rescaling. Using the relation between $\J$ and the scalar curvature
this is written as in the definition here.
\begin{definition}\label{Tdef}
On a pseudo-Riemannian manifold $(M,g)$ the operator
$$ Y^g: \ce\to \ce \quad \mbox{defined by}\quad Y^g:=\Delta^g+
\frac{n-2}{4(n-1)} \Sc^g
$$
is called the {\em conformal Laplacian} or, in Lorentzian signature, the 
{\em conformal wave operator}. 
\end{definition}
\begin{remark}
On Lorentzian signature manifolds it is often called the
{\em conformal wave operator} because the leading term agrees with the
operator giving usual wave equation. 
It seems that it was in this setting that the operator was first
discovered and applied \cite{Veblen,Dirac}.  On the other hand in the
setting of Riemannian signature $Y$ is often called the {\em Yamabe
operator} because of its role in Yamabe problem of scalar curvature
prescription. \endrk
\end{remark} 
\noindent According to our calculations above this has the following
remarkable symmetry property with respect to conformal rescaling.
\begin{proposition}\label{Yop} The conformal Laplacian is conformally covariant 
in the sense that
$$
Y^{\widehat{g}}\circ \Omega^{1-\frac{n}{2}}= \Omega^{-1-\frac{n}{2}}\circ Y^g.
$$
\end{proposition}

This property of the conformal Laplacian motivates a definition. 
\begin{definition}\label{ccov}
On pseudo-Riemannian manifolds a natural linear differential operator
$P^g$, on a function or tensor/spinor field is said to be a {\em conformally
  covariant operator} if for all positive functions $\Omega$
$$
P^{\widehat{g}} \circ \Omega^{w_1}= \Omega^{w_2}\circ 
P^g   ,
$$ 
where $\widehat{g}=\Omega^2 g$, $(w_1,w_2)\in \mathbb{R}\times \mathbb{R}$, 
and where we view
the powers of
$\Omega$ as multiplication operators.
\end{definition}
\noindent In this definition it is not meant that the
domain and target bundles are necessarily the same. The example in our next
exercise will be important for our later discussions.
\begin{exe}\label{AEex} On pseudo-Riemannian manifolds $(M^{n\geq3},g)$ 
show that
\begin{equation}\label{eq:AE}
\begin{split}
 A^g_{ab}:\mathcal{E}&\longrightarrow\mathcal{E}_{(ab)_0} \; \mbox{ given
   by}\\ f &\longmapsto\nabla_{(a}\nabla_{b)_0} f +\V_{(ab)_0} f
\end{split}
\end{equation}
is conformally covariant with $(w_1,w_2)=(1,1)$. 
That is if 
$\widehat{g}=\Omega^2 g$, for some positive function $\Omega$, then 
$$
A^{\widehat{g}} (\Omega f) = \Omega (A^g f) .
$$
\end{exe}

\subsection{Conformal Geometry}\label{cg}

Recall that we defined a conformal manifold as a manifold $M$ equipped
only with an equivalence class of conformally related metric (see
Definition \ref{cman}).  Conformally covariant operators, as in
Definition \ref{ccov}, have good conformal properties but (in general)
fail to be invariant in the sense of \nn{cildo}. This is not just an
aesthetic shortcoming, it means that they are not well-defined on
conformal manifolds $(M,\cc)$. To construct operators on $(M,g)$ that do
descend to the corresponding conformal structure $(M,\cc=[g])$ we need
the notion of conformal densities.

\subsubsection{Conformal densities and the conformal metric}

\newcommand{\cq}{{\Cal Q}} 
\newcommand{\vol}{\operatorname{vol}}
\newcommand{\Ga}{\Gamma}

Let $(M,\cc)$ be a  {\em conformal
manifold\/} of signature $(p,q)$ (with $p+q=n$).
For a point $x\in M$, and two metrics $g$ and $\hat g$ from the
conformal class, there is an element $s\in\mathbb{R}_+$ such that $\hat
g_x=s^2g_x$ (where the squaring of $s$ is a convenient convention). Thus, we may view the conformal class as being
given by a smooth ray subbundle $\cq\subset S^2T^*M$, whose fibre at $x$ is
formed by the values of $g_x$ for all metrics $g$ in the conformal
class. By construction, $\cq$ has fibre $\mathbb{R}_+$ and the metrics in
the conformal class are in bijective correspondence with smooth
sections of $\cq$.

Denoting by $\pi:\cq\to M$ the restriction to $\cq$ of the canonical
projection $S^2T^*M\to M$, we can view this as a principal bundle with
structure group $\mathbb{R}_+$. The usual convention is to rescale a metric $ g$ to $\hat g=\Omega^2g$. This corresponds to a principal action given by $\rho^s(g_x)=s^2g_x$ for $s\in\mathbb{R}_+$ and $g_x\in \cq_x$, the fibre of $\cq$ over $x\in M$.

Having this, we immediately get a family of basic real line bundles
$\ce[w]\to M$ for $w\in\mathbb{R}$ by defining $\ce[w]$ to be the
associated bundle to $\cq$ with respect to the action of $\mathbb{R}_+$ on
$\mathbb{R}$ given by $s\cdot t:=s^{-w}t$. The usual correspondence
between sections of an associated bundle and equivariant functions on
the total space of a principal bundle then identifies the space
$\Ga(\ce[w])$ of smooth sections of $\ce[w]$ with the space of all
smooth functions $f:\cq\to\mathbb{R}$ such that $f(\rho^s(g_x))=s^w
f(g_x)$ for all $s\in\mathbb{R}_+$. We shall call $\ce[w]$ the bundle of
{\em conformal densities of weight} $w$. Note that each bundle
$\ce[w]$ is trivial, and inherits an orientation from that on
$\mathbb{R}$. We write $\ce_+[w]$ for the ray subbundle consisting of
positive elements. 

If $\widehat{g}=\Omega^2 g\in c$ then the conformally related metrics
$\widehat{g}$ and $g$ each determine sections of $\cq$. We may pull
back $f$ via these sections and we obtain functions on $M$ related by
$$
f^{\widehat{g}}= \Omega^w f^g.
$$ With $w=1-\frac{n}{2}$ this explains the ``strange move'' in Section
\ref{Y} for the domain function of the conformal wave operator; the conformal Wave operator is really an operator on the bundle $\ce[1-\frac{n}{2}]$.

Although the bundle $\ce[w]$ as we defined it depends on the choice of
the conformal structure, it is naturally isomorphic to a density
bundle (which is independent of the conformal structure). Recall that
the bundle of $\al$--densities is associated to the full linear frame
bundle of $M$ with respect to the $1$--dimensional representation
$A\mapsto |\det(A)|^{-\al}$ of the group $GL(n,\mathbb{R})$. In
particular, $2$-densities may be canonically identified with the
oriented bundle $(\Lambda^nT^*M)^2$, and 
$1$--densities are exactly the
geometric objects on manifolds that may be integrated (in a coordinate--independent way).

 To obtain the link with conformal densities, as defined above, recall
that any  metric $g$ on $M$ determines a nowhere vanishing
1--density, the volume density $\vol(g)$. In a local frame, this
density is given by $\sqrt{|\det(g_{ij})|}$, which implies
that for a positive function $\Omega$ we get $\vol(\Omega^2g)=\Omega^n\vol(g)$.
So there is bijective a map from 1-densities to functions $\cq\to\mathbb{R}$ that are homogeneous of degree $-n$ given by the map 
$$
\phi\mapsto \ph(x)/\vol(g)(x),
$$
and this gives an
identification of the $1$-density bundle with $\ce[-n]$ and thus an
identification of $\ce[w]$ with the bundle of
$(-\frac{w}{n})$--densities on $M$. 

So we may think of conformal density bundles as those bundles
associated to the frame bundle via 1-dimensional representations, just
as tensor bundles are associated to higher rank representations. Given any vector bundle $\cB$ we will use the  notation
$$
\cB[w]:=\cB\otimes \ce[w],
$$
and say the bundle is {\em weighted} of weight $w$. Note that $\ce[w]\otimes\ce[w']= \ce[w+ w']$ and that in the above we assume that $\cB$ is not a density bundle itself and is unweighted (weight zero).

Clearly, sections of such weighted bundles may be viewed as homogeneous
(along the fibres of $\cq$) sections of the pullback along $\pi:
\cq\to M$. Now the tautological inclusion of
$\tilde{\bg}:\cq \to \pi^*S^2T^*M$ is evidently homogeneous of degree
2, as for $(s^2g_x,x)\in \cq$, we have $\tilde{\bg}(s^2g_x,x)=
(s^2g_x,x)\in\pi^*S^2T^*M$. So $\tilde{\bg}$ may be identified with a canonical
section of $ \bg\in \Gamma(S^2T^*M[2])$ which provides another
description of the conformal class.  We call $\bg$ the {\em conformal
  metric}. 

Another way to recover $\bg$ is to observe that any metric $g\in \cc$
is a section of $\cq$, and hence determines a section $\si_g\in
\Gamma(\ce_+[1])$ with the characterising property that the
corresponding homogeneous function $\tilde{\si_g}$ on $\cq$ takes the
value 1 along the section $g$. Then 
\begin{equation}\label{easybg}
\bg=(\si_g)^2g
\end{equation}
and it is easily
verified that this is independent of the choice of $g\in c$. Conversely it is clear that any section  $\si\in \Gamma(\ce_+[1])$ determines a metric via
$$ 
g:=\si^{-2}\bg.
$$ 
On a conformal manifold we call $\si\in \Gamma(\ce_+[1])$, or equivalently the corresponding $g\in \cc$, a {\em choice of scale}.

A nice application of $\bg$ is that we can use it to raise, lower,
and contract tensor indices on a conformal manifold, for example
$$
\bg_{ab}:\ce^a\to \ce_b[2] \quad \mbox{by}\quad v^a\mapsto \bg_{ab}v^a, 
$$ just as we use the metric in pseudo-Riemannian geometry. 
Also $\bg$ gives the isomorphism 
\begin{equation}\label{bgcong}
\otimes^n \bg: \big( \Lambda^n TM \big)^2\stackrel{\simeq}{\longrightarrow}\ce[2n].
\end{equation}

\subsubsection{Some calculus with conformal densities}\label{dcalc}

Observe that a choice of scale $g\in c$ determines a connection on
$\ce[w]$ via the formula
\begin{equation}\label{dconn}
\nabla^g \tau : = \si^w \big( \mathrm{d} (\si^{-w} \tau) \big), \quad \tau\in \Gamma(\ce[w]) ,
\end{equation}
where $\mathrm{d}$ is the exterior derivative and $\si\in \Gamma(\ce_+[1])$
satisfies $\bg=\si^2 g$, as $\si^{-w} \tau$ is a function (i.e.\ is
a section of $\ce[0]=\ce$).  Coupling this to the Levi-Civita connection
for $g$ (and denoting both $\nabla^g$) we have at once that $\nabla^g
\si=0$ and hence 
\begin{equation}\label{bgpar}
\nabla^g \bg=0.
\end{equation}

On the other hand the Levi-Civita connection directly determines a
linear connection on $\ce[w]$ since the latter is associated to the
frame bundle, as mentioned above. But \nn{bgpar}, with \nn{bgcong},
shows that this agrees with \nn{dconn}. That is, \nn{dconn} is the
Levi-Civita connection on $\ce[w]$. Thus we have:
\begin{proposition}\label{bgparprop} On a conformal manifold $(M,\cc)$ 
the conformal metric $\bg$ is preserved by the Levi-Civita connection
$\nabla^g$ for any $g\in \cc$.
\end{proposition}

In view of  Proposition \ref{bgparprop} it is reasonable to
use the conformal metric to raise and lower tensor indices even when
working in a scale! We shall henceforth do this unless we state otherwise. 
This enables us to give formulae for natural
conformally invariant operators acting between density bundles. For example 
now choosing $g\in \cc$ and forming the Laplacian, and trace of Schouten,   
$$
\Delta:= \bg^{ab}\nabla_a\nabla_b \quad\mbox{and}\quad \J:=\bg^{ab} P_{ab} ,
$$
we have the following. The  conformal Laplacian can be interpreted as a  differential operator
$$
Y:\ce[1-\frac{n}{2}]\to \ce[-1-\frac{n}{2}], \quad\mbox{given by}\quad 
\Delta+\left(1-\frac{n}{2}\right)\J ,
$$
that is {\em conformally invariant}, meaning that it is well-defined on conformal manifolds.

Similarly the operator $A^g_{ab}$ of \nn{eq:AE} is equivalent to a conformally invariant operator,
$$
A_{ab}:\ce[1] \to \ce_{(ab)_0}[1].
$$ 

\subsubsection{Conformal transformations}\label{ctagain}

The above constructions enable us to understand how conformally covariant
objects may be reinterpreted as objects that descend to invariant
operators on conformal manifolds. A similar result applies to
curvature covariants. However we have not advanced the problem of
constructing these.

It follows at once from the formula \nn{dconn} that under conformal
transformations $g\mapsto \widehat{g}=\Omega^2 g$ the Levi-Civita
connection on $\ce[w]$ transforms by
\begin{equation}\label{lct}
\nabla^{\widehat{g}}_a \tau= \nabla^g_a\tau + w\Upsilon_a \tau,
\end{equation}
since $\si_{\widehat{g}}=\Omega^{-1}\si_g$. We can combine this with
the transformation formulae \nn{eq:one}, \nn{eq:two}, and \nn{Ptrans}
to compute the conformal transformations of weighted tensors and
Riemannian invariants. However this remains a hopeless approach to finding
conformal invariants.

\section{Lecture 3: Prolongation and the tractor connection}

If we are going to be successful at calculating conformal invariants we are going to need a better way to calculate, one which builds in conformal invariance from the start. In the next two lectures we develop such a calculus, the \emph{conformal tractor calculus}, which can be used to proliferate conformally invariant tensor (or tractor) expressions.

For treating conformal geometry it would be clearly desirable to find
an analogue of the Ricci calculus available in the pseudo-Riemannian
setting. On the tangent bundle a conformal manifold $(M,\cc)$ has a
distinguished equivalence class of connections but no distinguished
connection from this class. So at first the situation does not look
promising. However we will see that if we pass from the tangent bundle to a vector bundle with two more dimensions (the standard tractor bundle), then there is indeed a distinguished connection.

There are many ways to see how the tractor calculus arises on a conformal manifold; we will give a very explicit construction which facilitates calculation, but first it will be very helpful to examine how the tractor calculus arises on the flat model space of (Riemannian signature) conformal geometry, the conformal sphere. 

\begin{remark}
Recall that inverse stereographic projection maps Euclidean space conformally into the sphere as a one point (conformal) compactification, so that it makes sense to think of the sphere as the conformally ``flat'' model of Riemannian signature conformal geometry. The real reason however is that the conformal sphere arises naturally as the geometry of a homogeneous space of Lie groups and that conformal geometry can be thought of as the geometry of curved analogues of this homogeneous geometry in the sense of \'Elie Cartan (see, e.g. \cite{CSbook}). \endrk
\end{remark}

%Better way to study? The model.
%Tractors on the model. The (almost Einstein) AE equation - what it
%means geometrically as conformal to Einstein (and link to the model).]

%Idea: There is a very general theory surrounding principal bundles and
%the Levi-Civita connection arises as a special case. But in practise
%its characterisation leads to the Koszul formula and this is what is
%most often used. We want to find the conformal analogue in an
%analogous simple way.

\subsection{The model of conformal geometry -- the conformal 
sphere}\label{model}

We now look at the conformal sphere, which is an extremely important
example. We shall see that the sphere can be viewed as a homogeneous space on which is naturally endowed a conformal structure, and that the conformal tractor calculus arises naturally from this picture. This is the model for Riemannian signature conformal geometry, but a minor variation of this applies to other signatures.

First some notation.  Consider an $(n+2)$-dimensional real vector space
$\bV\cong \mathbb{R}^{n+2}$. Considering the equivalence relation on
$\bV \setminus \{0 \}$ given by
$$ v\sim v', \quad \mbox{if and only if} \quad v'=rv \quad \mbox{for
  some} \quad r>0
$$
 we now write 
$$
\mathbb{P}_+(\bV):= \{ [v] \mid v \in \bV \}
$$ where $[v]$ denotes the equivalence class of $v$. We view this as a
smooth manifold by the identification with one, equivalently any,
round sphere in $\bV$.

Suppose now that $\bV$ is equipped with a non-degenerate bilinear form
$\mathcal{H}$ of signature $(n+1,1)$.  The {\em null cone} $\cN$ of
zero-length vectors forms a quadratic variety in $\bV$. Choosing a time
orientation,  let us
write $\cN_+$ for the forward part of $\cN\setminus \{0 \}$.  Under
the the $\mathbb{R}_+$-ray projectivisation of $\bV$, meaning the
natural map to equivalence classes $\bV\to \mathbb{P}_+(\bV)$, the
forward cone $\cN_+$ is mapped to a quadric in $\mathbb{P}_+(\bV)$.  This 
image is topologically a sphere $S^n$ and we
will write $\pi$ for the submersion $\cN_+\to S^n=\mathbb{P}_+(\cN_+)$.

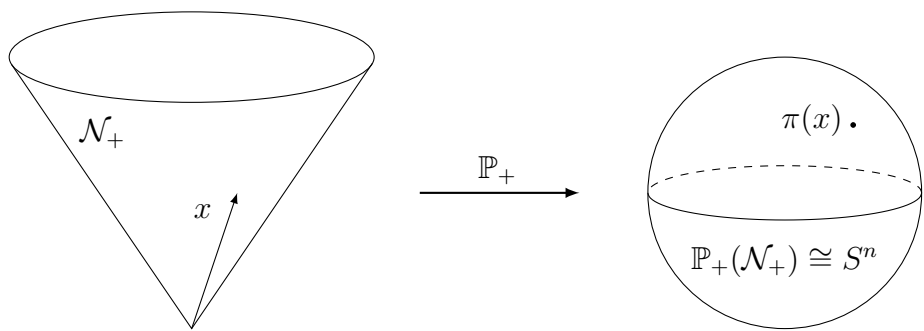
\begin{figure}[ht]
\begin{center}
	\begin{tikzpicture}[xscale=0.6,yscale=0.6]
		%Cone	
		\draw  (0,3) ellipse (4 and 1);
		\draw (-3.95,2.84) -- (0,-3) -- (3.96,2.86);
		\node at (-1.95,1.3) {$\mathcal{N}_+$} ;
		
		%Arrow
		\draw[-latex, thick] (5,0)--(8.5,0);
		\node at (6.7,0.5) {$\mathbb{P}_+$} ;
		
		%Sphere
		\draw  (13,0) circle (3);
		\draw[dashed,xscale=1,yscale=1,domain=0:3.141,smooth,variable=\t] plot ({3*cos(\t r)+13},{0.6*sin(\t r)});
		\draw[xscale=1,yscale=1,domain=3.141:6.283,smooth,variable=\t] plot ({3*cos(\t r)+13},{0.6*sin(\t r)});
		\node at (13,-1.5) {$\mathbb{P}_+(\mathcal{N}_+) \cong S^n$} ;

		%vector and its ray
		\draw [-latex] (0,-3)--(1,0);
		\node at (0.2452,-0.3952) {$x$} ;
		\draw  [fill] (14.5,1.5) ellipse (0.05 and 0.05);
		\node at (13.6,1.6) {$\pi(x)$} ;

	\end{tikzpicture}
	%\caption{The conformal sphere}
	%\label{fig:ConformalSphereModel}
\end{center}
\end{figure}

Each point $x\in \cN_+$ determines a positive definite inner product
$g_x$ on $T_{\pi{(x)}} S^n$ by $g_x(u,v)=\mathcal{H}_x(u',v')$ where
$u',v'\in T_x\cN_+$ are {\em lifts} of $u,v\in T_{\pi{(x)}} S^n$, meaning
that $\pi(u')=u$, $\pi(v')=v$. For a given vector $u\in T_{\pi{(x)}} S^{n}$
two lifts to $x\in \cN_+$ differ by a vertical vector.  By
differentiating the defining equation for the cone we see that any
vertical vector is normal (with respect to $\mathcal{H}$) to the cone,
and so it follows that $g_x$ is independent of the choices of lifts.
Clearly then, each section of $\pi$ determines a metric on $S^n$ and
by construction this is smooth if the section is. Evidently the
metric agrees with the pull-back of $\mathcal{H}$ via the section
concerned. We may choose a coordinates $X^A$, $A=0,\cdots ,n+1$, for
$\bV$ so that $\cN$ is the zero locus of the form $-(X^0)^2+
(X^1)^2+\cdots +(X^{n+1})^2$, in which terms the usual round sphere
arises as the section $X^0=1$ of $\pi$.

Now, viewed as a metric on $T\bV$, $\mathcal{H}$ is homogeneous of
degree 2 with respect to the standard Euler vector field $E$ on $\bV$,
that is $\cL_E \mathcal{H}=2 \mathcal{H}$, where $\cL$ denotes the Lie
derivative. In particular this holds on the cone, which we note is
generated by $E$.  Write $\bg$ for the restriction of $\mathcal{H}$ to
vector fields in $T\cN_+$ which are the lifts of vector fields on
$S^n$. Note that $u'$ is the lift of a vector field $u$ on $S^n$
means that for all $x\in \cN_+$, $d\pi (u'(x))=u(\pi(x))$, 
and so $\cL_E u'=0$ on $\cN_+$.
Thus for any pair $u,v\in \Gamma(TS^n)$, with lifts to vector fields
$u',v'$ on $\cN_+$, $\bg(u',v')$ is a function on $\cN_+$ homogeneous
of degree 2, and which is independent of how the vector fields were
lifted.  It follows that if $s>0$ then $g_{sx}=s^2g_x$, for all $x\in
\cN_+$.  Evidently $\cN_+$ may be identified with the total space
$\mathcal{Q}$  of a bundle of conformally related metrics on $\mathbb{P}_+ (\cN_+)$. Thus $\bg(u',v')$ may be
identified with a conformal density of weight $2$ on $S^n$. That is,
this construction canonically determines a section of $\ce_{(ab)}[2]$
that we shall also denote by $\bg$. This has the property that if
$\si$ is any section of $\ce_+[1]$ then $\si^{-2}\bg$ is a metric
$g^\si$.  Obviously different sections of $\ce_+[1]$ determine
conformally related metrics and, by the last observation in the
previous paragraph, there is a section $\si_o$ of $\ce_+[1]$ so so
that $g^{\si_o}$ is the round metric.

Thus we see that $\mathbb{P}_+(\cN_+)$ is canonically equipped with
the standard conformal structure on the sphere, but with no preferred
metric from this class. Furthermore $\bg$, which arises here from
$\mathcal{H}$ by restriction, is the conformal metric on
$\mathbb{P}_+(\cN_+)$.
In summary then we have the following. 
\begin{lemma}\label{NidQ} Let $\cc$ be the conformal 
class of $S^n=\mathbb{P}_+(\cN_+)$ determined canonically by
$\cH$. This includes the round metric.  The map $\cN_+\ni 
(\pi(x),g_x) \in \cQ$ gives an identification of $\cN_+$ with $\cQ$,
the bundle of conformally related metrics on $(S^n,\cc)$.
 
Via this identification: functions homogeneous of degree $w$ on $\cN_+$
are equivalent to functions homogeneous of degree $w$ on $\cQ$ and
hence correspond to conformal densities of weight $w$ on $(S^n,\cc)$; the 
conformal metric $\bg$ on  $(S^n,\cc)$ agrees with, and is determined by, the restriction of $\cH$ to the lifts of vector fields on
$\mathbb{P}_+(\cN_+)$.
\end{lemma}

The conformal sphere, as constructed here, is acted on transitively by
$G={\rm O}_+(\cH)\cong {\rm O}_+(n+1,1)$, where this is the time
orientation preserving subgroup of orthogonal group preserving
$\mathcal{H}$, ${\rm O}(\cH)\cong {\rm O}(n+1,1)$. Thus as a
homogeneous space $\mathbb{P}_+(\cN_+)$ may be identified with $G/P$,
where $P$ is the (parabolic) Lie subgroup of $G$ preserving a
nominated null ray in $\cN_+$.

\subsubsection{Canonical Calculus on the model}\label{mc}
Here we sketch briefly one way to see this on the model
$\mathbb{P}_+(\cN_+)$.

Note that as a manifold $\mathbb{V}$ has some special structures that
we have already used. In particular an origin and, from the vector
space structure of $\mathbb{V}$, the Euler vector field $E$ which
assigns to each point $X\in \mathbb{V}$ the vector $X\in
T_X\mathbb{V}$, via the canonical identification of $T_X\mathbb{V}$
with $\mathbb{V}$.

The vector
space $\mathbb{V}$ has an affine structure and this induces a global
parallelism: the tangent space $T_x\mathbb{V}$ at any point $x\in
\mathbb{V}$ may be canonically identified with $\mathbb{V}$.  Thus, in
particular, for any parameterised curve in $\mathbb{V}$ there is a
canonical notion of parallel transport along the given curve. This
exactly means that viewing $\mathbb{V}$ as a manifold, it is equipped
with a canonical affine connection $\nabla^{\mathbb{V}}$. The affine
structure gives more than this of course. It is isomorphic to
$\mathbb{R}^{n+2}$ with its usual affine structure, and so the tangent
bundle to $\mathbb{V}$ is trivialised by everywhere parallel tangent
fields. It follows that the canonical connection $\nabla^{\mathbb{V}}$
is flat and has trivial holonomy.

Next observe that $\cH$ determines a signature $(n+1,1)$ metric on
$\mathbb{V}$, where the latter is viewed as an affine manifold. By the
definition of its promotion from bilinear form to metric, one sees at
once that for any vectors $U,V$ that are parallel on $\mathbb{V}$ the
quantity $\cH(U,V)$ is constant. This means that $\cH$ is itself
parallel since for any vector field $W$ we have
$$
(\nabla_W \cH)(U,V)= W\cdot \cH(U,V)- \cH(\nabla_W U,V)- \cH( U,\nabla_W V)=0.
$$

The second key observation is that a restriction of these structures
descend to the model. We observed above that $\cN_+$ is an $\mathbb{R}_+$-ray bundle over $S^n$. We may identify $S^n$ with $\cN_+/\sim$ where the
equivalence relation is that $x\sim y$ if and only if $x$ and $y$ are
points of the same fibre $\pi^{-1}(x')$ for some $x'\in S^n$. The 
restriction $T\mathbb{V}|_{\mbox{\scriptsize{$\cN_+$}}}$
is a rank $n+2$ vector bundle over $\cN_+$.  Now we may define an
equivalence relation on $T\mathbb{V}|_{\mbox{\scriptsize{$\cN_+$}}}$
that covers the relation on $\cN_+$. Namely we decree $U_x\sim V_y$ if
and only if $x,y\in \pi^{-1}(x')$ for some $x'\in S^n$, {\em and} $U_x$
and $V_y$ are parallel. Considering parallel transport up the fibres
of $\pi$, it follows that
$T\mathbb{V}|_{\mbox{\scriptsize{$\cN_+$}}}/\sim$ is isomorphic to the
restriction $T\mathbb{V}|_{\operatorname{im} (S)}$ where $S$ is any
section of $\pi$ (that $S:S^n\to \cN_+$ is a smooth map such that
$\pi\circ S=\operatorname{id}_{S^n}$). But $\operatorname{im} (S)$ is
identified with ${S^n}$ via $\pi$ and it follows that
$T\mathbb{V}|_{\mbox{\scriptsize{$\cN_+$}}}/\sim$ may be viewed as a
vector bundle $\cT$ on ${S^n}$. Furthermore it is clear from the
definition of the equivalence relation on
$T\mathbb{V}|_{\mbox{\scriptsize{$\cN_+$}}}$ that $\cT$ is independent
of $S$. The vector bundle $\cT$ on $S^n$ is the \emph{(standard) tractor bundle} of $(S^n,\cc)$.

\begin{figure}[ht]
\begin{center}
	\begin{tikzpicture}[xscale=0.6,yscale=0.6]
		%Cone	
		\draw  (0,3) ellipse (4 and 1);
		\draw (-3.95,2.84) -- (0,-3) -- (3.96,2.86);
		\node at (-1.95,1.3) {$\mathcal{N}_+$} ;
		
		%Tractor
		\draw [-latex] (0.67,-2)--(2.67,-1);
		\draw [-latex] (1.35,-1)--(3.35,0);
		\draw [-latex] (2.03,0)--(4.03,1);
		\draw [-latex] (2.7,1)--(4.7,2);
		
		\draw [decorate, decoration={brace, amplitude=5pt}, shift={(-0.8,0.1)}] (5.7684,1.8941)--(3.5892,-1.3162);
		\node at (6,0) {\begin{tabular}{c}a tractor \\ at $\pi(x)$ \end{tabular}} ;
		
		%Vector
		%\draw [-latex] (0,-3)--(1.98,-0.07);
		\draw [-latex] (0,-3)--(1.75,-0.41);
		\node at (1.33,-0.3144) {$x$} ;
	\end{tikzpicture}
	\caption{An element of $\cT_{\pi(x)}$ corresponds to a homogeneous of degree zero vector field along the ray generated by $x$.}
\end{center}
\end{figure}
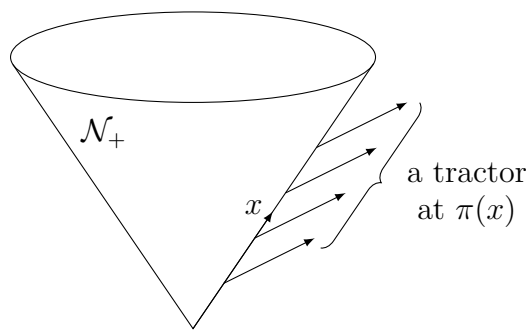

By restriction $\cH$ and $\nabla^{\mathbb{V}}$ determine,
respectively, a (signature $(n+1,1)$) metric and connection on the
bundle $T\mathbb{V}|_{\mbox{\scriptsize{$\cN_+$}}}$ that we shall denote
with the same notation.  Since a vector field which is parallel along
a curve $\gamma$ in $\cN_+$ may be uniquely extended to a vector field
which is also parallel along every fibre of $\pi$ through the curve
$\gamma$, it is clear that $\nabla^{\mathbb{V}}$ canonically
determines a connection on $\cT$ that we shall denote
$\nabla^{\mbox{\scriptsize{$\cT$}}}$. Sections $U,V\in \Gamma(\cT)$
are represented on $\cN_+$ by vector fields $\tilde{U},\tilde{V}$ that
are parallel in the direction of the fibres of $\pi:\cN_+\to {S^n}$.  On the
other hand $\cH$ is also parallel along each fibres of $\pi$ and so
$\cH(\tilde{U},\tilde{V} )$ is constant on each fibre. Thus $\cH$
determines a signature $(n+1,1)$ metric $h$ on $\cT$ satisfying
$h(U,V)= \cH(\tilde{U},\tilde{V} )$. What is more, since
$\nabla^{\mathbb{V}}\cH=0$ on $\cN_+$, it follows that $h$ is preserved
by $\nabla^{\mbox{\scriptsize{$\cT$}}}$, that is
$$
\nabla^{\mbox{\scriptsize{$\cT$}}} h=0.
$$

Summarising the situation thus far we have the following.  
\begin{theorem}\label{mtr}
The model $({S^n},\cc)$ is canonically equipped with the following:
 a canonical rank $n+2$
bundle $\cT$;  a  signature $(n+1,1)$ metric $h$ on this; and a connection
$\nabla^{\mbox{\scriptsize{$\cT$}}}$ on $\cT$ that preserves $h$.
 \end{theorem}

Although we shall not go into details here it is straightforward to show the following:
\begin{proposition}\label{Tstr}
The tractor bundle $\cT$ of the model $({S^n},\cc)$ has 
a composition structure
\begin{equation}\label{mTcomp}
\cT=\ce[1]\lpl T{S^n}[-1]\lpl \ce[-1].
\end{equation}
The restriction of $h$ to the subbundle $\cT^0= T{S^n}[-1]\lpl \ce[-1]$
induces the conformal metric $\bg: T{S^n}[-1]\times T{S^n}[-1]\to \ce$. 
Any null $Y\in \Gamma{\cT}[-1]$ satisfying $h(X,Y)=1$ determines a 
splitting 
$\cT\stackrel{\cong}{\lto} \ce[1]\oplus
T{S^n}[-1]\oplus \ce[-1]$ such that the metric $h$ is given by $(\si,~\mu,~\rho
)\mapsto 2\si \rho+ \bg(\mu,\mu)$ as a quadratic form.
\end{proposition}
It is easily seen how this composition structure arises
geometrically. The subbundle $\cT^0$ of $\cT$ corresponds to the fact
that $T\cN_+$ is naturally identified with a subbundle of
$T\mathbb{V}|_{\cN_+}$. The vertical directions in there correspond to the fact
that $\ce[-1]$ is a subbundle of $\cT^0$, and the {\em semidirect sum}
symbols $\lpl$ record this structure.

\subsubsection{The abstract approach to the tractor connection}
An alternative way of seeing how the tractor connection arises on the flat model is via the group picture. If $G/H$ is a homogeneous space of Lie groups, then the canonical projection $G\rightarrow G/H$ gives rise to a principal $H$-bundle over $G/H$ (with total space $G$). If $\bV$ is a representation of $H$ then one obtains a homogeneous vector bundle $\cV:=G\times_H \bV$ over $G/H$ whose total space is defined to be the quotient of $G\times \bV$ by the equivalence relation
$$
(g,v) \sim (gh,h^{-1}\cdot v).
$$
If $\bV$ is in fact a representation of $G$, then the bundle $\cV:=G\times_H \bV$ is trivial. The isomorphism
$$
G\times_H \bV \cong (G/H)\times \bV
$$
is given by
$$
[g,v]\mapsto (gH,g\cdot v)
$$
which is easily checked to be well defined. This trivialisation gives rise to a flat connection on $G\times_H \bV$. In the case where we have the conformal sphere $G/P\cong S^n$ and $\bV$ is the defining representation of $G$ (i.e. $\mathbb{R}^{n+2}$) then $G\times_P \bV$ is (naturally identified with) the tractor bundle $\cT$ and the connection induced by the trivialisation $G\times_P \bV \cong (G/P)\times \bV$ is the tractor connection. The trivialisation $\cT \cong (G/P)\times \bV$ also immediately gives the existence of a bundle metric preserved by the tractor connection induced by the bilinear form on $\bV$.

\subsubsection{The conformal model in other signatures} \label{modelOS}
With only a little more effort one can treat the case of general
signature $(p,q)$.  This again begins with a real vector space
$\bV\cong \mathbb{R}^{n+2}$, but now we equip it with a non-degenerate
bilinear form $\mathcal{H}$ of signature $(p+1,q+1)$, where $p+q=n$. 

Writing again $\cN$ for the quadratic variety of vectors which have
zero-length with respect to $\cH$, we see that the space of null rays
in $\cN\setminus \{0 \}$ has the topology of $S^p\times S^q$. This is connected, unless $p$ or $q$ is zero in either which case we get two copies of $S^n$. In any case an easy adaption of the earlier discussion shows that $\mathbb{P}_+(\cN)$ is equipped canonically with a conformal structure of signature $(p,q)$ and on this a tractor connection preserving now a tractor metric of signature $(p+1,q+1)$. (One can easily check that the conformal class $\cc$ of $\mathbb{P}_+(\cN)$ contains a metric $g$ which is of the form $g_{S^p}-g_{S^q}$ for some identification of $\mathbb{P}_+(\cN)$ with $S^p\times S^q$.)

As constructed here the conformal space $\mathbb{P}_+(\cN)$ is acted
on transitively by $G={\rm O}(\cH)\cong {\rm O}(p+1,q+1)$, the
orthogonal group preserving $\mathcal{H}$. As a homogeneous space
$\mathbb{P}_+(\cN)$ may be identified with $G/P$, where $P$ is the
Lie subgroup of $G$ preserving a nominated null ray in $\cN$.  We
may think of this group picture as a good model for general conformal
manifolds. Of course there are other possible choices of $G/P$ (which may result in models which are only locally equivalent to the ones here), as we have already seen in the Riemannian case. See \cite{GWsubtleties} for a discussion of the possible choices of $(G,P)$ and the connection with global aspects of the conformal tractor calculus.

\begin{remark} 
Note that the model space $S^1\times S^{n-1}$ of Lorentzian signature conformal geometry is simply the quotient of the (conformal) Einstein universe $\mathbb{R}\times S^{n-1}$ by integer times $2\pi$ translations. Thus the usual embeddings of Minkowski and de Sitter space into the Einstein universe can be seen (by passing to the quotient) as conformal embeddings into the flat model space. In fact, $S^1\times S^{n-1}$ can be seen as two copies of Minkowski space glued together along a null boundary with two cone points, or as two copies of de Sitter space glued together along a spacelike boundary with two connected components which are $(n-1)$-spheres. The significance of this will become clearer as we continue to develop the tractor calculus and then move on to study conformally compactified geometries.  \endrk
\end{remark}

\subsection{Prolongation and the tractor connection}

Here we construct the tractor bundle, connection, and metric on a conformal manifold of dimension at least three. We will see that the  ``conformal to Einstein'' condition plays an important role. The tractor bundle and connection are obtained by ``prolonging'' the ``almost Einstein equation'' \nn{eq:AE}.

First we state what is meant by Einstein here.
\begin{definition}
In dimensions $n\geq 3$, a metric will be said to be {\em Einstein} if 
$$
\Ric^g =\lambda g
$$ 
for some function $\lambda$. 
\end{definition}The Bianchi identities imply that,
for any metric satisfying this equation, $\lambda$ is constant (as we
assume $M$ connected). Throughout the following we shall assume that $n\geq 3$.

Recalling the model above, note that a parallel co-tractor $I$
corresponds to a homogeneous polynomial which, in standard coordinates
$X^A$ on $\mathbb{V}\stackrel{\simeq}{\lto}\mathbb{R}^{n+2}$, is given by
$\tilde{\si}=I_AX^A$. Then $\tilde{\si}=1$ is a section of $\cN_+$
that corresponds to the intersection of $\cN_+$ with the hyperplane
$I_AX^A=1$, in $\bV$. For at least some of these distinguished (conic)
sections the resulting metric on $S^n$, $g=\tilde{\si}^{-2}\bg$ (on
the open set where the corresponding density $\si\in \Gamma (\ce[1])$
is nowhere vanishing) is obviously Einstein. For example the round metric was already discussed. It is in fact true that all metrics obtained on open regions of $S^n$ in this way are Einstein, as we shall shortly see.

\subsection{The almost Einstein equation} \label{AEeq}

Let $(M,\cc)$ be a conformal manifold with $\dim (M)\geq  3$. It is clear that $g^o\in \cc$ is Einstein if and only if $\V^{g^o}_{(ab)_0}=0$, and this is the link with the equation \nn{eq:AE}.  As a conformally invariant
equation this takes the form
\begin{equation}\tag{AE}\label{AE}
\nabla^g_{(a}\nabla^g_{b)_0}\sigma+\V^g_{(ab)_0}\sigma=0,
\end{equation}
where $\si\in \ce_+[1]$ encodes $g^o=\si^{-2}\bg$ and we have used the superscripts to emphasise that
we have picked some metric $g\in \cc$ in order to write the equation.
 
Suppose that $\sigma$ is a solution with the property that it is
nowhere zero. Then, without loss of generality, we may assume that
$\si$ is positive, that is $\sigma\in \mathcal{E}_+[1]$. So $\si$ is
a scale, and $g^o=\sigma^{-2}\bg$ is a well-defined metric.  Since
\nn{eq:AE} is conformal invariant, there is no loss if we
calculate the equation \nn{AE} in this scale. But then
$\nabla^{g^o}\sigma=0$. Thus we conclude that
$$ \V_{(ab)_0}=0.$$

Conversely suppose that $\V^{g^o}_{(ab)_0}=0$ for some $g^o\in c$.
Then $g^o=\sigma^{-2}\bg$ for some $\sigma\in\mathcal{E}_+[1]$.
Therefore $\sigma$ solves \nn{AE} by the reverse of the same argument. 
Thus in summary we have the following, cf. \cite{LeBrunAmbi85}.
\begin{proposition}\label{Ecase}
$(M,\cc)$ is conformally Einstein (i.e.\ there is an Einstein metric $g^o$
in $\cc$) if and only if there exists $\sigma\in \mathcal{E}_+[1]$
that solves \nn{AE}. If $\sigma\in \mathcal{E}_+[1]$ solves
\nn{AE} then $g^o:=\si^{-2}\bg$ is the corresponding
Einstein metric.
\end{proposition}

There are some important points to make here.  
\begin{remark}
Equation \nn{AE} is equivalent to a system of
$\frac{(n+2)(n-1)}{2}$ scalar equations on one scalar variable. So it
is overdetermined and we do not expect it to have solutions in
general.  \endrk  %What Proposition \nn{Ecase} shows is that we may view the
%Einstein equations as integrability conditions for the existence
%of a positive solution to \nn{AE}. 
\end{remark}

\subsection{The connection determined by a conformal structure}\label{AEsec}

Proposition \ref{Ecase}  has the artificial feature that it is a
statement about nowhere vanishing sections of $\ce[1]$. Let us rectify
this.
\begin{definition} Let $(M,\cc)$ be a conformal manifold, of any signature, and $\si\in \ce[1]$. 
We say that $(M,\cc,\sigma)$ is an \emph{almost Einstein structure} if
$\sigma \in \ce[1]$ solves equation \nn{AE}. 
\end{definition}
\noindent We shall term \nn{AE} the (conformal) almost Einstein
equation. 

Since the
almost Einstein equation is conformally invariant it is natural to
seek integrability conditions that are also conformally
invariant. There is a systematic approach to this using a procedure
known as prolongation that goes as follows. We fix a metric $g\in \cc$
to facilitate the calculations. With this understood, for the most
part in the following we will omit the decoration by $g$ of the
natural objects it determines; for example we shall write $\nabla_a$
rather than $\nabla^g_a$.

As a first step observe that the equation \nn{AE} is equivalent to the 
equation
\begin{equation}\label{aeVar}
\nabla_a\nabla_b\sigma+\V_{ab}\sigma+\bg_{ ab}\rho=0
\end{equation}
 where we have introduced the new variable $\rho\in\ce[-1]$ to
absorb the trace terms. 
 The key idea is to attempt to construct an
equivalent first order closed system.  We introduce $\mu_a\in
\mathcal{E}_a[1]$, so our equation is replaced by the equivalent system
 \begin{equation}\label{smu}
 %\begin{split}
 \nabla_a\sigma-\mu_a=0 , \quad \mbox{ and} \quad
 \nabla_a\mu_b+\V_{ab}\sigma+\bg_{ab}\rho=0 ~.
% \end{split}
\end{equation}
This system is almost closed in the sense that the derivatives of
$\sigma$ and $\mu_b$ are given algebraically in terms of the unknowns
$\sigma$, $\mu_b$, and $\rho$. However to obtain a similar result for
$\rho$ we must differentiate the system; by definition {\em
  (differential) prolongation} is precisely concerned with this
process of producing higher order systems, and their consequences. 
 Here we use notation from earlier.

The Levi-Civita covariant derivative of \nn{aeVar} gives
\begin{equation}\nonumber
\begin{split}
\nabla_a\nabla_b\nabla_c \sigma+\bg_{bc}\nabla_a\rho+
(\nabla_a \V_{bc})\sigma+\V_{bc}\nabla_a\sigma=& ~0 .\\
\end{split}
\end{equation}
Contracting this using, respectively, $\bg^{ab}$ and $\bg^{bc}$ yields
 \begin{equation}\nonumber
\begin{split}
\bg^{ab}:\hspace{10mm}\Delta\nabla_c\sigma+\nabla_c\rho+
(\nabla^a \V_{ac})\sigma+\V^{a}_{\phantom{a}c}\nabla_a\sigma=& 
0 \hspace{3mm}(1)\\
\bg^{bc}:\hspace{14mm}\nabla_c\Delta\sigma+n\nabla_c\rho+
(\nabla_c J)\sigma+J\nabla_c\sigma=& 0\hspace{3mm} (2).\\
\end{split}
\end{equation}
Then the difference $(2)-(1)$ is simply 
$$
(n-1) \nabla_c\rho+J\nabla_c\sigma-\V^a_{\phantom{a}c}\nabla_a\sigma+
R_{cb\phantom{b}d}^{\phantom{cb}b}\nabla^d\sigma=0
$$ where we have used the contracted Bianchi identity $\nabla^a
 \V_{ac}=\nabla_c \J$ and computed the commutator $[\nabla_c,\Delta]$
 acting on $\sigma$. But 
$$
R_{cb\phantom{b}a}^{\phantom{cb}b}\nabla^a\sigma=-R_{ca}\nabla^{a}\sigma
=(2-n)\V_c{}^a\nabla_a\sigma-J\nabla_c\sigma ,
$$
and using this we obtain
\begin{equation}\label{rhoeq}
\nabla_c\rho- \V_c^{\phantom{c}a}\mu_a=0,
\end{equation}
after dividing by the overall factor $(n-1)$.  So we have our closed
system and, what is more, this system yields a linear connection. We
discuss this now.

On a conformal manifold $(M,\cc)$ let us write $\left[\mathcal{T}\right]_g$
to mean the pair consisting of a direct
sum bundle and $g\in \cc$, as follows: 
\begin{equation}\label{bform}
\left[\mathcal{T}\right]_g:=\big(
\mathcal{E}[1]\oplus\mathcal{E}_a[1]\oplus\mathcal{E}[-1], g \big)
\end{equation}
\begin{proposition}\label{tp} On a conformal manifold $(M,\cc)$, fix any metric
$g\in \cc$. There is a linear connection $\nabla^{\cT}$ on the bundle
$$
 \left[\mathcal{T}\right]_g \cong \begin{array}{c}
\mathcal{E}[1] \\ 
\oplus \\ 
\mathcal{E}_a[1] \\ 
\oplus\\ 
\mathcal{E}[-1]
\end{array} 
$$
given by 
\begin{equation}\label{cform}
\nabla_a^{\mathcal{T}} \left(\begin{array}{c}
\sigma \\ 
\mu_b \\ 
\rho
\end{array}\right):=\left(
\begin{array}{c}
\nabla_a\sigma -\mu_a \\ 
\nabla_a\mu_b+\bg_{ab}\rho+\V_{ab}\sigma\\ 
\nabla_a\rho-\V_{ab}\mu^b
\end{array} 
\right).
\end{equation}
Solutions of the almost Einstein equation \nn{AE} are in one-to-one
correspondence with sections of the bundle $ [{\cT}]_g$
that are parallel for the  connection $\nabla^\cT$. 
\end{proposition}
\begin{proof}
It remains only to prove that $\nabla^\cT$ is a linear connection. But this is
an immediate consequence of its explicit formula 
which we see takes the form $\nabla+\Phi$ where $\nabla$ is the Levi-Civita
connection on the direct sum bundle $\left[\mathcal{T}\right]_g =
\mathcal{E}[1] \oplus \mathcal{E}_a[1] \oplus \mathcal{E}[-1]$
and $\Phi$ is a section of $\End (\left[\mathcal{T}\right]_g) $.
\end{proof}

We shall call $\nabla^\cT$, of \nn{cform}, the {\em (conformal)
tractor connection}.  Interpreted na\"{\i}vely the statement in the
Proposition might appear to be not very strong: we have already
remarked that on a particular manifold it can be that the equation
\nn{AE} has no non-trivial solutions.  However this is a universal
result, and so it in fact gives an extremely useful tool for
investigating the existence of solutions. Note that the connection
\nn{cform} is well defined on any pseudo-Riemannian manifold. The
point is that in using \nn{cform} for any application, we have
immediately available the powerful theory of linear connections (e.g. parallel transport, curvature, etc.).

It is an immediate consequence of Proposition \ref{tp} that the almost Einstein equation \nn{AE} can have at most $n+2$ linearly independent solutions. However we shall see that far stronger results are available after we refine our understanding of the tractor connection \nn{cform}. Furthermore this connection will be seen to have a role that extends far beyond the almost Einstein equation.

\subsection{Conformal properties of the tractor connection} \label{conft}

Although derived from a conformally invariant equation, the bundle and
connection described in \nn{bform} and \nn{cform} are expressed in a
way depending a priori on a choice of $g\in c$, so we wish to study
their conformal properties. Let us again fix some choice $g\in c$,
before investigating the consequences of changing this conformally.

Given the data of $(\sigma, \mu_b,\rho)\in [{\cT}]_g$ at $x\in M$, it
follows from the general properties of linear connections that we may always
solve 
\begin{equation}\label{atx}
\nabla^\cT (\sigma, \mu_b,\rho)=0, \quad \mbox{at} \quad x\in M.
\end{equation}
This imposes no
restriction on either the conformal class $c$ or the choice $g\in
c$. Examining the formula \nn{cform} for the tractor connection we see
that if \nn{atx} holds then, at the point $x$, we necessarily have
\begin{equation}\label{Dsys}
\mu_b=\nabla_b \sigma, \quad \rho =-\frac{1}{n}(\Delta \si+ \J \sigma),
\end{equation}
where the second equation follows by taking a $\bg^{ab}$ trace of
the middle entry on the right hand side of \nn{cform}. Thus canonically associated to the tractor connection there is the second order differential $ \ce[1]\to [{\cT}]_g$ given by 
\begin{equation}\label{D1form}
[D \sigma ]_g = \left(\begin{array}{c}
n \sigma \\
n \nabla_b \sigma \\
- (\Delta \si+ \J \sigma)
\end{array}\right),
\end{equation} where we have included the normalising factor 
$n$ (=$\dim (M)$) for
later convenience.

Recall that we know how the Levi-Civita connection,
and hence also $\Delta$ and $\J$ here, transform conformally. Thus it
follows that $[D \sigma ]_g$, or equivalently \nn{Dsys}, determines
how the variables $\sigma$, $\mu_b$ and $\rho$ of the prolonged system
must transform if they are to remain compatible with $\nabla^\cT$
under conformal changes. If $\widehat{g}=\Omega^2 g$, for some
positive function $\Omega$, then a brief calculation reveals
$$ \nabla^{\widehat{g}}_b\sigma =\nabla^{g}_b\sigma +
\Upsilon_b\sigma, \quad \mbox{and} \quad
-\frac{1}{n}(\Delta^{\widehat{g}} \si+ \J^{\widehat{g}} \sigma) =
-\frac{1}{n}(\Delta^{g} \si+ \J^{g} \sigma) -\Upsilon^b\nabla_b\sigma
-\frac{1}{2}\sigma\Upsilon^b \Upsilon_b ,
$$
where as usual $\Upsilon$ denotes $d \Omega$. Thus we decree
$$ 
\widehat{\sigma}:= \sigma, \quad \widehat{\mu}_b:=
\mu_b+\Upsilon_b\sigma, \quad \widehat{\rho}:= \rho -\Upsilon^b\mu_b
-\frac{1}{2}\sigma\Upsilon^b \Upsilon_b ,
$$
or, writing $\Upsilon^2:=\Upsilon^a\Upsilon_a$, this may be otherwise
written using an obvious matrix notation:
\begin{equation}\label{ttrans}
[\cT]_{\widehat{g}}\ni \left(\begin{array}{c}
\widehat{\sigma}\\
\widehat{\mu}_b\\
\widehat{\rho}
\end{array}\right)=
\left(\begin{array}{ccc}
1 & 0 & 0\\
\Upsilon_b & \delta^c_b & 0\\
-\frac{1}{2}\Upsilon^2 & -\Upsilon^c & 1
\end{array}\right)\left(\begin{array}{c}
\sigma\\
\mu_c\\
\rho
\end{array}\right)~\sim~ \left(\begin{array}{c}
\sigma\\
\mu_b\\
\rho
\end{array}\right)\in [\cT]_g.
\end{equation} 
Note that at each point $x\in M$ the transformation here is manifestly
by a group action. Thus in the obvious way this defines an equivalence
relation among the direct sum bundles $[\cT]_g$ (of \nn{bform}) that
covers the conformal equivalence of metrics in $\cc$, and the quotient
by this defines what we shall call the conformal standard tractor
bundle $\cT$ on $(M,\cc)$. More precisely we have the following
definition.
\begin{definition}\label{cTdef} On a conformal manifold $(M,\cc)$
the {\em standard tractor bundle} is 
$$
\cT:= \bigsqcup_{g\in c} [\cT]_g /\sim ,
$$ meaning the disjoint union of the $[\cT]_g$ (parameterised by $g\in
c$) modulo equivalence relation given by \nn{ttrans}. We shall also use the 
abstract index notation $\ce_A$ for $\cT$.
\end{definition}
We shall carry many conventions from tensor calculus over to bundles
and tractor fields. For example we shall write $ \ce_{(AB)} [w]$ to
mean $S^2 \cT \otimes \ce[w]$, and so forth.

There are some immediate consequences of the Definition \ref{cTdef}
that we should observe. First, from this definition, the next statement
follows tautologically.
\begin{proposition}\label{D1}
The formula \nn{D1form} determines a conformally invariant
differential operator
$$
D:\ce[1]\to \cT.
$$
\end{proposition}
This operator is evidently intimately connected with the very
definition of the standard tractor bundle. Because of its fundamental role we
make the following definition.   
\begin{definition}\label{scaletractor}
Given a section of $\sigma\in \ce[1]$ we shall call 
$$
I:= \frac{1}{n}D \sigma
$$
the {\em scale tractor} corresponding to $\si$. 
\end{definition} 

Next observe also that from \nn{ttrans} it is clear that $\cT$ is a
filtered bundle; we summarise this using a semi-direct sum notation
\begin{equation}\label{Tcomp}
\cT= \ce[1]\lpl \ce_a[1]\lpl \ce[-1]
\end{equation}
meaning that $\cT$ has a subbundle $\cT^1\subset \cT$ isomorphic to
$\ce[-1]$, $\ce_a[1]$ is isomorphic to a subbundle of the quotient
$\cT/\cT^1$ bundle, and $\ce[1]$ is the final factor. 
We use $X$, to denote the bundle surjection $X: \cT \to \ce[1]$
or in abstract indices:
\begin{equation}\label{cant}
X^A :\ce_A\to \ce[1].
\end{equation}
Note that given any metric $g\in \cc$ we may interpret $X$ as the map $[\cT]_g \to \ce[1]$ given by
$$
\left(\begin{array}{c}
\si\\
\mu_b\\
\rho
\end{array}\right) \mapsto \si.
$$

For our current purposes the critical result at this point is that the tractor
connection $\nabla^\cT$ ``intertwines'' with the transformation
\nn{ttrans} in the following sense.
\begin{exe} \label{tint}
Let $V=(\sigma, \mu_b, \rho)$, a section of $[\cT]_g$, and
$\widehat{V}=(\widehat{\sigma}, \widehat{\mu}_b, \widehat{\rho})$, a section
of $[\cT]_{\widehat{g}}$, be
related by \nn{ttrans}, where  $\widehat{g}=\Omega^2 g$.  Show that then
\begin{equation}\nonumber
\left(\begin{array}{c}
\widehat{\nabla_a\sigma-\mu_a}\\
\widehat{\nabla_a\mu_b+g_{ab}\rho+\V_{ab}\sigma}\\
\widehat{\nabla_a\rho-\V_{ab}\mu^b}
\end{array}
\right)=
\left(\begin{array}{ccc}
1 & 0 & 0\\
\Upsilon_b & \delta^c_b & 0\\
-\frac{1}{2}\Upsilon^2 & -\Upsilon^c & 1
\end{array}\right)
\left(
\begin{array}{c}
\nabla_a \sigma-\mu_a\\
\nabla_a\mu_b+g_{ac}\rho+\V_{ac}\sigma\\
\nabla_a\rho-\V_{ac}\mu^c
\end{array}
\right).
\end{equation}
Here $\Upsilon =d\Omega$, as usual, and for example
$\widehat{\nabla_a\sigma-\mu_a}$ means
$\widehat{\nabla}_a\widehat{\sigma}-\widehat{\mu}_a$.
\end{exe}
Given a tangent vector field $v^a$, the exercise shows that
$v^a\nabla_a^\cT V$ {\em transforms conformally as a standard
tractor field}, that is by \nn{ttrans}. Whence $\nabla^\cT$ descends to a
well defined connection on $\cT$. Let us summarise as follows.
\begin{theorem}\label{tthm}
Let $(M,\cc)$ be a conformal manifold of dimension at least 3. The
formula \nn{cform} determines a conformally invariant connection 
$$
\nabla^\cT: \cT\to \Lambda^1\otimes \cT~.
$$
\end{theorem}
\noindent For obvious reasons this will also be called the {\em conformal
tractor connection} (on the standard tractor bundle $\cT$); the formula
\nn{cform}  is henceforth regarded as the incarnation of this
conformally invariant object on the realisation $[\cT]_g$ of $\cT$, as 
determined by the choice $g\in \cc$.

It is important to realise that the conformal tractor connection
exists canonically on any conformal manifold (of dimension at least
3). (One also has the tractor connection on 2 dimensional \emph{M\"obius} conformal manifolds). In particular its existence does not rely on solutions to the equation \nn{AE}.  Nevertheless by its construction in
Section \ref{AEsec} (as a prolongation of the equation \nn{AE}), and
using also equation \nn{Dsys}, we have at once the following important
property. 
\begin{theorem} \label{para}
On a conformal manifold $(M,\cc)$  we have the following. 
There is a 1-1 correspondence
between sections $\si\in \ce[1]$, satisfying the conformal equation
$$ \mbox{\rm (AE)} \quad \quad \quad
\nabla_{(a}\nabla_{b)_0}\sigma+\V_{(ab)_0}\sigma=0,
$$ 
and parallel standard tractors $I$.
The mapping from
almost Einstein scales to parallel tractors is given by $\si\mapsto
\frac{1}{n}D_A \si$ while the inverse map is $I_A \mapsto X^A I_A$.
\end{theorem}
\noindent So a parallel tractor is necessarily a scale tractor, as in
Definition \ref{scaletractor}, but in general the converse does not hold.

\subsection{The tractor metric}\label{msec}
It turns out that the tractor bundle has beautiful and important
structure that is perhaps not at expected from its origins via
prolongation in Section \ref{AEsec} above.
\begin{proposition} \label{metp} Let $(M,\cc)$ be a conformal manifold of signature $(p,q)$. The formula
\begin{equation}\nonumber 
[V_A]_g=(\sigma,\mu_a,\rho)\mapsto 2\sigma\rho+\bg^{ab}\mu_a\mu_b=:h(V,V)
\end{equation}
defines, by polarisation, a signature $(p+1,q+1)$ metric on $\mathcal{T}$.
\end{proposition}
\begin{proof}
As a symmetric bilinear form field on the bundle $[\cT]_g$, $h$ takes the form
\begin{equation}\label{mform}
h(V',V)\stackrel{g}{=}
\begin{array}{c c c} \boldsymbol{(}~\sigma' & \mu' & \rho'~\boldsymbol{)} \\ 
&&\\
&&
\end{array} 
\left(
\begin{array}{c c c}
0 & 0 & 1\\ 
0 & \phantom{1}\bg^{-1} & 0\\
1 & 0 & 0
\end{array}
\right)\left(
\begin{array}{c}
\sigma\\ 
 \mu \\
\rho
\end{array}
\right),
\end{equation}
where $\stackrel{g}{=}$ should be read as ``equals, calculating in the
scale $g$''. So we see that the signature is as claimed. By
construction $h(V,V')$ has weight 0.  It remains to check the
conformal invariance. Here we use the notation from \nn{ttrans}:
\begin{equation}\nonumber
\begin{split}
2\widehat{\sigma}\widehat{\rho}+\widehat{\mu}^a\widehat{\mu}_a=&
2\sigma(\rho-\Upsilon_c\mu^c-\frac{1}{2}\Upsilon^2\sigma)+
(\mu^a+\Upsilon^a\sigma)(\mu_a+\Upsilon_a\sigma)\\
=& 2\sigma\rho+\mu^a\mu_a-2\sigma\Upsilon\mu-\sigma^2\Upsilon^2+
2\Upsilon\mu\sigma+\Upsilon^2\sigma^2\\
=& 2\sigma\rho+\mu^a\mu_a
\end{split}
\end{equation}
\end{proof}
In the abstract index notation the tractor metric is  $h_{AB}\in \Gamma(\ce_{(AB)})$,
and its inverse $h^{BC}$. 
 
The standard
tractor bundle, as introduced in Sections \ref{AEsec} and \ref{conft},
would more naturally have been defined as the dual tractor bundle. But
Proposition \nn{metp} shows that we have not damaged our development;
the tractor bundle is canonically isomorphic to its dual and normally
we do not distinguish these, except by the raising and lowering of
abstract indices using $h_{AB}$.

From these considerations we see that there is no ambiguity in viewing
$X$, of \nn{cant}, as a section of $\cT\otimes \ce[1]\cong \cT[1]$. In
this spirit we refer to $X$ as the {\em canonical tractor}.  At this
point it is useful to note that in view of this canonical self duality
$\cT\cong\cT^*$, and the formula \nn{mform}, we have the following result.
 \begin{proposition}\label{Xnull}
The canonical tractor $X^A$ is null, in that 
$$
h_{AB}X^AX^B =0.
$$ 
Furthermore $X_A =h_{AB}X^B$ gives the canonical inclusion of
$\ce[-1]$ into $\ce_A$:
$$
X_A: \ce[-1]\to \ce_A. 
$$
In terms of the decomposition of $\ce_A$ given by a choice of metric this is simply
$$
\rho \mapsto \left(
\begin{array}{c}
0 \\ 
0 \\
\rho
\end{array}
\right).
$$
\end{proposition}

As another immediate application of the metric we observe the
following.  Given a choice of scale $\si\in \ce_+[1]$, consider the
corresponding scale tractor $I$, and in particular is squared length
$h(I,I)$. This evidently has conformal weight zero and so is a
function on $(M,\cc)$ determined only by the choice of scale
$\si$. Explicitly we have
\begin{equation}\label{I2}
h(I,I) \stackrel{g}{=} \bg^{ab}(\nabla_a \si)(\nabla_b \si) -
\frac{2}{n} \si (\J+\Delta )\si
\end{equation}
from Definition \ref{scaletractor} with \nn{D1form} and
\nn{mform}. Here we have calculated the right hand side in terms of some metric $g$ in the conformal class (hence the notation $\stackrel{g}{=}$). But since $\si$ is a scale we may, in particular, use $g:=\si^{-2}\bg$. Then $\nabla^g \si=0$ and we find the following result.
\begin{proposition}\label{scalarcurv}
On a conformal manifold $(M,\cc)$, let $\sigma\in \ce_+[1]$, and $I=
\frac{1}{n}D\si $ the corresponding scale tractor.  Then
$$
h^{AB} I_B I_C = - \frac{2}{n} \J^\si ,
$$
where $\J^\si:=g^{ab} \V_{ab}$, and $g_{ab}=\si^{-2}\bg_{ab}$.
\end{proposition}
\noindent Note we usually write $\J$ to mean the density
$\bg^{ab}\V_{ab}$, as calculated in the scale $g$. So here
$\J^\si=\si^2 \J$. In a nutshell the conformal meaning of scalar
curvature is that it is the length squared of the scale tractor.

Next we come to the main reason the tractor metric is important,
namely because it is preserved by the tractor connection. With $V$ as in
Proposition \ref{metp} we have
\begin{equation}\nonumber
\begin{split}
\nabla_a h(V,V)\overset{g}{=}& 
2[\rho\nabla_a\sigma+\sigma\nabla_a\rho+\bg^{bc}\mu_b\nabla_a\mu_c]\\
=&2[\rho(\nabla_a\sigma-\mu_a)+\sigma(\nabla_a\rho-P_{ab}\mu^b)+
\bg^{bc}\mu_b(\nabla_a\mu_c+g_{ac}\rho+P_{ac}\sigma)]\\
=&2h(V,\nabla_a^{\mathcal{T}}V).
\end{split}
\end{equation}
We summarise this with also the previous result.
\begin{theorem}\label{mpara}
On a conformal manifold $(M,\cc)$ of signature $(p,q)$, the tractor
bundle carries a canonical conformally invariant metric $h$ of
signature $(p+1,q+1)$. This is preserved by the tractor connection.
\end{theorem}

We see at this point that the tractor calculus is beginning to look
like an analogue for conformal geometry of the Ricci calculus of
(pseudo-)Riemannian geometry: A metric on a manifold canonically
determines a unique Levi-Civita connection on the tangent bundle
preserving the metric. The analogue here is that a conformal structure
of any signature (and dimension at least 3) determines canonically the
standard tractor bundle $\cT$ equipped with the connection
$\nabla^\cT$ and a metric $h$ preserved by $\nabla^\cT$.

Note also that the tractor bundle, metric, and connection seem to be
analogues of the corresponding structures found for model in Theorem
\ref{mtr}. Especially in view of matching filtration structures:
\nn{Tcomp} should be compared with that found on the model in
Proposition \ref{mTcomp} (noting that $TS^n[-1]\cong T^*S^n[1]$). In
fact the tractor connection here of Theorem \ref{tthm}, and the
related structures, generalise the corresponding objects on the
model. This follows by more general results in \cite{CapGoamb}, or
alternatively may be verified directly by computing the formula for
the connection of Theorem \ref{mtr} in terms of the Levi-Civita
connection on the round sphere.

\section{Lecture 4: The tractor curvature, conformal invariants and invariant operators.} 

Let us return briefly to our motivating problems: the construction of
invariants and invariant operators. 

\subsection{Tractor curvature}\label{Tcurv}

Since the tractor connection $\nabla^\cT$ is well defined on a
conformal manifold its curvature $\kappal$  on $\ce^A$ depends only on
the conformal structure;  by construction it is an invariant of
that. 
%Let us write $\nabla$ to denote the coupled tractor-Levi-Civita connection.
If we couple the tractor connection with any torsion free connection (in particular with the Levi-Civita connection of any metric in the conformal calss) then, according to our conventions from Lecture
\ref{L1}, the curvature of the tractor connection is given by
$$
(\nabla_a\nabla_b-\nabla_b\nabla_a) U^C={\kappal}_{ab}{}^C{}_D U^D 
\quad \mbox{for all} \quad U^C \in \Gamma(\ce^C) ,
$$
where $\nabla$ denotes the coupled connection.  Now using that
the tractor connection preserves the inverse metric $h^{CD}$ we have
$$
0=(\nabla_a\nabla_b-\nabla_b\nabla_a)h^{CD}
= {\kappal}_{ab}{}^C{}_Eh^{ED}+{\kappal}_{ab}{}^D{}_Eh^{CE}.
$$
In other words, raising an index with $h^{CD}$ 
\begin{equation}\label{kskew}
{\kappal}_{ab}{}^{CD}=-{\kappal}_{ab}{}^{DC} .
\end{equation}

It is straightforward to compute $\kappal$ in a scale $g$. We have 
\begin{equation}\label{curvform}
(\nabla_a\nabla_b-\nabla_b\nabla_a) \left(\begin{array}{c} \si\\
\mu^c\\ 
\rho
\end{array}\right) =\left(\begin{array}{ccc} 0 & 0 & 0 \\
C_{ab}{}^c & W_{ab}{}^c{}_d & 0 \\ 
0 & -C_{abd} & 0 \\
\end{array}\right) \left(\begin{array}{c} \si\\
\mu^d\\ 
\rho
\end{array}\right)
\end{equation}
where, recall $W$ is the Weyl curvature and $C$ is the Cotton tensor,
 $$
 C_{abc}:= 2\nabla_{[a}\V_{b]c} .
 $$

In the expression for $\kappal$,  the
  zero in the top right follows from the skew symmetry \nn{kskew},
  while the remaining zeros of the right column show that the
  canonical tractor $X^D$ annihilates the curvature,
$$
{\kappal}_{ab}{}^C{}_D X^D =0. 
$$ 
This with the skew symmetry determines the top row of the curvature
matrix. It follows from the conformal transformation properties of the
tractor splittings that the central entry of the matrix is conformally
invariant, and this is consistent with the appearance there of the
Weyl curvature $W_{ab}{}^c{}_d $. Note that in dimension 3 this
necessarily vanishes, and so it
follows that then the tractor curvature is fully
captured by and equivalent to the Cotton curvature $C_{abc}$. Again
this is consistent with the well known conformal invariance of that quantity in dimension
3. Thus 
we have the following result. 
\begin{proposition}\label{curvOb}
The normal conformal tractor connection $\nabla^\cT$ is flat if and
only if the conformal manifold is locally equivalent to the flat model. 
\end{proposition}
\noindent So we shall say a conformal manifold $(M,\cc)$ is conformally flat if $\kappal=0$.

\subsubsection{Application: Conformally invariant obstructions to metric being conformal-to-Einstein}\label{obstr}

\newcommand{\kk}{\mbox{\large{$\kappa$}}}
Note that as an immediate application we can use the tractor curvature
to easily manufacture obstructions to the existence of an Einstein
metric in the conformal class $\cc$. 

From Proposition \ref{Ecase} and Theorem \ref{para} a metric $g$ is Einstein
if and only if the corresponding scale tractor $I_A=\frac{1}{n}D_A \si$ is parallel, where $g=\si^{-2}\bg$.
Thus if $g$ is Einstein then we have 
\begin{equation}\label{compat}
\kk_{ab}{}^C{}_DI^D=0 .
\end{equation}
Recall that $\kk_{ab}{}^C{}_DX^D=0$. If at any point $p$ the kernel
of $\kk_{ab}{}^C{}_D:\ce^D\to \ce_{ab}{}^C$ is exactly 1-dimensional
then we say that tractor curvature has maximal rank. We have:
\begin{proposition}\label{obp} Let $(M,\cc)$ be a conformal manifold.
If at any point $p\in M$ the tractor curvature has maximal rank 
then
there is no Einstein metric in the conformal equivalence class $\cc$.
\end{proposition}
From this observation it is easy to manufacture conformal invariants
that must vanish on an Einstein manifold, see \cite{GN}.
\begin{proof}
If a metric $g$ is Einstein then $\si$ is a true scale and hence
nowhere zero.  Thus $I^D_p$ is not parallel to $X^D_p$, and the
result follows from \nn{compat}. 
\end{proof}

\subsection{Toward tractor calculus} \label{trc}
Although the tractor connection and its curvature are conformally
invariant it is still not evident how to easily manufacture conformal
invariants. The tractor curvature takes values in
$\Lambda^2(T^*M)\otimes \End(\cT)$ and this not a bundle on which
tractor connection acts. 

To deal with this and the related problem of constructing differential
invariants of tensors and densities we need need additional tools.

The simplest of these is the Thomas-D operator. Recall that in
Proposition \ref{D1} (cf.\ \nn{D1form}) we constructed a differential
operator $D:\ce[1]\to \cT$ that was, by construction, tautologically
conformally invariant. This generalises. Let us write $\ce^\Phi[w]$ to
denote any tractor bundle of weight $w$. Then:
\begin{proposition}\label{Def}
There is a conformally invariant differential operator 
$$
D_A: \ce^\Phi[w]\to \ce_A\otimes \ce^\Phi[w-1],
$$
defined in a scale $g$ by
$$
 V \mapsto [D_A V]_g:= \left(\begin{array}{c}
(n+2w-2)w V \\
(n+2w-2) \nabla_a V \\
- (\Delta V + w \J V)
\end{array}\right)
$$
where $\nabla$ denotes the coupled tractor--Levi-Civita connection.
\end{proposition}
\begin{proof} Under a conformal transformation 
$g\mapsto \widehat{g}=\Omega^{2}g$,
$[D_A f]_g$ transforms by \nn{ttrans}.
\end{proof}
\noindent Note that this result is not as trivial as the written proof
suggests since $V$ (with any indices suppressed) is a section of any
tractor bundle, and the operator is second order. In fact there are
nice ways to construct the Thomas-D operator from more elementary
invariant operators \cite{CGTracBund,Go-srni1}.

\subsubsection{Application: differential  invariants of densities and weighted tractors } \label{appSec}
A key point about the Thomas-D operator is that it can be iterated. For $V\in \Gamma(\ce^\Phi[w])$ we may form (suppressing all indices):
$$
V\mapsto (V,~D V,~D\circ D V, ~D\circ D\circ D V, \cdots )
$$ and, for generic weights $w\in \mathbb{R}$, in a conformally
invariant way this encodes the jets of the section $V$ entirely into
weighted tractor bundles. Thus we can proliferate invariants of $V$. 

For example for $f\in \Gamma(\ce[w])$ we can form 
$$
(D^A f)D_Af= -2w(n+2w-2)f(\Delta f + w\J f) +(n+2w-2)^2 (\nabla^a f)\nabla_a f.
$$ By construction this is conformally invariant, for {\em any weight}
$w$. So in fact it is a family of invariants (of densities).
Similarly we may form
$$
(D^A D^B f)D_AB_B f= -2(n+2w-4)(n+2w-2)w(w-1) f \Delta^2 f 
+ \mbox{\it{lower order terms}},
$$
and so forth.

\subsubsection{Conformal Laplacian-type linear operators}\label{lap1}

One might hope that the tractor-D operator is also effective for the
construction of conformally invariant linear differential
operators. In particular the construction of Laplacian type operators
is important. Certainly $D^AD_A$ is by construction conformally invariant but, on any weighted tractor bundle:
$$
D^AD_A =0.
$$ 
This is verified by a straightforward calculation, but it should not
be surprising as by construction it would be invariant on, for
example, densities of any weight. On the standard conformal sphere
it is well-known that there is no non-trivial operator with this property. 

From our earlier work we know the domain bundle of the conformal Laplacian is
$\ce[1-\frac{n}{2}]$. Observe that for $V\in
\Gamma(\ce^\Phi[1-\frac{n}{2}])$ we have $n+2w-2=0$, and
$$
D_Af \stackrel{g}{=}\left(\begin{array}{c}
0 \\
0 \\
- (\Delta  + \frac{2-n}{2}\J )V 
\end{array}
\right), \quad\mbox{that is} \quad D_A V= -X_A\Box V
$$
where $\Box$ is the (tractor-twisted) conformal Laplacian. 
In particular, the proof of Proposition \ref{Def} was also a proof of this result:  
\begin{lemma} The operator 
$\Box\stackrel{g}{=}(\Delta  + \frac{2-n}{2}\J )$ is a conformally invariant differential operator 
$$
\Box : \ce^\Phi[1-\frac{n}{2}]\to\ce^\Phi[-1-\frac{n}{2}], 
$$
where $\ce^\Phi$ is any tractor bundle.
\end{lemma}

This is already quite useful, as the next exercise shows.
\begin{exe}
Show that if $f\in\ce[2-\frac{n}{2}]$ then 
\begin{equation}\label{pan}
\Box D_A f=-X_A P_4 f .
\end{equation}
Thus there is a conformally invariant differential operator 
$$
P_4 :\ce[2-\frac{n}{2}]\to \ce[-2-\frac{n}{2}] \quad\mbox{where}\quad P_4\stackrel{g}{=}\Delta^2+ \mbox{\it{lower order terms}}.  
$$ 
In fact this is the celebrated {\em Paneitz operator}
discovered by Stephen Paneitz in 1983, see \cite{PanSIGMA} for a
reproduction of his preprint from the time. See  \cite{Go-srni1,GoPet-Lap} for further discussion and generalisations.
\end{exe}
An important point is that \nn{pan} does \underline{not} hold if we
replace $f$ with a tractor field of the same weight! Those who 
complete the exercise will observe the first hint of
this subtlety, in that during the calculation derivatives will need to
be commuted.

\subsection{Splitting tractors}\label{splitt} 
Although the importance of the tractor connection stems from its
conformal invariance we need efficient ways to handle the decomposition of the
tractor bundle corresponding to a choice of scale.

Recall that a metric $g\in \cc$ determines an isomorphism 
\begin{equation}\label{isot}
 \begin{array}{c}
\mathcal{E}[1] \\ 
\oplus \\ 
\mathcal{E}_a[1] \\ 
\oplus\\ 
\mathcal{E}[-1]
\end{array}  \stackrel{\simeq}{\longrightarrow} \ce^A, \quad
\mbox{mapping}
\quad 
 \left(\begin{array}{c} \si\\
\mu_a\\ 
\rho
\end{array}\right) =[U^A]_g \mapsto U^A\in \ce^A.
\end{equation} 
The inclusion of $\ce[-1]$ into the direct sum followed by this map is just 
$X^A:\ce[-1]\to \ce^A$, as observed earlier in Proposition \nn{Xnull}.
This is conformally invariant. However let
us now fix the notation
\begin{equation}\label{YZ}
Z^{Aa}: \mathcal{E}_a[1]\to \ce^A, \quad
\mbox{and}
\quad
Y^A: \ce[-1]\to \ce^A ,
\end{equation}
for the other two bundle maps determined by \nn{isot}. We call these (along with $X^A$) 
the {\em tractor projectors} and view them as bundle sections
$Z^{Aa}\in \Gamma(\ce^{Aa})[-1]$, and $Y^A\in \ce^A[-1]$.
So in summary $[U^A]_g=(\si,\mu_a,\rho)$ is equivalent to 
\begin{equation}\label{Uform}
U^A= Y^A\si + Z^{Aa}\mu_a +X^A \rho .
\end{equation}
Using the formula \nn{mform} for the tractor metric it follows at once that 
$X^AY_A=1$, $Z^{Aa}Z_{Ab}=\delta^a_b$ and all other quadratic combinations of the 
$X$, $Y$, and $Z$ are zero as summarised in Figure \ref{tmf}.
\begin{figure}[ht]
$$
\begin{array}{l|ccc}
& Y^A & Z^{Ac} & X^{A}
\\
\hline
Y_{A} & 0 & 0 & 1
\\
Z_{Ab} & 0 & \delta_{b}{}^{c} & 0
\\
X_{A} & 1 & 0 & 0
\end{array}
$$
\caption{Tractor inner product}
\label{tmf}
\end{figure}
Thus we also have
$
Y_AU^A=\rho,\ \ X_AV^A=\si,\ \ Z_{Ab}U^A=\mu_b\, $
and the metric may be decomposed into a sum of projections, $
h_{AB}=Z_A{}^cZ_{Bc}+X_AY_B+Y_AX_B\,$.

The projectors $Y$ and $Z$ depend on the metric $g\in \cc$.  If
$\hat{Y}^A$ and $ \hat{Z}^A{}_b$ are the corresponding quantities in
terms of the metric $ \hat{g}=\Omega^2g$ then altogether we have
\begin{equation}\label{XYZtrans}
\textstyle
\widehat{X}^A=X^A, \hspace{1mm}  \widehat{Z}^{Ab}=Z^{Ab}+\Up^bX^A, \hspace{1mm}
\widehat{Y}^A=Y^A-\Up_bZ^{Ab}-\frac12\Up_b\Up^bX^A
\end{equation}
as follows immediately from \nn{Uform} and \nn{ttrans}.

\begin{remark}\label{adaptedf}
In the notation $Z^A{}_a$ the tractor and tensor indices are both
abstract.  If we move to a concrete frame field for $TM$, $e_1,\cdots
,e_n$ and then write $Z^A{}_i$, $i=1,\cdots, n$, for the contraction
$Z^A{}_a e^a_i$ we come to the (weighted) tractor frame field:
$$
X^A,Z^A{}_1,\cdots ,Z^A{}_n, Y^A .  
$$ 
This is a frame for the tractor bundle adapted to the filtration
\nn{Tcomp}, as reflected in the conformal and inner product properties described in \nn{XYZtrans} and figure \ref{tmf}. \endrk
\end{remark}

Of course tensor products of the tractor bundle are also decomposed by
the isomorphism \nn{isot} and this is described by the tensor products
of the projectors in an obvious way. To illustrate consider the case
of the bundle of {\em tractor $k$-forms} $\Lambda^k \cT$ (which note is
non-zero for $k=0,\cdots ,n+2$.)  The composition series $\cT=\ce[1]\lpl
T^*M[-1]\lpl \ce[-1]$ determines the composition series for $\Lambda^k
\cT$,
\begin{equation} \label{formtractorcomp}
\Lambda^k{\mathcal{T}}\cong \Lambda^{k-1} [k]\lpl \begin{array}{c} \Lambda^{k-2}[k-2] \\ \oplus \\ \Lambda^k [k]\end{array}\lpl \Lambda^{k-1} [k-2] .
\end{equation}
Given a choice of metric $g$ from the conformal class there is a
splitting of this composition series corresponding to the splitting
\nn{isot} of ${\mathcal{T}}$, and this is easily computed and dealt with 
using the ``projectors'' $X$, $Y$, and $Z$, see e.g.\ \cite{BrGo-deRham}.

\subsection{The connection} 
Just as connections are often described in terms of their action on a
suitable frame field it is useful to give the tractor connection in
terms of its action on the projectors $X$, $Y$, and $Z$. The tractor
covariant derivative of a field $U^A\in \Gamma(\ce^A)$, as in
\nn{Uform}, is given by \nn{cform}. Using the isomorphism \nn{isot} this 
is written
$$
\nabla_a U^B=  Y^B(\nabla_a \si-\mu_a) + Z^{Bb}(\nabla_a\mu_b +\bg_{ab} \rho +\V_{ab}\si)+X^B(\nabla_a \rho-\V_{ab}\mu^b) .
$$ 
On the other hand applying the connection directly to $U^A$
expanded as in \nn{Uform} we have 
$$
\nabla_a U^B= Y^B\nabla_a \si + \si \nabla_a Y^B + Z^{Bb}\nabla_a\mu_b +\mu_b\nabla_a Z^{Bb} +X^B\nabla_a \rho + \rho \nabla_a X^B,
$$ 
where we have used the Leibniz rule for $\nabla_a$, viewed as the
coupled tractor-Levi-Civita connection. Comparing these we obtain:
\begin{equation}\label{connids}
\nd_aX^B=Z^B{}_{a}\,, \hspace{1mm}
\nd_aZ^B{}_{b}=-\V_{ab}X^B-\bg_{ab}Y^B\,,\hspace{1mm} \nd_aY^B=\V_{ab}Z^{Bb} .
\end{equation}
This gives the transport equations for the projectors (and determines
these for the adapted frame as in the Remark \ref{adaptedf}
above). From a practical point of view the formulae \nn{connids} here
enable the easy computation of the connection acting on a tractor field of
any valence. 

\subsubsection{Application: Computing, and conformal Laplacian operators} \label{laplcn}
The formulae are effective for reducing  most 
tractor calculations to a routine task. For example suppose we want to compute $D^A X_A f = D^A (X_A f)$ for $f\in \Gamma(\ce^\Phi[w])$, i.e. a section of any weighted tractor bundle. We calculate in some scale $g$. First note that 
\begin{equation}\nonumber
\begin{split}
 X^A (\Delta +(w+1)\J)X_ A f=&\, X^A \Delta X_ A f, \\
=&\, X^A [\Delta, X_ A]  f, \\
=&\,X^A \nabla^b[\nabla_b, X_ A]  f+ X^A [\nabla_b, X_ A]\nabla^b f\\
=&\, X^A (\nabla^b Z_{bA}) f + 2 X^A Z_{bA} \nabla^b f\\
= &\, X^A (-\J X_A -n Y_A) f\\
= &\, -n f,
\end{split}
\end{equation}
where we used that $X^AX_A=0$,  $X^A Z_{bA}=0 $, and $X^AY_A=1$. 
Thus we have 
\begin{equation}\nonumber
\begin{split}
 D^A X_ A f=&\, [(w+1)(n+2w)Y^A + (n+2w)Z^{Aa}\nabla_a -X^A \Delta] X_A f\\
=&\,(w+1)(n+2w) f + (n+2w)Z^{aA}Z_{aA} f + n f\\
=&\, (w+1)(n+2w) f + (n+2w)n f + n f,\\
\end{split}
\end{equation}
and collecting terms we come to 
\begin{equation}\label{DX}
D^A X_A f= (n+2w+2)(n+w) f .
\end{equation}

A straightforward induction using \nn{DX} and \nn{connids} then enables us to show that:
\begin{theorem}\label{lapthm}
On a conformal manifold $(M^n,\cc)$,  for any tractor bundle $\ce^\Phi$, and for each $k\in \mathbb{Z}_{\geq 1}$ if $n$ is odd (or $k\in \mathbb{Z}_{\geq 1}$ and $2k<n$ if $n$ is even) there is a conformally invariant differential operator 
$$
\Box_{2k}: \ce^\Phi[k-\frac{n}{2}]\to\ce^\Phi[-k-\frac{n}{2}] 
$$
of the form $\Delta^k+\mbox{\it{lower order terms}}$ (up to a non-zero constant factor). These are given by
$$
\Box_{2k} := D^{A_1}\cdots D^{A_{k-1}}\Box D_{A_{k-1}}\cdots D_{A_1} .
$$
\end{theorem}

\begin{remark}
The operators in the Theorem were first reported in \cite{Go-srni1},
(and with a different proof) as part of joint work of the second
author with M.G. Eastwood.

Acting on densities of weight $(k-n/2)$ there are the GJMS operators
of \cite{GJMS}.  For $k\geq 3$ these operators of Theorem \ref{lapthm}
differ from the GJMS operators, as follows easily from the discussion
in \cite{GoPet-Lap}. \endrk
\end{remark}

\subsection{Constructing invariants} 

We can now put together the above tools and proliferate curvature invariants. 
As a first step, following \cite{Go-Adv} we may form 
$$
W_{AB}{}^{K}{}_{L}:=\frac{3}{n-2}D^PX_{[P}Z_{A}{}^aZ_{B]}^b \kappal_{ab}{}^{K}{}_{L} .
$$
This is conformally invariant by construction, as it is immediate from \nn{XYZtrans} that 
$X_{[P}Z_{A}{}^aZ_{B]}^b$ is conformally invariant. It is exactly the object $\bbX^3$ which, for example, gives the conformally invariant injection
$$
\bbX^3:\Lambda^2[1]\to \Lambda^3\cT.
$$ 
%where $\Lambda^k[w]$ is a shorthand for $(\Lambda^kT^*M)[w]$. 
It turns out that the $W$-tractor $W_{ABCD}$ has the symmetries of an
algebraic Weyl tensor. 
In fact in a choice of conformal scale, $W_{ABCE}$ is given by
\begin{equation}\label{Wform}
\begin{array}{l}
(n-4)\left( Z_A{}^aZ_B{}^bZ_C{}^cZ_E{}^e W_{abce}
-2 Z_A{}^aZ_B{}^bX_{[C}Z_{E]}{}^e C_{abe}\right. \\ 
\left.-2 X_{[A}Z_{B]}{}^b Z_C{}^cZ_E{}^e C_{ceb} \right)
+ 4 X_{[A}Z_{B]}{}^b X_{[C} Z_{E]}{}^e B_{eb},
\end{array}
\end{equation}
where $C_{abc} $ is the Cotton tensor,
%% %
%% \begin{equation}\label{cot}
%% C_{bca}:=2\nabla_{[b}\V_{c]a} ,
%% \end{equation} 
 and 
\begin{equation}\label{Bachform}
B_{ab}:=\nabla^c
C_{cba}+\V^{dc}W_{dacb},
\end{equation}
see \cite{GoPetObstrn}.
Note that from \nn{Wform} it follows that, in dimension 4, $B_{eb}$ is
conformally invariant. 
This is the {\em Bach tensor}. 

Since $W_{ABCD}$ takes values in a weighted tractor bundle we may
apply the Thomas-D to this to capture jets of structure invariantly 
$$
W_{ABCD} \mapsto (W,~D W,~D\circ D W, ~D\circ D\circ D W, \cdots ) 
$$ following the idea of Section \ref{appSec}.  Contractions of these
terms then yield invariants. On odd dimensional manifolds this idea
with some minor additional input produces a generating set of scalar
conformal invariants \cite{Go-Adv}.  An alternative (but closely
related \cite{CapGoamb}) approach to the construction of conformal
invariants uses the Fefferman-Graham ambient metric
\cite{BEGr,FGbook}. With either approach, in even dimensions the
situation is far more subtle, deeper ideas are needed, and even for
the construction of scalar conformal invariants open problems remain (as
mentioned in the sources referenced).

\section{Lecture 5: Conformal compactification of pseudo-Riemannian 
manifolds}\label{compS}

At this point we change directions from developing the theory of tractor calculus for conformal geometries to applying this calculus to various problems inspired by asymptotic analysis in general relativity. We will see that the conformal treatment of infinity in general relativity fits very nicely with the conformal tractor calculus, and our basic motivation will be to produce results which are useful in this setting. For the most part will work fairly generally however, and much of what is presented will be applicable in other situations involving hypersurfaces or boundaries (such as Cauchy surfaces or various kinds of horizons in general relativity, or to the study of Poincar\'e-Einstein metrics in differential geometry). From this point on we will adopt the convention that $d=n+1$ (rather than $n$ as before) is the dimension of our manifold $M$, so that $n$ will be the dimension of $\partial M$ or of a hypersurface in $M$. 

\subsection{Asymptotic flatness and conformal infinity in general relativity}

It is natural in seeking to understand and describe the physics of general
relativity to want to study isolated systems. In particular we want
to be able to understand the mass and/or energy of the system (as
well as other physical quantities) and how the system radiates gravitational energy (or how the system interacts with gravitational radiation coming in from infinity, i.e. from outside the system). Clearly it is quite unnatural to try to isolate a physical system from a spacetime by simply considering the inside of a timelike tube containing the system
(events inside the tube would depend causally on events outside, and
there would be no natural choice of such a tube anyway). From early
on in the history of general relativity then physicists have sought
to define isolated systems in terms of spacetimes which are \emph{asymptotically flat}, that is which approach the geometry of Minkowski space in a suitable way as you approach ``infinity''.

Definitions of asymptotic flatness typically involve the existence
of special coordinate systems in which the metric components and other
physical fields fall off sufficiently quickly as you approach infinity
(infinity being defined by the coordinate system). Once a definition
of asymptotically flat spacetimes is established one can talk about
the \emph{asymptotic symmetry group} of such spacetimes (which is
not in general the Poincare group) and the corresponding physical
quantities such as mass and energy. One must be careful in defining
the asymptotic flatness condition to not require the fields to fall
off too fast (thereby excluding massive spacetimes, or gravitationally
radiating ones) nor too slow (so that one loses the asymptotic symmetries
needed to define physical quantities). 

A condition on asymptotically flat spacetimes which came to be seen
as important is the \emph{Sachs peeling-off property}, this is the
condition that along a future directed null geodesic which goes out
to infinity with affine parameter $\lambda$ the Weyl curvature $W^{a}{_{bcd}}$
satisfies
\[
W=\frac{W^{(4)}}{\lambda}+\frac{W^{(3)}}{\lambda^{2}}+\frac{W^{(2)}}{\lambda^{3}}+\frac{W^{(1)}}{\lambda^{4}}+O\left(\frac{1}{\lambda^{5}}\right)
\]
where each tensor $W^{(k)}$ is of special algebraic type with the
tangent vector to the null geodesic as a $k$-fold repeated principal
null direction. Typically asymptotically flat spacetimes were required
to be Ricci flat near infinity, so the Weyl curvature was in fact
the full curvature tensor in this region; the term $W^{(4)}$ in the
asymptotic expansion was interpreted as the gravitational radiation
reaching infinity.

The mass associated with this approach to isolated systems is called
the \emph{Bondi mass}, it is not a conserved quantity but satisfies
a \emph{mass loss formula} as energy is radiated away from the spacetime.
The asymptotic symmetry group (which is the same abstract group for
any spacetime) is known as the \emph{BMS} (Bondi-Metzner-Sachs) \emph{group}.
There is an alternative way of defining asymptotic flatness using
the 3+1 formalism. In this case one talks about the spacetime being
\emph{asymptotically flat at spatial infinity}, the corresponding
mass is the \emph{ADM} (Arnowitt-Deser-Misner) \emph{mass}, and the
asymptotic symmetry group is known as the \emph{SPI group}. Questions
of how these two notions of asymptotic flatness are related to each
other can be subtle and tricky. For a good introductory survey of
these issues see the chapter on asymptotic flatness in \cite{Wald}.

After a great deal of work in these areas, a new approach was
suggested by Roger Penrose in the 1960s
\cite{Penrose-ConfInfLetter,Penrose-LightConeInf,Penrose-zrmFieldAsympt}.
Penrose required of an asymptotically flat spacetime that the
conformal structure of spacetime extend to a pair of null
hypersurfaces called future and past null infinity. This mirrored the
the conformal compactification of Minkowski space obtained by adding a
lightcone at infinity. It was quickly shown that (i) Penrose's notion
of conformal infinity satisfied the appropriate uniqueness property \cite{Geroch-LocalChar},
(ii) this form of asymptotic flatness implied the peeling property \cite{Penrose-ConfInfLetter,Penrose-zrmFieldAsympt}, and (iii) the corresponding group of asymptotic symmetries was the
usual BMS group \cite{Geroch-Asympt}. It is not totally surprising that this idea worked
out so well: the causal (or light-cone) structure of spacetime, which is encoded by the conformal structure, had
played an important role in the analysis of gravitational radiation up
to that point. The conformal invariance of the zero rest mass equations for arbitrary
spin particles and their peeling-off properties also fit nicely with Penrose's proposal, providing further motivation. What is really nice about this
approach however is that it is both natural and coordinate free.

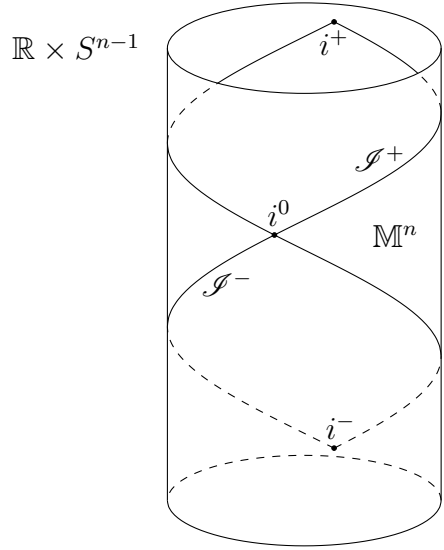
\begin{figure}[ht]
\begin{center}
	\begin{tikzpicture}[xscale=0.6,yscale=0.6]
		%Cylinder
		\draw  (0,3)  ellipse (3 and 1);
		%\draw  (0,-7) ellipse (3 and 1);
		\draw[dashed,xscale=1,yscale=1,domain=0:3.141,smooth,variable=\t] plot ({3*cos(\t r)},{sin(\t r)- 7});
		\draw[xscale=1,yscale=1,domain=3.141:6.282,smooth,variable=\t] plot ({3*cos(\t r)},{sin(\t r)- 7});
		\draw (-3,3) -- (-3,-7);
		\draw (3,3) -- (3,-7);
		\node at (-5,3) {$\mathbb{R}\times S^{n-1}$} ;
		%for symmetry - to center the cylinder:
		\node at (5,3) {$\phantom{\mathbb{R}\times S^{n-1}}$} ;
				
		%Boundary of Minkowski Space
		\node at (2,-1.15) {$\mathbb{M}^n$} ; 
		\draw[xscale=1,yscale=1,domain=-0.69:0.22,smooth,variable=\t] plot ({3*sin(\t r)},{1.5*\t +3.253}); %{sin(\t r)} means sine of \t in radians
		\draw[dashed,xscale=1,yscale=1,domain=-1.57:-0.69,smooth,variable=\t] plot ({3*sin(\t r)},{1.5*\t +3.253});  
		\draw[xscale=1,yscale=1,domain=-4.7:-1.57,smooth,variable=\t] plot ({3*sin(\t r)},{1.5*\t +3.253}); 
		\draw[dashed, xscale=1,yscale=1,domain=-6.062:-4.7,smooth,variable=\t] plot ({3*sin(\t r)},{1.5*\t +3.253}); 
				
		\draw[xscale=1,yscale=1,domain=2.16:2.924,smooth,variable=\t] plot ({3*sin(\t r)},{1.5*\t - 0.8});
		\draw[dashed,xscale=1,yscale=1,domain=1.61:2.16,smooth,variable=\t] plot ({3*sin(\t r)},{1.5*\t - 0.8});
		\draw[xscale=1,yscale=1,domain=-1.57:1.61,smooth,variable=\t] plot ({3*sin(\t r)},{1.5*\t - 0.8});	\draw[dashed,xscale=1,yscale=1,domain=-3.3639:-1.57,smooth,variable=\t] plot ({3*sin(\t r)},{1.5*\t - 0.8});
				
		\draw  [fill] (0.654,3.58) ellipse (0.05 and 0.05);
		\draw  [fill] (0.654,-5.84) ellipse (0.05 and 0.05);
		\draw  [fill] (-0.654,-1.13) ellipse (0.05 and 0.05);
		\node at (1.65,0.55) {$\mathscr{I}^+$} ;
		\node at (-1.7,-2.2) {$\mathscr{I}^-$} ;
		\node at (-0.55,-0.65) {$i^0$} ;
		\node at (0.65,3.12) {$i^+$} ;
		\node at (0.8,-5.36) {$i^-$} ;
	\end{tikzpicture}
	\caption{The standard conformal embedding of $n$-dimensional Minkowski space $\mathbb{M}^n$ into the Einstein cylinder.}
	\label{Fig:ConfMinkowskiEmb}
\end{center}
\end{figure}

Let us now give the formal definition(s), following \cite{Joerg-ConfInf}.

\begin{definition}
A smooth (time- and space-orientable) spacetime $(\M,\g)$ is
called \emph{asymptotically simple} if there exists another smooth
Lorentzian manifold $(M,g)$ such that
\begin{enumerate}[label=(\roman*)]
\item $\M$ is an open submanifold of $M$ with smooth boundary $\partial\M=\scrI$;
\item there exists a smooth scalar field $\Omega$ on $M$, such that
$g=\Omega^{2}\g$ on $\M$, and so that $\Omega=0$, $\mathrm{d}\Omega\neq0$
on $\scrI$;
\item every null geodesic in $\M$ acquires a future and a past endpoint on $\scrI$.
\end{enumerate}
An asymptotically simple spacetime is called \emph{asymptotically
flat} if in addition $\Rico=0$ in a neighbourhood of $\scrI$.
\end{definition}

The Lorentz manifold $(M,g)$ is commonly referred to as the \emph{unphysical
spacetim}e, and $g$ is called the \emph{unphysical metric}. One can
easily see that $\scrI$ must have two connected components
(assuming $\M$ is connected) $\scrI^{+}$ and $\scrI^{-}$ consisting
of the future and past endpoints of null geodesics respectively. It
is also easy to see that $\scrI^{+}$ and $\scrI^{-}$ are booth smooth
null hypersurfaces. Usually (in the $4$-dimensional setting) $\scrI^{+}$ and $\scrI^{-}$ will have topology $\mathbb{S}^{2}\times\mathbb{R}$, however this is not necessarily the case.
\begin{remark}
The third condition is in some cases too strong a requirement, for
instance in a Schwarzschild black hole spacetime there are null geodesics which circle about the singularity forever. To include Schwarzschild and
other such sapcetimes one must talk about \emph{weakly asymptotically
flat} spacetimes (see \cite{Penrose-str}). \endrk
\end{remark}
\begin{remark}
(For those wondering ``whatever happened to the cosmological constant?'')
One can also talk about \emph{asymptotically de Sitter} spacetimes
in which case the boundary hypersurface(s) will be spacelike and one
asks for the spacetime to be Einstein, rather than Ricci flat, in
a neighbourhood of the conformal infinity. Similarly one can talk
about \emph{asymptotically anti-de Sitter} spacetimes, which have
a timelike hypersurface as conformal infinity. \endrk
\end{remark}
\medskip{}
Just how many spacetimes satisfy this definition of asyptotic flatness,
and how do we get our hands on them? These questions lead us to the
discussion of Friedrich's conformal field equations. We will see that the conformal field equations arise very naturally from the tractor picture. However, first we will discuss more generally the mathematics of conformal compactification (inspired by the conformal treatment of infinity above) for pseudo-Riemannian manifolds, and the geometric constraints placed on such ``compactifications''. Again we shall see that this fits nicely within a tractor point of view.

\subsection{Conformal compactification}

A pseudo-Riemannian manifold $(\M,\g)$ is said to {\em conformally compact} if $\M$ may be identified with the interior of a
smooth compact manifold with boundary $M$ and there is a defining
function $r$ for the boundary $\Sigma =\partial \M$ so that
$$
\g=r^{-2}g  \quad \mbox{on }~ \M
$$ where $g$ is a smooth metric on $M$. In calling $r$  a  {\em defining function} for $\Sigma$ we mean that $r$ is a smooth real valued function on $M$ such that $\Sigma$ is exactly the zero locus $\cZ(r)$ of $r$ and furthermore that  $\mathrm{d} r$ is non-zero at every point of $\Sigma$.

Any such $g$ induces a metric $\og$ on $\Sigma$, but the
defining function $r$ by $r'=f\cdot r$ where $f$ is any non-vanishing
function. This changes $\og$ conformally and so canonically the
boundary has a conformal structure determined by $\g$, but no
metric. The metric $\g$ is then complete and $(\Sigma,\occ)$ is
sometimes called the conformal infinity of $\M$. Actually for our
discussion here we are mainly interested in the structure and geometry
of the boundary and the asymptotics of $\M$ near this, so it is not
really important to us that $M$ is compact. An important problem is
how to link the conformal geometry and conformal field theory of
$\Sigma$ to the corresponding pseudo-Riemannian objects on $\M$.

We will see that certainly this kind of ``compactification'' is not
always possible, but it is useful in a number of settings. Clearly
conformal geometry is involved. In a sense it arises in two ways
(that are linked). Most obviously the boundary has a conformal
structure. Secondly the interior was conformally rescaled to obtain
the metric $g$ that extends to the boundary. This suggests a strong
role for conformal geometry -- and it is this that we want to discuss.
 
Toward our subsequent discussion let us first make a first step by
linking this notion of compactification to our conformal tools and notations. Let us fix a choice of $r$ and hence $g$ above. 
If $\tau\in \Gamma (\ce[1])$ is any non-vanishing scale on $M$ then $\si:=r\tau$ is also a section of $\ce[1]$ but now with zero locus
$\cZ(\si)=\Sigma$. This satisfies that $\nabla^{g}\si $ is nowhere
zero along $\Sigma$ and so we say that $\si$ is a {\em defining density} for $\Sigma$. Clearly we can choose $\tau$ so that 
$$
g = \tau^{-2}\bg
$$
where $\bg$ is the conformal metric on $M$. Then 
$$
\g = \si^{-2}\bg .
$$
Thus we may think of a conformally compact manifold, as defined above, as a conformal manifold with boundary $(M,\cc)$ equipped with a section $\si\in \Gamma
(\ce[1])$ that is a defining density for the boundary $\partial M$. 
  
%By considering densities which arise from constant tractors on the flat model spaces we can come up with examples...hyperbolic ball as symmetry reduction example... plus Minkowski space example

The key examples of conformally compactified manifolds fit nicely
within this framework. Indeed, consider the Poincare ball model of
hyperbolic space $\mathbb{H}^d$. This conformal compactification can
be realised by considering the extra structure induced on the
conformal sphere $(S^d,\cc)$ by a choice of constant spacelike tractor
field $I$ on $S^d$ (i.e. a constant vector $I$ in $\mathbb{V}$). This
choice gives a symmetry reduction of the conformal group
$G=\mathrm{SO}_+(\cH)$ to the subgroup $H\cong \mathrm{SO}_+(d,1)$
which stabilises $I$. On each of the two open orbits of $S^d$ under
the action of $H$ there is induced a hyperbolic metric, whereas the
closed orbit (an $n$-sphere) recieves only the conformal structure
induced from $(S^d,\cc)$. The two open orbits correpond to the two
regions where the scale $\sigma$ of $I$ is positive and negative
respectively, and the closed orbit is the zero locus of $\sigma$. If
we choose coordinates $X^A$ for $\mathbb{V}$ as before, then the
$1$-density $\sigma$ corresponds to the homogeneous degree one
polynomial $\tilde{\si}= I_A X^A$ (dualising $I$ using $\cH$); that
$\sigma^{-2}\bg$ gives a hyperbolic metric on each of the open orbits
can be seen by noting that the hyperplanes $I_A X^A=\pm 1$ intersect
the future light cone $\mathcal{N}_+$ in hyperbolic sections, and the
fact that $\cZ(\si)$ recieves only a conformal structure can be seen
from the fact that the hyperplane $I_A X^A=0$ intersects
$\mathcal{N}_+$ in a subcone. Thus we see that the scale corresponding
to a constant spacelike tractor $I$ on $(S^d,\cc)$ gives rise to a
decomposition of the conformal $d$-sphere into two copies of
conformally compactified hyperbolic space glued along their
boundaries.

\begin{figure}[ht]
\begin{center}
	\begin{tikzpicture}[xscale=0.6,yscale=0.6]
		%Cone	
		\draw  (0,3) ellipse (4 and 1);
		\draw (-3.95,2.84) -- (0,-3) -- (3.96,2.86);
		\node at (-2,1.3) {$\mathcal{N}_+$} ;
		
		%Tractor I, subcone and hyperbolic section
		\draw[-latex] (0,-3) --(1.5,-3);
		\node at (1.5,-3.5) {$I$} ;
		\draw [dotted] (0,-3)  -- (0.99,3.95);
		\draw [dotted] (0,-3) -- (-0.99,2.03);
		\draw [dotted] plot[variable=\t,samples=1000,domain=-41.1:49.11] ({tan(\t)+1.33},{8.9*(sec(\t)-1)-0.908});
		\node at (1,4.35) {$\scriptstyle{\tilde{\sigma}=0}$} ;
		\node at (2.5,4.15) {$\scriptstyle{\tilde{\sigma}=1}$} ;
		\node at (3.5,-1.5) {$\tilde{\sigma}= I_A X^A$} ;
		%\draw[-latex] (1.5,-0) --(3,0);
		
		%Arrow
		\draw[-latex, thick] (5,0)--(8.5,0);
		\node at (6.7,0.5) {$\mathbb{P}_+$} ;
		
		%Sphere
		\draw  (13,0) circle (3);
		\draw[dashed,xscale=1,yscale=1,domain=0:3.141,smooth,variable=\t] plot ({0.6*sin(\t r)+13},{3*cos(\t r)});
		\draw[xscale=1,yscale=1,domain=3.141:6.283,smooth,variable=\t] plot ({0.6*sin(\t r)+13},{3*cos(\t r)});
		\node at (14.75,0) {$\mathbb{H}^d$} ;
		\draw[thin] [-stealth] (13,3.7) -- (13,3.06);
		\node at (13.37,4.1) {$S^{d-1}$};
	\end{tikzpicture}
	\label{Fig:HyperbolicSpaceViaHolonomy}
	\caption{The orbit decomposition of the conformal sphere corresponding to the subgroup $H$ of the conformal group $G$ preserving a fixed spacelike vector $I$. An open orbit may be thought of as $\mathbb{H}^d$ conformally embedded into $\mathbb{S}^d$.}
\end{center}
\end{figure}
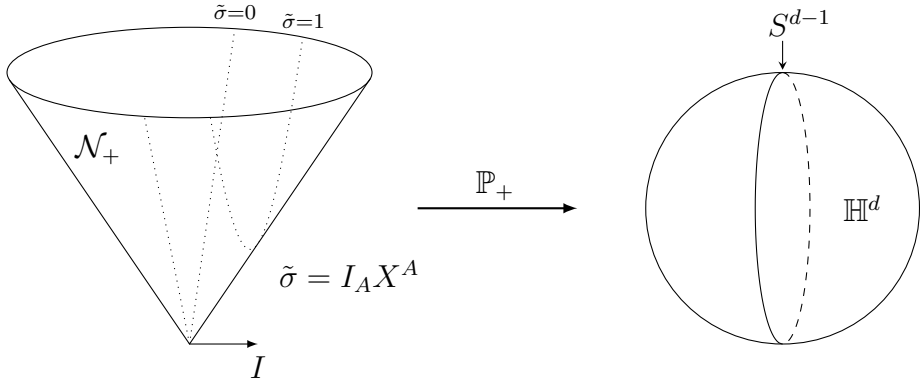

One can repeat the construction above in the case of the Lorentzian signature model space $S^n\times S^1$ in which case one obtains a decomposition of $S^n\times S^1$ into two copies of comformally compactified $AdS^d/\mathbb{Z}$ glued together along their conformal infinities (where the $\mathbb{Z}$-action on anti-de Sitter space is given by integer times $2\pi$ time translations); this conformal compactification of $AdS^d/\mathbb{Z}$ corresponds to the usual embedding of $AdS^d$ into (half of) the Einstein universe after we take the quotient by $\mathbb{Z}$. If one instead takes a constant null tractor $I$ on $S^n\times S^1$ then the corresponding scale $\sigma$ gives decomposition of $S^n\times S^1$ into two copies of conformally compactified Minkowski space glued along their null boundaries (plus two isolated points where $\cZ(\si)$ has a double cone-like singularity).

\subsection{Geometry of scale}\label{scale}

There is an interesting perspective which emerges from the above correspondence between constant (parallel) tractors and conformal compactifications. This may be motivated by Proposition \ref{tp} and Theorem \ref{para} which linked almost Einstein scales to parallel tractors. Together they imply:
\begin{proposition}
An Einstein manifold $(M,g)$ is the same as a conformal manifold
$(M,\cc)$ equipped with a parallel standard tractor $I$ such that
$\si:=X^AI_A$ is nowhere zero: Given $\cc$ and $I$ the metric is recovered by 
\begin{equation}\label{rel}
g=\si^{-2}\bg.
\end{equation}
Conversely a metric $g$ determines $\cc:=[g]$, and $\si\in \ce_+[1]$
by \nn{rel} again, now used as an expression for this variable. Then
$I=\frac{1}{n}D\si$.
\end{proposition}

So this suggests that if we wish to draw on conformal geometry then
using the package $(M,\cc,I)$ may give perspectives not easily seen
via the $(M,g)$ framework. Some questions arise:
\begin{itemize}[noitemsep, topsep=0pt]
\item[(1)] First the restriction that $\si:=X^AI_A$ is nowhere zero
  seems rather unnatural from this point of view. So what happens if
  we drop that? Then, if $I$ is a parallel standard tractor, recall we
  say that $(M,\cc,I)$ is an almost Einstein structure. What, for
  example, does the zero locus $\cZ(\si)$ of $\si=X^AI_A$ look like in
  this case?
\item[(2)] Is there a sensible way to drop the Einstein condition and use this approach on general pseudo-Riemannian manifolds?
\end{itemize}

\subsubsection{The zero locus of almost Einstein and ASC manifolds} \label{HolonomyRed}
In fact there is now known a way to answer the first question via a
very general theory.  Using ``the package $(M,\cc,I)$'' amounts to
recovering the underlying pseudo-Riemannian structure (and its
generalisations as below) as a type of {\em structure group reduction}
of a conformal Cartan geometry. If $I$ is parallel then this a {\em
  holonomy reduction}.  On an extension of the conformal Cartan bundle
to a principal bundle with fibre group $G=\mathrm{O}(p+1,q+1)$, the
parallel tractor $I$ gives a bundle reduction to a principal bundle
with fibre $H$ where this is: (i) $\mathrm{O}(p,q+1)$, if $I$ is
spacelike; (ii) $\mathrm{O}(p+1,q)$, if $I$ is timelike; and (iii) a
pseudo-Euclidean group $\mathbb{R}^d \rtimes O(p,q)$, if $I$ is null.

By the general theory of Cartan holonomy reductions \cite{CGH} any
such reduction yields a canonical stratification of the underlying
manifold into a disjoint union of {\em curved orbits}. These are
parametrised by $H\backslash G/P$, as are the orbits of $H$ on
$G/P$. Each curved orbit is an initial submanifold carrying a
canonically induced Cartan geometry of the same type as that of the
corresponding orbit in the model. Furthermore this curved orbit
decomposition must look locally like the decomposition of the model
$G/P$ into $H$-orbits. This means that in an open neighbourhood $U$ of
any point on $M$ there is a diffeomorphism from $U$ to the an open set
$U_\circ$ in the model that maps each curved orbit (intersected with
$U$) diffeomorphically to the corresponding $H$-orbit (intersected
with $U_\circ$) of $G/P$.

In our current setting $P$ is the stabiliser in $G$ of a null ray in $\mathbb{R}^{p+1,q+1}$ and the model $G/P$ is isomorphic to conformal $S^d\times \{1,-1\}$, if $\cc$ is Riemannian, and conformal $S^p\times S^q$ otherwise. The curved orbits arise here because as we move
around the manifold the algebraic relationship between the parallel
object $I$ and the canonical tractor $X$ changes. (For general
holonomy reductions of Cartan geometries the situation is a simple
generalisation of this.)  In particular, in Riemannian signature (or if $I^2\neq 0$), it is easily verified using these tools that the curved orbits (and the $H$-orbits on the model) are distinguished by the strict sign of $\si=X^AI_A$, see \cite[Section 3.5]{CGH}. By examining these sets on the model we conclude.
\begin{theorem}\label{AE-dec} The curved orbit decomposition of an 
almost Einstein manifold $(M,\cc,I)$ is according to the strict sign
of $\si=I_AX^A$.  The zero locus satisfies: 
\begin{itemize}[itemsep=2pt]
\item If $I^2\neq 0$ (i.e.\ $g^o$ Einstein and not Ricci flat) then $\cZ(\si)$ is
either empty or is a smooth embedded hypersurface.
\item If $I^2=0$ (i.e.\ $g^o$ Ricci flat) then $\cZ(\si)$ is
either empty or, after excluding isolated points from $\cZ(\si)$, is a smooth embedded hypersurface.
\end{itemize} 
\end{theorem}
\noindent Here $g^o$ means the metric $\si^{-2}\bg$ on the open orbits (where $\si$ is nowhere zero).

\begin{figure}[ht]
\begin{center}
	\begin{tikzpicture}[xscale=0.6,yscale=0.6]
			%Manifold
		\draw  plot[smooth cycle, tension=.7] coordinates {(0,0) (-0.6144,-0.0792) (-1.2476,-0.5479) (-1.8886,-1.5127) (-3.0696,-2.2064) (-4.2909,-2.1141) (-5.6555,-0.8581) (-6.2517,1.4852) (-6.0128,3.5791) (-4.7755,5.1173) (-3.1071,4.7968) (-1.9073,3.8474) (-0.97,2.5164) (0.0798,1.6166) (1.0312,1.742) (1.8606,3.0771) (3.0229,4.4081) (4.9726,5.0669) (7.0534,4.5989) (8.2532,2.9867) (8.1407,1.2368) (6.716,-0.6191) (6.0071,-1.9692) (4.8683,-2.4613) (3.4763,-2.2926) (2.3234,-1.688) (1.6545,-0.9879) (1.2795,-0.4817) (0.6375,-0.1222)};
		\draw  plot[smooth, tension=.7] coordinates {(0,0) (-0.2179,0.0547) (-0.3697,0.2797)  (-0.4259,0.5946)  (-0.3585,0.9714)  (-0.1954,1.2807)  (-0.021,1.4663)  (0.1589,1.5507)  (0.3164,1.545)};
		\draw[dashed]  plot[smooth, tension=.7] coordinates {(0,0) (0.2714,0.066) (0.502,0.2853)  (0.6201,0.6115)  (0.6314,0.9433)  (0.5976,1.1739)  (0.547,1.3145)  (0.4514,1.4494)  (0.3164,1.545)};
		\node at (4.9948,2.5018) {$M_+$} ;
		\node at (-3.9472,3.3876) {$M_-$} ;
		\node at (0.4401,2.3999) {$M_0$} ;
		\draw[thin] [-stealth] (0.3783,2.0906) -- (0.2827,1.6238);
		
		%sigma
		\node at (4.3621,-2.9393) {$\sigma > 0$} ;
		\node at (-3.6379,-2.8128) {$\sigma < 0$} ;
		\node at (-0.026,-0.54) {$\sigma=0$} ;
		
		%Holes
		\draw  plot[smooth, tension=.7] coordinates {(-4.402,1.86) (-4.0456,1.3489) (-3.3567,1.2927) (-2.5553,1.7848)};
		\draw  plot[smooth, tension=.7] coordinates {(-4.1996,1.4888) (-3.834,1.7138) (-3.007,1.4719)};
		\draw  plot[smooth, tension=.7] coordinates {(3.1811,-0.3383) (3.8278,-1.0834) (4.798,-1.1959) (5.3182,-0.4367)};
		\draw  plot[smooth, tension=.7] coordinates {(3.4939,-0.7889) (3.8278,-0.3242) (4.5589,-0.3664) (4.9807,-1.0412)};
		\end{tikzpicture}
		\label{ConfCurvedOrbits}
		\caption{The curved orbit decomposition of an almost Einstein manifold with $\cZ(\si)$ an embedded separating hypersurface.}
\end{center}
\end{figure}
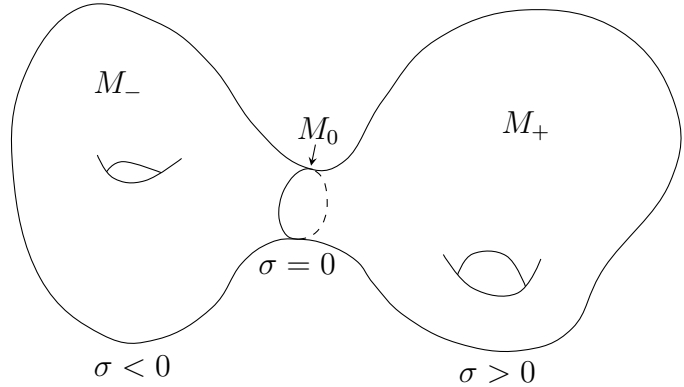

\begin{remark}
Much more can be said using the tools mentioned. For example in the case of  Riemannian signature it is easily shown that:
\begin{itemize}[itemsep=2pt, topsep=0pt]
\item If $I^2<0$ (i.e.\ $g^o$ Einstein with positive scalar curvature) then 
$\cZ(\si)$ is empty.
\item If $I^2=0$ (i.e.\ $g^o$ Ricci flat) then $\cZ(\si)$ is
either empty or consists of isolated points.
\item If $I^2>0$ (i.e.\ $g^o$ Einstein with negative scalar curvature) then $\cZ(\si)$ is either empty or is a smooth embedded separating hypersurface.
\end{itemize}
This holds because this is how things are locally on the flat model \cite{Gal,CGH}; if $I^2<0$ then on the model we have a round metric induced on the whole conformal sphere, the one open orbit is the whole space; if $I^2=0$ then on the model we are looking at the one point conformal compactification of Euclidean space given by inverse stereographic projection, and there are two orbits, one open and one an isolated point, so in the curved case $\cZ(\si)$ is either empty or consists of isolated points; if $I^2>0$ then on the model we are looking at two copies of conformally compactified hyperbolic space glued along their boundaries as discussed earlier, in this case the closed orbit is a separating hypersurface so in the curved case $\cZ(\si)$ is either empty or is a smooth embedded separating hypersurface. In the Lorentzian setting one can obtain a similar improvement of Theorem \ref{AE-dec} by considering the model cases. \endrk
\end{remark}

In fact similar results hold in greater generality (related to the question (2) above) and this is easily seen using the earlier tractor calculus and elementary considerations. We learn from the Einstein case above that an important role is played by the scale tractor
$$
I_A=\frac{1}{d}D_A \si .
$$
In the Einstein case this is parallel and hence non-zero everywhere. Let us drop the condition that $I$ is parallel and for convenience say that a structure 
$$
(M^d,\cc,\si) \quad \mbox{where} \quad \si\in \Gamma(\ce[1])
$$ is {\em almost pseudo-Riemannian} if the scale tractor
$I_A:=\frac{1}{d}D_A\si$ is nowhere zero. Note then that $\si$ is
non-zero on an open dense set, since $D_A\si$ encodes part of the
2-jet of $\si$. So on an almost pseudo-Riemannian manifold there is
the pseudo-Riemannian metric $g^o=\si^{-2}\bg$ on the same open dense
set. In the following the notation $I$ will always refer to
a scale tractor, so $I=\frac{1}{d}D\si$, for some $\si\in
\Gamma(\ce[1])$. Then we often mention $I$ instead of $\si$ and refer to $(M,\cc,I)$ as an almost pseudo-Riemannian manifold.

Now recall from \nn{I2} that 
\begin{equation}\label{I22}
I^2 \stackrel{g}{=} \bg^{ab}(\nabla_a \si)(\nabla_b \si) -
\frac{2}{d} \si (\J+\Delta )\si
\end{equation}
where $g$ is any metric from $\cc$ and $\nabla$ its Levi-Civita connection. 
This is well-defined everywhere on an almost pseudo-Riemannian manifold, while
according to Proposition \ref{scalarcurv}, where $\si$ is non-zero, it computes 
$$
I^2= -\frac{2}{d}\J^{g^o}= -\frac{\Sc^{g^o}}{d(d-1)} \quad \mbox{where} \quad g^o=\si^{-2}\bg .
$$ Thus $I^2$ gives a generalisation of the scalar curvature (up to a
constant factor $-1/d(d-1)$); it is canonical and smoothly extends the
scalar curvature to include the zero set of $\si$. We shall use the
term {\em ASC manifold} (where ASC means almost scalar constant) to
mean an almost pseudo-Riemannian manifold with
$I^2=\mbox{\it{constant}}$. Since the tractor connection preserves the
tractor metric, an almost Einstein manifold is a special case, just as
Einstein manifolds have constant scalar curvature.

Much of the previous theorem still holds in the almost pseudo-Riemannian setting when $I^2\neq 0$:
\begin{theorem}\label{APRthm} Let $(M,\cc,I)$ be an almost pseudo-Riemannian manifold with $I^2\neq 0$. Then $\cZ(\si)$, if not empty, is a smooth embedded separating hypersurface. This has a spacelike (resp.\ timelike) normal if $g$ has negative scalar (resp.\  positive) scalar curvature.

If $\cc$ has Riemannian signature and $I^2< 0$ then $\cZ(\si)$ is empty.
\end{theorem}
\begin{proof}
This is an immediate consequence of \nn{I22}: Along the zero locus $\cZ(\si)$ of $\si$ (assuming it is non empty) we have 
$$
I^2=\bg^{ab}(\nabla_a \si)(\nabla_b\si) .
$$ 
in particular $\nabla \si$ is nowhere zero on $\cZ(\si)$, and so $\si$ is a defining density. Thus $\cZ(\si)$ is a smoothly embedded hypersurface
by the implicit function theorem. Evidently $\cZ(\si)$ separates $M$ according to the sign of $\si$. Also $\nabla \si$ is a (weight 1) conormal field along $\cZ(\si)$ so the claimed signs for the normal also follow from the display. 

Finally if $\cc$ (and hence $\bg$) has Riemannian signature, then the
display shows that at any point of $Z(\si)$ the constant $I^2$ must be
positive. This is a contradiction if $I^2<0$ and so
$Z(\si)=\emptyset$.
\end{proof}

\begin{remark}\label{moregen}
Note that if $M$ is compact then Theorem \ref{APRthm} gives a decomposition of $M$ into conformally compact manifolds glued along their conformal infinities. Note also that if $M$ is allowed to have boundary then we only mean that $\cZ(\si)$ is separating if it is not a boundary component of $M$. \endrk
\end{remark}
 
What we can conclude from all this is that almost pseudo-Riemannian manifolds $(M,\cc,I)$ naturally give rise to nicely conformally compactified metrics (at least in the case where $I^2\neq0$). If the scalar curvature of the manifold $(\M,\g)$ admitting a conformal compactification $(M,c)$ is bounded away from zero then the conformal compactification must arise from an almost pseudo-Riemannian manifold $(M,\cc,I)$ in this way because $I^2$ is then bounded away from zero on $\M$ and hence $I$ extends to be nowhere zero on $M$ (all that we really need is that $I^2\neq 0$ on $\partial \M$). Thus we can apply the results above to the conformal compactification of $(\M,\g)$. This tells us for instance that Riemannian manifolds with negative scalar curvature bounded away from zero must have a nice smooth boundary as conformal infinity if they can be conformally compactified. In the next section we develop this kind of idea.

\subsection{Constraints on possible conformal compactifications} \label{cinf}

Here we show that elementary geometric considerations restrict the topological (and geometric) possibilities for a conformal infinity.

\begin{theorem}\label{newcpt} 
Let $(\M,\g)$ be a geodesically complete pseudo-Riemannian manifold and $i:\M\to M$ an embedding of $\M$ as an open submanifold in a closed conformal manifold $(M,\cc)$ and $\si$ a smooth section of $\ce[1]$ on $M$ so that on the image $i(\M)$, which we identify with $\M$, we have that
$$
\g = \si^{-2}\bg
$$ 
where $\bg$ is the conformal metric on $M$. If the scalar curvature
of $\g$ is bounded away from zero then either $\M=M$ or  the boundary points of $\M$ in $M$ form a smooth embedded hypersurface in $M$. This has a timelike normal field if $\Sc^{\g}>0$ and a spacelike normal field if $\Sc^{\g} <0$. 
\end{theorem}
\begin{proof}
We can form the scale tractor 
$$
I_A=\frac{1}{d}D_A \si. 
$$  
Then $I^2$ is a smooth function on $M$ that is bounded away from zero on the open set $\M \subset M$. Thus $I^2 \neq 0$ on the topological closure $\overline{M}_+$ of $\M$ in $M$.  Now since $I^2 \neq 0$ on $\overline{M}_+$ we have at any point $p\in \cZ(\si) \cap \overline{M}_+$ that $\nabla \si (p)\neq 0$. Thus the zero locus of $\si$ (intersected with $\overline{M}_+$) is a regular embedded submanifold of $\overline{M}_+$. By construction this lies in the set $\partial \M$ of boundary points to $\M$. 

On the other hand any boundary point is in the zero locus, as otherwise
for any $p\in \partial \M$ s.t. $\si(p)\neq 0$ there is an open
neighbourhood of $p$ in $\overline{M}_+$ such that $\si$ is nowhere zero
and so $g$ extends as a metric to this neighbourhood. So then there is a
geodesic from a point in $\M$ that reaches $p$ in finite time, which
contradicts that $\M$ is geodesically complete.
\end{proof}

In particular if $\g$ is Einstein with non-zero cosmological constant
then $I$ is parallel, and hence $I^2$ is constant and nonzero. So the above theorem applies in this case. In fact in the Einstein case we can also get results which in some aspects are stronger. For example:

\begin{theorem} \label{cify} Suppose that
$\M$ is an open dense submanifold in a compact connected conformal manifold $(M^d,\cc)$, possibly with boundary, and that one of the following two possibilities hold: either $M$ is a manifold with boundary $\partial M$ and $M\setminus \M =\partial M$, or $M$ is closed and $M\setminus \M$ is contained in a smoothly embedded submanifold of $M$ of codimension at least 2. Suppose also that $\g$ is a geodesically complete Einstein, but not Ricci flat, pseudo-Riemannian metric on $\M$ such that on $\M$ the conformal structure $[\g]$ coincides with the restriction of $\cc$.
 Then, either
    \begin{itemize}
        \item $\M$ is closed and $\M = M$; or
        \item we are in the first setting with
          $M\setminus \M$ the smooth $n$-dimensional
          boundary for $M$.  
    \end{itemize}
\end{theorem}
\begin{proof}
$\M$ is  canonically equipped with  a parallel standard tractor $I$
such that $I^2=c\neq 0$, where $c$ is constant.

Now $[\g]$ on $\M$ is the restriction of a smooth conformal structure
$\cc$ on $M$. Thus the conformal tractor connection on $\M$ is the
restriction of the smooth tractor connection on $M$. Working locally
it is straightforward to use parallel transport along a congruence of
curves to give a smooth extension of $I$ to a sufficiently small open
neighborhood of any point in $M\setminus \M$, and since $\M$ is dense
in $M$ the extension is parallel and unique. (See \cite{DiScalaManno} for
this extension result in general for the case where $M\setminus \M$ lies in
a submanifold of dimension at least 2.)  It follows that $I$
extends as a parallel field to all of $M$.

It follows that $\si:=I_A X^A$ also extends smoothly to
all of $M$. This puts us back in the setting of Theorem \ref{newcpt}, so that $M\setminus \M$ is either empty or is an embedded hypersurface. If $M\setminus \M$ is empty then $M$ is closed and $M=N$. If $M\setminus \M$ is an embedded hypersurface then, buy our assumptions, we must be in the case where $M$ is a manifold with boundary and $M\setminus \M=\partial M$ (since an embedded hypersurface cannot be contained in a submanifold of codimension two).
\end{proof}

\begin{remark}
In Theorem \ref{newcpt} we had to assume that there was a smooth globally defined conformal 1 density $\sigma$ which glued the metric on $\M$ smoothly to the conformal structure of $M$. In Theorem \ref{cify} we are able to drop this condition because we are requiring that the metric $\g$ on $\M$ is Einstein with conformal structure agreeing with that on $M$, allowing us to recover the smooth global 1 density $\sigma$ by parallel extension of the scale tractor for $\g$ to $M\setminus \M$. However, it is not hard to see that one must place some restriction on the submanifold $M\setminus \M$ in order for a parallel extension to exist, this is where the condition that $M\setminus \M$ must lie in a submanifold of codimension two comes in for the case where $M$ is closed. \endrk
\end{remark}

\subsection{Friedrich's conformal field equations and tractors}

In this section we briefly discuss Friedrich's conformal field equations and their applications. We will see that the equations can be very easily arrived at using tractor calculus, and that the tractor point of view is also nicely compatible with the way the equations are used in applications.

\subsubsection{The equations derived}

Suppose that $(\M,\g)$ is an asymptotically flat spacetime with corresponding unphysical spacetime $(M,g)$ and conformal factor $\Omega$. (We may assume that $M=\M\cup\scrI$.) Then since $(\M,\g)$ is Ricci
flat near the conformal infinity $\scrI$ the conformal scale tractor
$I$ defined on $\M$ corresponding to the metric $\g$ is parallel
(and null) near $\scrI$ and thus has a natural extension (via parallel
transport, or simply by taking the limit) to all of $M$. Let us assume
for simplicity that $(\M,\g)$ is globally Ricci flat, then the naturally
extended tractor $I$ is globally parallel on the non-physical spacetime
$(M,g)$. If we write out the equation $\nabla_{a}I^{B}=0$ on $M$
in slots using the decomposition of the standard tractor bundle of
$(M,[g])$ induced by $g$ then we get the system of equations

\begin{eqnarray*}
\nabla_{a}\sigma & = & \mu_{a}\\
\nabla_{a}\mu_{b} & = & -\sigma \V_{ab}-\rho\boldsymbol{g}_{ab}\\
\nabla_{a}\rho & = & \V_{ab}\mu^{b}
\end{eqnarray*}
where $I\overset{g}{=}(\sigma,\mu_{a},\rho)$. If we also trivialise
the conformal density bundles using $g$ then $\sigma$ becomes $\Omega$,
$\boldsymbol{g}$ becomes $g$, and we recognise the above three equations
as the first three equations in what are commonly known as \emph{Friedrich's
conformal field equations}. A fourth member of the conformal field
equations can be obtained by writing $I_{A}I^{A}=0$ out as
\[
2\sigma\rho+\mu_{a}\mu^{a}=0.
\]
Next we observe that $\nabla_{a}I^{B}=0$ clearly implies $\kappal_{ab}{^{C}}_{D}I^{D}=0$,
and the projecting part of this equation gives the \emph{conformal
C-space equation}
\[
\nabla_{b} \V_{ac}-\nabla_{a}\V_{bc}=W_{abcd}\sigma^{-1}\mu^{d}
\]
after multiplication by $\sigma^{-1}$. (The ``bottom slot'' of
$\kappal_{ab}{^{C}}_{D}I^{D}=0$ taken w.r.t. $g$ is $(\nabla_{b}\V_{ac}-\nabla_{a}\V_{bc})\mu^{c}=0$
which also follows from contracting the above displayed equation with
$\mu^{c}$.) The contracted Bianchi identity states that 
\[
\nabla^{d}W_{abcd}=\nabla_{b}\V_{ac}-\nabla_{a}\V_{bc}
\]
and from this and the previous display it follows that
\[
\nabla^{d}(\sigma^{-1}W_{abcd})=0
\]
(where we have used that $\nabla_{a}\sigma=\mu_{a}$). It can be shown
\cite{Joerg-ConfInf} that the field $\sigma^{-1}W_{ab}{^c}_d$ is regular
at $\scrI$ (recall the Sachs peeling property), and we will write this field as $K_{ab}{^c}_d$. If we substitute $K_{abcd}=\bg_{ce}K_{ab}{^e}_d$ for $\sigma^{-1}W_{abcd}$ in the C-space equation and in the last equation displayed above we get the remaining two equations from Friedrich's conformal system
\[
\nabla_{b}\V_{ac}-\nabla_{a}\V_{bc}=K_{abcd}\mu^{d}
\]
and
\[
\nabla^{a}K_{abcd}=0,
\]
which hold not only on $\M$ but on all of $M$. This pair of equations encode not only the fact that $(\M,\g)$ is Cotton flat (both equations being in some sense conformal C-space equations), but also by their difference they encode the contracted Bianchi identity (and hence the full Bianchi identity if $\M$ is $4$-dimensional \cite{Friedrich-ConfEinsteinEv}).

The equations from Friedrich's conformal system are thus seen to be
elementary consequences of the system of four tractor equations
\begin{eqnarray*} %\label{TractConfFieldEqns}
\nabla_{a}I^{B} & = & 0\\
I^{A}I_{A} & = & 0\\
\kappal_{ab}{^{C}}_{D}I^{D} & = & 0\\
\nabla_{[a}\kappal_{bc]DE} & = & 0.
\end{eqnarray*}

\begin{remark}
There are in fact more equations that could be considered as a part
of Friedrich's conformal system, however these simply define the connection
and the curvature terms used in the equations we have given. These
equations would be the expressions $\nabla g=0$ and $T^{\nabla}=0$
satisfied by the Levi-Civita connection of $g$, and the decomposition
of the Riemannian curvature tensor of $\nabla$ as
\[
R_{abcd}=\sigma K_{abcd}+2\V_{a[c}\boldsymbol{g}_{d]b}-2\V_{b[c}\boldsymbol{g}_{d]a}
\]
which serves to define the curvature tensors $K_{abcd}$ and $\V_{ab}$
(along with the appropriate symmetry conditions on $K_{abcd}$ and
$\V_{ab}$ as well as the condition that $\boldsymbol{g}^{ac}K_{abcd}=0$). 

Note that we can easily allow for a nonzero cosmological constant in our tractor system by taking $I^{A}I_{A}$ to be a constant rather than simply zero. \endrk
\end{remark}

\subsubsection{The purpose of the equations}
What are the conformal Field equations for? The answer is that one seeks solutions to them in the same way that one seeks solutions
to Einstein's field equations. A global solution to the conformal
field equations gives a conformally compactified solution of Einstein's
field equations. (In fact solutions of the conformal field equations
may extend to regions ``beyond infinity'' where $\Omega$ becomes
negative.) We should note one can very easily allow for a non zero
cosmological constant and even for non zero matter fields (especially
ones with nice conformal behaviour) in the conformal field equations.
The conformal field equations are then a tool for obtaining and investigating
isolated systems in general relativity.

Like Einstein's field equations, the conformal field equations have
an initial value formulation where initial data is specified on a
Cauchy hypersurface. But when working with the conformal field equations is that it is also natural to prescribe data on the conformal infinity; when initial data is prescribed on (part of) $\scrI^{-}$ as well as on an ingoing null hypersurface which meets $\scrI^{-}$ transversally we have the
\emph{characteristic initial value problem}. When initial data is
specified on a spacelike hypersurface which meets conformal infinity
transversally we have the \emph{hyperboloidal initial value problem}.
(The name comes from the fact that spacelike hyperboloid in Minkowski
space are the prime examples of such initial data hypersurfaces.)
If the hyperboloidal initial data hypersurface meets $\scrI^{-}$
(rather than $\scrI^{+}$) then one should also prescribe data on
the part of $\scrI^{-}$ to the future of the hypersurface, giving
rise to an initial-boundary value problem.

What is being sought in the study of these various geometric PDE problems?
Firstly information about when spacetimes will admit a conformal infinity
and what kind of smoothness it might have. Secondly information about
gravitational radiation produced by various gravitational systems
as well as the way that such systems interact with gravitational radiation
(scattering properties). Thirdly, one is obviously interested in the
end in having a general understanding of the solutions of the conformal
field equations (though this is a very hard problem). These problems
have been studied a good deal, both analytically and
numerically, however there are many questions left to be answered.
For helpful overviews of this work consult \cite{Joerg-ConfInf,Friedrich-ConfEinsteinEv}.

\subsubsection{Regularity}
An important feature of the conformal field equations is that they
are \emph{regular} at infinity. From the tractor calculus point of
view this is obvious since the conformal structure extends to the
conformal infinity so the tractor system displayed above cannot break
down there, but one can also easily see this from the usual form of
the equations. The significance of this is highlighted when one considers
what field equations one might naively expect to use in this setting:
the most obvious guess would be to rewrite $\Rico=0$ in terms of the
Ricci tensor of the unphysical metric $g$ and it's Levi-Civita connection
yielding 
\[
\Ric_{ab}+\frac{d-2}{\Omega}\nabla_{a}\nabla_{b}\Omega-\boldsymbol{g}_{ab}\left(\frac{1}{\Omega}\nabla^{c}\nabla_{c}\Omega-\frac{d-1}{\Omega^{2}}(\nabla^{c}\Omega)\nabla_{c}\Omega\right)=0
\]
which degenerates as $\Omega\rightarrow0$.

It is worth noting the significance of the regularity of the conformal
field equations in the area of numerical relativity. The appeal of
``conformal compactification'' to numerical relativists should be
obvious: it means that one is dealing with finite domains. %If the field equations were not regular, but degenerated at the conformal infinity, then the advantage of having a finite domain would be lost as any useful partition of the domain would have to become finer and finer as you approached infinity. The regularity of the field equations ensures that numerical analysis in the conformally compact setting can give meaningful results.

\subsubsection{Different reductions and different forms of the equations}

The conformal field equations are a system of geometric partial differential
equations. In order to study them analytically or numerically they
need to be reduced to a classical system of PDE; this involves introducing
coordinates and adding conditions on various fields to pin down the
natural gauge freedom in the equations (gauge fixing). There is a
significant amount of freedom in how one reduces the system. If this
process is done carefully with the conformal field equations one can
obtain (in $4$ dimensions) a symmetric hyperbolic system of evolution equations together with an elliptic system of constraint equations (reflecting the fact that the conformal field equations are overdetermined). The constraint equations can be taken as conditions on initial data for the Cauchy problem for the conformal field equations, whereas the evolution equations can be taken as prescribing how such data will evolve off the Cauchy hypersurface. For further discussion of the
reduction process and a demonstration of how the conformal field equations
can be reduced see \cite{Joerg-ConfInf}.

One may employ different reductions of the conformal field equations
in different settings and for different purposes. Indeed, the form
of Friedrich's conformal field equations which we presented above
is by no means the only form of the equations that is used, and the
significance of the other forms is that they allow for a still broader
range of different reductions. As an example of this, note that one
could instead have employed the splitting of the tractor bundle induced
by a Weyl connection \cite{CGTracBund} to obtain a system of
conformal field equations in a similar to what we did above; this
would result in a different form of the equations which have different
gauge freedoms and from which we can obtain different reduced systems;
the equations which one would obtain this way are what Friedrich calls
the \emph{general conformal field equations} whereas the equations
we presented above are referred to as the \emph{metric conformal field
equations} \cite{Friedrich-ConfEinsteinEv}. Friedrich also frequently
casts the equations in spinor form, and now we have seen that the
conformal field equations take a very simple tractor form 
\begin{eqnarray*}
\nabla_{a}I^{B} & = & 0\\
I^{A}I_{A} & = & 0\\
\kappal_{ab}{^{C}}_{D}I^{D} & = & 0\\
\nabla_{[a}\kappal_{bc]DE} & = & 0
\end{eqnarray*}
(where we also need to impose the appropriate conditions on $\nabla$
and $\kappal$ as variables). It may indeed prove profitable to see the various reductions of the conformal field equations as being reductions of this tractor system. The gauge freedom(s) in Friedrich's field equations can be seen as coming from the splitting of the tractor equations into tensor equations (along with freedom to choose coordinates). Friedrich's use of conformal geodesics \cite{Friedrich-ConfEinsteinEv} in constructing coordinates as part of the reduction process also fits quite nicely with the tractor picture (see, e.g., \cite{Luebbe, LuebbeTod}).

\begin{remark}
In \cite{FrauSpar} Fraundiener and Sparling observed that in the $4$-dimensional (spin) case Friedrich's conformal field equations could be recast in terms of local twistors involving the so called ``infinity twistor'' $I^{\alpha \beta}$. In this case the bundle of real bitwistors is canonically isomorphic to the standard tractor bundle, and the ``infinity twistor'' $I^{\alpha \beta}$ corresponds (up to a constant factor) to the scale tractor $I^A$, so that the local twistor system in \cite{FrauSpar} is closely related to the tractor system above. Indeed the connection is so close (the modern approach to conformal tractors having developed out of study of the local twistor calculus \cite{BEG}) that we may consider the local twistor formulation to be the origin of the above tractor system. The conformal field equations have also been presented in terms of tractors by Christian L\"ubbe \cite{Luebbe}. L\"ubbe and Tod have applied the tractor calculus to the study of conformal gauge singularities in general relativity (see, e.g., \cite{LuebbeTod}). \endrk
\end{remark}

\section{Lecture 6: Conformal hypersurfaces} \label{ConfHyp}

In order to progress in our study of conformal compactification we
first need to spend some time considering the geometry of embedded
hypersurfaces in conformal manifolds (or of boundaries to conformal
manifolds if you like). In particular we will need to examine how the
standard tractor bundle (and connection, etc.) of the hypersurface
with its induced conformal structure is related to the standard
tractor bundle (and connection, etc.) of the ambient space. Although
our main motivation is the study of conformal compactification for
this lecture we will consider conformal hypersurfaces more generally
since there are many other important kinds of hypersurface which turn
up in general relativity. We will however restrict ourselves to the
case of nondegenerate hypersurfaces, in particular we do not give a
treatment of null hypersurfaces here.

\subsection{Conformal hypersurfaces}
\newcommand{\sbg}{\mbox{\boldmath{\scriptsize$ g$}}}

By a {\em hypersurface} $\Sigma$ in a manifold $M$ we mean a smoothly
embedded codimension 1 submanifold of $M$.  We recall some facts
concerning hypersurfaces in a conformal manifold
$(M^d,c)$, $d\geq 3$. In fact we wish to include the case that $\Sigma$ might be a boundary component. In the case of a pseudo-Riemannian (or conformal pseudo-Riemannian) manifold with boundary then, without further
comment, we will assume that the conformal structure extends smoothly
to the boundary.

Here we shall restrict to hypersurfaces $\Sigma$ with the property
that the any conormal field along $\Sigma$ is nowhere null (i.e. to nondegenerate hypersurfaces).  In this case the restriction of any metric $g\in \cc$ gives a metric $\bar{g}$ on $\Sigma$. Different metrics, among the metrics so obtained, are related conformally and so the conformal class $\cc$ determines a conformal structure $\overline{\cc}$ on $\Sigma$. To distinguish from the ambient objects, we shall overline the corresponding objects intrinsic to this conformal structure. For example $\overline{\bg}$ denotes the conformal metric on $\Sigma$.

Then by working locally we may assume that there is a section
$n_a\in \Gamma(\ce_a[1])$ on $M$ such that, along $\Sigma$, $n_a$ is a
conormal satisfying $|n|^2_{\sbg}:=\bg^{ab}n_a n_b=\pm 1$. This means that if $g=\sigma ^{-2}\bg$ is any metric in $\cc$ then $n^g_a=\sigma^{-1}n_a$ is an extension to $M$ of a unit conormal to $\Sigma$ in $(M,g)$. So $n_a|_\Sigma$ is the conformally invariant version of a unit conormal in pseudo-Riemannian geometry; $n_a$ must have conformal weight 1 since $\bg^{-1}$ has conformal weight $-2$. 

We choose to work with the weighted (extended) conormal field $n_a$ even in the presence of a metric $g$ from the conformal class. Thus we end up with a weight 1 {\em second fundamental form} $L_{ab}$ by restricting, along $\Sigma$, $\nabla_a n_b$ to
$T\Sigma\times T\Sigma\subset (TM\times TM)|_\Sigma$, where $\nabla=\nabla^g$. (Here we are viewing $T^*\Sigma$ as the subbundle of $T^*M|_\Sigma$ orthogonal to $n^a$.) Explicitly $L_{ab}$ is given by
$$
L_{ab}:=\nabla_a n_b\mp n_an^c\nabla_c n_b   \quad \mbox{along }~\Sigma ,
$$ 
since $|n|^2_{\sbg}$ is constant along $\Sigma$.   From this formula, it is easily verified that $L_{ab}$ is independent of how $n_a$ is extended off $\Sigma$. It is timely to note that $L_{ab}$ harbours a hypersurface conformal invariant: Using the formulae \nn{eq:two} and \nn{lct} we compute that under a conformal rescaling, $g\mapsto \widehat{g}=e^{2\om} g$, $L_{ab}$ transforms according to 
$$
L^{\widehat{g}}_{ab}=L^g_{ab}+\overline{\bg}_{ab}\Upsilon_cn^c,
$$
where as usual $\Upsilon$ is the exterior derivative of $\omega$ (which is equal to the log exterior derivative $\Omega^{-1}\mathrm{d}\Omega$ of $\Omega=e^\omega$) and we use this notation below without further mention. 
Thus we see easily the following well-known result:
\begin{proposition}\label{tf2}
The trace-free part of the second fundamental form 
$$
\mathring{L}_{ab}= L_{ab}-H\overline{\bg}_{ab}, \quad \mbox{where}, \quad H:= \frac{1}{d-1}\overline{\bg}^{cd}L_{cd} 
$$ 
is conformally invariant.
\end{proposition} 
The averaged trace of $L$ ($=L^g$), denoted $H$ above, is the {\em mean curvature} of $\Sigma$. Evidently this a conformal $-1$-density and under a conformal rescaling, $g\mapsto \widehat{g}=e^{2\om} g$, $H^g$ transforms to $H^{\widehat{g}}=H^g+n^a\Upsilon_a$.

Thus we obtain a conformally invariant section $N$ of $\cT|_\Sigma$
$$
N_A \stackrel{g}{=}\left(\begin{array}{c}0\\
n_a\\
-H^g\end{array}\right),
$$ 
and from \nn{mform} $h(N,N)=\pm 1$ along $\Sigma$; the conformal invariance follows because the right-hand-side of the display transforms according to the ``tractor characterising'' transformation \nn{ttrans}.  Obviously $N$ is independent of any choices in the extension of $n_a$ off
$\Sigma$. This is the {\em normal tractor} of \cite{BEG} and may be
viewed as a tractor bundle analogue of the unit conormal field from
the theory of pseudo-Riemannian hypersurfaces.

\subsubsection{Umbilicity} 
A point $p$ in a hypersurface is said to be an {\em umbilic point} if,
at that point, the trace-free part $\mathring{L}$ of the second
fundamental form is zero. Evidently this is a conformally invariant
condition.  A hypersurface is {\em totally umbilic} if this holds at
all points.  As an easy first application of the normal tractor we
recall that it leads to a nice characterisation of the the umbilicity
condition.

Differentiating $N$ tangentially along $\Sigma$ using
$\nd^\cT$, we obtain the following result. 
\begin{lemma}
\begin{equation}\label{nN}
\mathbb{L}_{aB}:=\underline{\nabla}_a N_B \stackrel{g_{cb}}{=} \left(\begin{array}{c}0\\
\mathring{L}_{ab}\\
-\frac{1}{d-2}\nabla^b\mathring{L}_{ab} \end{array}\right)
\end{equation}
where $\underline{\nabla}$ is the pullback to
$\Sigma$ of the ambient tractor connection.
\end{lemma}
\begin{proof}
Using the formula \nn{cform} for the
tractor connection, we have
$$
\nabla_c N_B \stackrel{g}{=} \left(\begin{array}{c} -n_c \\
\nabla_c n_b-\bg H\\
-\nabla_c H -P_{cb}n^b  \end{array}\right).
$$ 
Thus applying the orthogonal $TM|_\Sigma \to T\Sigma$
 projector $\Pi^c_a:=(\delta^c_a\mp n^cn_a)$ we
obtain immediately the top two terms on the right-hand-side of
\nn{nN}. The remaining term follows after using the hypersurface 
Codazzi equation for pseudo-Riemannian geometry
$$
\overline{\nabla}_a L_{bc} - \overline{\nabla}_b L_{ac}
= \Pi^{a'}_a \Pi^{b'}_b R_{a'b'cd}n^d,
$$ 
where $\overline{\nabla}$ denotes the Levi-Civita connection for the metric $\overline{g}$ induced by $g$ (for full details see \cite{YuriTh,GoY}).  
\end{proof}
Thus we recover the following result.
\begin{proposition}\label{umbilic}\cite{BEG}
Along a conformal hypersurface $\Sigma$, the normal tractor $N$ is parallel, with respect to $\nabla^{\cT}$,
if and only if  the hypersurface $\Sigma$ is
totally umbilic.
\end{proposition}

\subsubsection{Conformal calculus for hypersurfaces}\label{hcalc}
 
The local calculus for hypersurfaces in Riemannian geometry is to a
large extent straightforward because there is a particularly simple
formula, known as the {\em Gauss formula}, which relates the ambient
Levi-Civita connection to the Levi-Civita connection of the induced
metric.  It is natural to ask if we have the same with our conformal
tractor calculus. 

Given an ambient metric $g$ we write $\underline{\nabla}$ to denote the pullback of the ambient Levi-Civita connection along the embedding of the hypersurface $\Sigma$, i.e. the ambient connection differentiating sections of $TM|_\Sigma$ in directions tangent to the hypersurface. Then the Gauss formula may be expressed as 
$$
\underline{\nabla}_a v^b = \overline{\nabla}_a v^b \mp   n^b L_{ac}v^c
$$
for any tangent vector field $v$ to $\Sigma$ (thought of as a section $v^a$ of $\ce^a|_\Sigma$ satisfying $v^a n_a=0$). We now turn towards deriving a tractor analogue of this.

\newcommand{\Proj}{\operatorname{Proj}}
\renewcommand{\S}{\Sigma}
\newcommand{\bcc}{\overline{\cc}} \newcommand{\bcT}{\overline{\cT}}
Note that by dint of its conformal structure the hypersurface
$(\Sigma,\bcc)$ has its own {\em intrinsic} tractor calculus, and in
particular a rank $d+1$ standard tractor bundle $\bcT$.  Before we
could hope to address the question above we need to relate $\bcT$ to
the ambient tractor bundle $\cT$.

First observe that, along $\Sigma$, $\cT$ has a natural rank
$(d+1)$-subbundle, namely $N^\perp$ the orthogonal complement to
$N_B$. 
As noted in \cite{BrGonon,Grant}, there is a canonical (conformally
invariant) isomorphism
\begin{equation}\label{trisom}
N^\perp \stackrel{\simeq}{\longrightarrow} \ct_\S ~.
\end{equation} 
Calculating in a scale $g$ on $M$ the tractor bundle $\cT$, and hence also $N^\perp$, decomposes into a triple. Then the mapping of the isomorphism 
is 
\begin{equation}\label{xmap}
[N^\perp]_g\ni \left( \begin{array}{c}
\si\\
\mu_b\\ \rho
\end{array} \right) \mapsto 
\left( \begin{array}{c}
\si\\
\mu_b  \mp H n_b \si
\\ \rho \pm \frac{1}{2}H^2 \si
\end{array} \right) \in [\bcT]_{\overline{g}}
\end{equation}
 where, as usual, $H$ denotes the mean curvature of $\S$ in the scale
 $g$ and $\overline{g}$ is the pullback of $g$ to $\S$.  Since
 $(\si,\mu_b,\rho )$ is a section of $[N^\perp]_g$ we have $n^a \mu_a=
 H\si$. Using this one easily verifies that the mapping is conformally
 invariant: If we transform to $\widehat{g}=e^{2\om}g$, $\om\in \ce$,
 then $(\si,\mu_b,\rho )$ transforms according to \nn{ttrans}. Using
 that $\widehat{H}=H+n^a\Upsilon_a$ one calculates that the image of
 $(\si,\mu_b,\rho )$ (under the map displayed) transforms by the
 intrinsic version of \nn{ttrans}, that is by \nn{ttrans} except where
 $\Upsilon_a$ is replaced by $\overline{\Upsilon}_a=\Upsilon_a \mp
 n_an^b\Upsilon_b$ (which on $\S$ agrees with $\overline{\mathrm{d}}
 \omega$, the tangential derivative of $\omega$). This signals that the explicit map displayed in \nn{xmap}
 descends to a conformally invariant map \nn{trisom}. We henceforth
 use this to identify $N^\perp$ with $\cT_\Sigma$, and write
 $\Proj_{\Sigma}:\ct|_\Sigma\to \bcT$ for the orthogonal projection
 afforded by $N$ (or using abstract indices $\Pi^A_B=\delta^A_B \mp N^A N_B$).

It follows easily from \nn{xmap} that the tractor metric $\overline{h}$
on $\bcT$ agrees with the restriction of the ambient tractor metric
$h$ to $N^\perp$. In summary we have: 
\begin{theorem}\label{idts}
Let $(M^d,\cc)$ conformal manifold of dimension $d\geq 4$ and $\Sigma$
a regular hypersurface in $M$. Then, with $\ocT$ deonting the
intrinsic tractor bundle of the induced conformal structure
$\cc_\Sigma$, there is a canonical isomorphism
$$
\ocT\to N^\perp .
$$ Furthermore the tractor metric of $\cc_\Sigma$ coincides with the
pullback of the ambient tractor metric, under this map.
\end{theorem}
\noindent Henceforth we shall simply identify $\ocT$ and $N^\perp$. 

\begin{remark}\label{Hzero}
Note that if $\Sigma$ is minimal in the scale $g$, that is if $H^g=0$, then the isomorphism \nn{xmap} is simply
\begin{equation}\label{xmapHzero}
[N^\perp]_g\ni \left( \begin{array}{c}
\si\\
\mu_b\\ \rho
\end{array} \right) \mapsto 
\left( \begin{array}{c}
\si\\
\mu_b 
\\ \rho
\end{array} \right) \in [\bcT]_{\overline{g}}. 
\end{equation}
Moreover, it is easy to see that one can always find such a minimal scale $g$ for $\Sigma$. Let $g=\sigma^{-2}\bg$ be any metric in $\cc$ and let $\omega:=\mp s \sigma H^g$, where $s$ is a normalised defining function for $\Sigma$ (at least in a neighbourhood of $\Sigma$) and $H^g$ has been extended off $\Sigma$ arbitrarily. Then if $\hat{g}=e^{2\omega}g$ we have that, along $\Sigma$,
$$
H^{\hat{g}}=H^g+n^a\Upsilon_a = H^g+n^a\nabla_a \omega = 0 
$$
since by assumption $\nabla_a s= \sigma^{-1} n_a$ along $\Sigma$ and $s|_\Sigma =0$. \endrk
\end{remark}

\newcommand{\Rho}{P}
\newcommand{\chRho}{\pramb{\Rho}}
\newcommand{\bRho}{\overline{\Rho}}
\newcommand{\bn}{\overline{n}}
Now we have two connections on $\bcT$ that we may compare, namely the
intrinsic tractor connection $\overline{\nabla}^{\bcT}$, meaning the
normal tractor connection determined by conformal structure
$(\Sigma,\cc)$, and the {\em projected ambient tractor connection} 
$\chnabla$. 
On $U\in \Gamma(\bcT)$ the latter is defined by 
$$
\chnabla_a U^B:= \Pi^B_C (\Pi^c_a\nabla_c U^C)  \quad \mbox{along }~\Sigma ,
$$ 
where we view $U\in \Gamma(\bcT)$ as a section of $N^\perp$, and make an arbitrary smooth extension of this to a section of $\cT$ in a neighbourhood (in $M$) of $\Sigma$.  It is then easily verified that $\chnabla$ is a connection on $\bcT$. By construction it is conformally invariant. Thus the difference between this and the intrinsic tractor connection is some canonical conformally invariant section of
$T^*\Sigma\otimes \End (\bcT)$. 

The difference between the projected ambient and the intrinsic tractor
connections can be expressed using the tractor contorsion
$\TS_{a}{}^{B}{}_{C}$ defined by the equation 
\begin{equation} \label{eq:IntrTrCont}
\chnabla_{a}V^{B}=\bnabla_{a}V^{B}\mp\TS_{a}{}^{B}{}_{C}V^{C}
\end{equation}
where $V^{B}\in \Gamma(\overline{\ce}^{B})$ is an intrinsic tractor.
The intrinsic tractor contorsion can be computed explicitly in any scale $g$ to take the form 
\begin{equation} \label{eq:IntrTrContExp}
\TS_{a}{}^{B}{}_{C}\stackrel{g}{=} \bX^{B}\bZ_{C}{}^{c}\cF_{ac}-\bZ^{B}{}_{b}\bX_{C}\cF_{a}{}^{b} 
\end{equation}
or in other words 
\begin{equation} \label{eq:IntrTrContAdj}
\TS_{a B C} = \overline{\bbX}_{B C}{}^{c} \Fialkow_{a c} ,
\end{equation}
where evidently $\Fialkow_{a c}$ must be some conformal invariant of
hypersurfaces (here $\bX^B$ and $\bZ^B{_b}$ are standard tractor projectors for $\Sigma$ and $\overline{\bbX}_{B C}{}^{c}=2\bX_{[B}\bZ_{C]}{}^{c}$, which is conformally invariant). In fact the details of the computation (see \cite{Grant,YuriTh}) reveal this is the {\em Fialkow tensor} (cf.\ \cite{Fialkow})
\begin{equation}
\cF_{ab}=\tfrac{1}{n-2}\Big(W_{acbd}n^cn^d+\mathring{L}_{ab}^{2}-\tfrac{|{\mathring{L}|^2}}{2(n-1)}\overline\bg_{ab}\Big) ,\label{eq:FialkowExpl}
\end{equation}
where $\mathring{L}_{ab}^{2}:= \mathring{L}_{a}{}^c\mathring{L}_{cb}$ and $n=d-1$.
 Altogether we have,
\begin{equation}
\begin{split}
\tnabla_{a}V^{B} {=}& \Pi^{B}{}_{C}\tnabla_{a}V^{C} \pm \TN^{B}\TN_{C}\tnabla_{a}V^{C} \\ 
=& \chnabla_{a}V^{B} \mp \TN^{B}\mathbb{L}_{aC}V^{C} \\
=& \overline{\nabla}_a V^B - {{S_a}^B}_C V^C \mp \TN^B\mathbb{L}_{aC}V^C
\end{split}
\end{equation}
for an intrinsic tractor $V^{B}\in\bcT^{B}$. The tractor Gauss formula is therefore 
\begin{equation} \label{eq:TrGaussFla}
\tnabla_{a}V^{B} = \overline{\nabla}_a V^B \mp {{S_a}^B}_C V^C \mp \TN^B\mathbb{L}_{aC}V^C
\end{equation}
for any $V^{B}\in\bcT^{B}$. Here we see that the object 
$$
\bbL_{aB}:=\tnabla_{a}\TN_{B}
$$ is a tractor analogue of the second fundamental
form, we shall therefore call it the \emph{tractor shape form}.
 
These results provide the first steps in a calculus for conformal
hypersurfaces that is somewhat analogous to the local invariant
calculus for Riemannian hypersurfaces. In particular it can be used to
proliferate hypersurface conformal invariants and conformally
invariant operators \cite{YuriTh,GoY}. We will apply this calculus to the study of conformal infinities in the final two lectures.

\section{Lecture 7:  Geometry of conformal infinity}

Here we apply the results of the previous lecture on conformal hypersurfaces to the geometry of conformal infinity.

\subsection{Geometry of conformal infinity and its embedding}

We return now to the study of conformally compact geometries. We will consider in particular those which near the conformal infinity are asymptotically of constant nonzero scalar curvature. By imposing a constant dilation we may assume that $I^2$ approaches $\pm 1$.

We begin by observing that the normal tractor is linked, in an essential way, to the ambient geometry off the hypersurface $\Sigma:= \cZ(\si)$.
\begin{proposition}\label{ascN}
Let $(M^d,\cc,I)$ be an almost pseudo-Riemannian 
structure with scale singularity set $\Sigma\neq \emptyset$ and
$I^2=\pm 1 + \si^2 f$ for some smooth (weight $-2$) density $f$. Then by Theorem \ref{APRthm} $\Sigma$ is a smoothly embedded hypersurface and, with $N$ denoting the normal tractor for $\Sigma$, we
have $N=I|_\Sigma$.
\end{proposition}
\begin{proof}
For simplicity let us first assume $I^2=\pm 1$ (so $f=0$  and the structure is ASC). 
As usual let us write $\si:=h(X,I)$.  By definition 
$$
I_A=\frac{1}{d}D_A\si\stackrel{g}{=} 
\left(\begin{array}{c} 
\si \cr  \nabla_a \si\cr -\frac{1}{d}(\Delta \si +J \si) 
\end{array}\right) ~,
$$ 
where $g\in \cc$ and $\nabla$ denotes its Levi-Civita connection. 
 Let us write $n_a := \nabla_a \si$.  Along
$\Sigma$ we have $\si=0$, therefore 
 $$
I|_\Sigma  \stackrel{g}{=} \left(\begin{array}{c}
 0 \cr  n_a \cr
 - \frac{1}{d}\Delta \si \end{array}\right)  ~,
$$ along $\Sigma$. Clearly then $|n|^2_{\sbg}=\pm 1$, along $\Sigma$, since $I^2=\pm 1$. So $n_a|_\Sigma$ is a conformal weight 1 conormal field for $\Sigma$.

Next we calculate the mean curvature $H=H^g$ in terms of $\si$.  Recall $
(d-1)H=\nd^an_a \mp n^an^b\nd_b n_a , $ on $\Sigma$. 
 We calculate the right hand side in
a neighbourhood of $\Sigma$. Since $n_a=\nabla_a\si$, we have
$\nd^a n_a=\Delta \si$. On the other hand
$$
n^an^b\nd_b n_a=\frac{1}{2}n^b\nd_b
(n^an_a)=\frac{1}{2} n^b\nd_b(\pm 1+\frac{2}{d}\si\Delta \si +\frac{2}{d}J\si^2),
$$
where we used that $|\frac{1}{d}D \si|^2=\pm 1$ and so $n^an_a=\pm 1+\frac{2}{d}\si\Delta \si +\frac{2}{d}J\si^2 $.
Now along $\Sigma $ we have $\pm 1=n^an_a=n^a\nd_a \si$, and so there this simplifies to 
$$
n^an^b\nd_b n_a=\pm \frac{1}{d}\Delta \si.
$$
Putting these results together, we have
$$
(d-1)H= \frac{1}{d}(d-1)\Delta \si\quad \Rightarrow \quad H=\frac{1}{d}\Delta \si~ \quad \mbox{along } \Sigma .
$$
Thus 
$$
I|_\Sigma  \stackrel{g}{=} \left(\begin{array}{c}
0 \cr  n_a \cr
 -H \end{array}\right) ~,
$$ as claimed. Now note that if we repeat the calculation with $I^2=\pm 1 + \si^2 f$ then the result still holds, as in the calculation this relation was differentiated just once. 
\end{proof}

\begin{corollary}\label{aEu}
Let $(M^d,\cc,I)$ be an almost pseudo-Riemannian structure with scale
singularity set $\Sigma\neq \emptyset$, and that is asymptotically
Einstein in the sense that $I^2|_\Sigma= \pm 1$, and $\nabla_a I_B = \si f_{aB}$ for some smooth (weight -1) tractor valued 1-form $f_{aB}$.  Then $\Sigma $ is a totally umbilic hypersurface.
\end{corollary}
\begin{proof}
The assumptions on the scale tractor $I$ imply that $I^2=\pm 1 + \si^2 f$ for some smooth function $f$ (since we assume $\cc$ and $\si$ smooth).
Thus it follows from Proposition \ref{ascN} above that, along the singularity hypersurface, $I$ agrees with the normal tractor $N$.   Thus $N$ is parallel along $\Sigma$ and so, from Proposition \ref{umbilic}, $\Sigma$ is totally umbilic.
\end{proof}

Note that a hypersurface is totally umbilic if and only if is conformally
totally geodesic: if $\Sigma$ is totally umbilic then locally it is
straightforward to find a metric $g\in \cc$ so that $H^g =0$ whence
$L^g_{ab}=0$ (this was demonstrated in Remark \ref{Hzero}). In this scale any geodesic on the submanifold $\Sigma$, with its
induced metric $\og$, is also a geodesic of the ambient $(M,g)$.  So
the condition of being totally umbilic is a strong matching of the
conformal structures.

In fact for a conformally compact metric that is asymptotically
Einstein, as in Corollary \ref{aEu} above, there is an even stronger
compatibility (involving a higher order of contact) between the geometry of $\cc$ and $\overline{\cc}$.
First a preliminary result. 
\begin{proposition}\label{lastnail}
Let $(M^{d\geq 4},\cc,I)$ be an almost pseudo-Riemannian structure with scale singularity set $\Sigma\neq \emptyset$, and that is asymptotically
Einstein in the sense that $I^2|_\Sigma= \pm 1$, and $\nabla_a I_B = \si^2 f_{aB}$ for some smooth (weight $-2$) tractor valued 1-form  field $f_{aB}$. Then the Weyl curvature $W_{ab}{}^c{}_d$ satisfies 
$$
W_{ab}{}^c{}_dn^d =0, \quad \mbox{along }~\Sigma ,
$$
where $n^d$ is the normal field.
\end{proposition}
\begin{proof}
Since, along $\Sigma$, $I_B$ is parallel to the given order we have that the tractor curvature satisfies 
$$
\kappal_{ab}{}^C{}_D I^D =\kappal_{ab}{}^C{}_D N^D = 0  \quad \mbox{along }~\Sigma ,
$$
and from the formulae for $\kappal$ and $N$ the result is immediate.
\end{proof}

\begin{theorem}\label{umbiltr}
Let $(M^{d\geq 4},\cc,I)$ be an almost pseudo-Riemannian structure with scale singularity set $\Sigma\neq \emptyset$, and that is asymptotically
Einstein in the sense that $I^2|_\Sigma= \pm 1$, and $\nabla_a I_B = \si^2 f_{aB}$ for some smooth (weight $-2$) tractor valued 1-form field $f_{aB}$.
Then the tractor connection of $(M,\cc)$ preserves the intrinsic tractor
bundle of $\Sigma$, where the latter is viewed as a subbundle of the ambient tractors: $\cT_\Sigma\subset \cT$.  Furthermore the restriction of
the parallel transport of $\nabla^\cT$ coincides with the intrinsic tractor parallel transport of $\nabla^{\ct_\Sigma}$.
\end{theorem}

\begin{proof}
From Theorem \ref{idts} the tractor bundle
$\ocT$ of $(\Sigma,\cc)$ may be identified with the subbundle
$N^\perp$ in $\cT|_\Sigma$, consisting of standard tractors orthogonal to the normal tractor $N$. Since $N=I|_\Sigma$ is parallel along
$\Sigma$ this subbundle is preserved by the ambient tractor connection
and the projected ambient tractor connection $\chnabla$ coincides with restriction to $\ocT$ of the pullback of the ambient tractor connection: 
$$
\chnabla_a=\Pi_a^b\nabla^\cT_b \quad \mbox{on}\quad \ocT \subset \cT.
$$

Thus the result follows from the tractor Gauss formula if the Fialkow tensor 
$$
\cF_{ab}=\tfrac{1}{n-2}\Big(W_{acbd}n^cn^d+\mathring{L}_{ab}^{2}-\tfrac{|{\mathring{L}|^2}}{2(n-1)}\overline\bg_{ab}\Big),
$$
of \nn{eq:FialkowExpl} vanishes. Now from
Corollary \ref{aEu} 
$\Sigma$ is totally umbilic, and so $\mathring{L}_{ab}=0$, while Proposition \ref{lastnail} states that $W_{acbd}n^d=0$ along $\Sigma$. Thus $\cF_{ab}=0$.
\end{proof}

\begin{remark}
In the case of an almost Einstein manifold the Theorem \ref{umbiltr} can also be seen to follow 
directly from the general theory of \cite{CGH}, see Theorem 3.5. \endrk
\end{remark}

\section{Lecture 8:  Boundary calculus and asymptotic analysis}

In the previous lectures we have seen that for conformally compact
manifolds with a nondegenerate conformal infinity there are nice tools for studying the link between the conformal structure on the conformal infinity and its relation to the ambient pseudo-Riemannian structure. These showed for example that conditions like asymptotically ASC and then asymptotically Einstein lead to a higher order of contact between the ambient and boundary conformal structures. 

Here we show that the same tools, which are canonical to the geometry  of almost-pseudo Riemannian manifolds, lead first to an interesting canonical calculus along the boundary and then to canonical boundary problems that can be partly solved by these tools. Further details of the boundary calculus for conformally compactified manifolds presented in this lecture, and on applications, can be found in \cite{GLW,GoW-boundcalc, GoW-willmore}.

\newcommand{\w}{\mbox{\bf w}}
\newcommand{\ID}{I\hspace*{-1mm} \cdot \hspace*{-1mm} D}
\subsection{The canonical degenerate Laplacian} \label{degL} On an almost pseudo-Riemannian manifold $(M,\cc,I)$ there is a canonical degenerate Laplacian type differential operator, namely 
$$
\ID:= I^AD_A.
$$
This acts on any weighted tractor bundle, preserving its tensor type
but lowering the weight:
$$
\ID : \ce^\Phi[w]\to  \ce^\Phi[w-1].
$$

In the following it will be useful to define 
define the {\em weight operator} $\w$. On
sections of a conformal density bundle this is just the linear
operator that returns the weight. So if $\tau\in \Gamma (\ce[w_0])$ then
$$
\w \, \tau = w_0 \tau.
$$ Then this is extended in obvious way to weighted tensor or tractor
bundles. So if $\cB$ is some vector bundle of conformal weight zero
then $\w$ acts as the zero operator on its sections and then if $\beta
\in \Gamma (\cB[w_0])$ we have
$$
\w \, \beta = w_0 \beta.
$$ 

Now expanding $\ID$ in terms of some background metric $g\in\cc$, we have 
$$
\ID \stackrel{g}{=}
\left(\begin{array}{rrr}
-\frac{1}{d}(\Delta \si +\J\si) & \nabla^a \si & \si
\end{array}\right)
\left(\begin{array}{lll}
\w (d+2\w-2) \\  \nabla_a ( d+2\w-2) \\ -(\Delta +\J \w)
\end{array}\right).
$$
As an operator on any density or weighted tractor bundle of weight $w$ each occurrence of $\w$ evaluates to $w$. 
So then 
\begin{equation}\label{degLa}
\ID \stackrel{g}{=} -\si \Delta
+(d+2w-2)[(\nabla^a \si)\nabla_a - \frac{w}{d} (\Delta \si)] -\frac{2w}{d}(d+w-1)\si \J~
\end{equation} on $\ce^\Phi [w]$.
Now if we calculate in the metric $\g=\si^{-2}\bg$, away from  the
zero locus of $\si$, and trivialise the densities accordingly,  then $\si$ is represented by 1 in the trivialisation, and we have

$$
\ID \stackrel{\g}{=}-  \Big(\Delta^{\g} + \frac{2w(d+w-1)}{d}\J^{\g} \Big).
$$ In particular if $\g$ satisfies $\J^{\g}=\mp\frac{d}{2}$ ({\it
  i.e.}\ $\Sc^{\g}=\mp d(d-1)$ or equivalently $I^2=\pm 1$) then,
relabeling $d+w-1=:s$ and $d-1=:n$, we have
\begin{equation}\label{IdotDsc}
\ID \stackrel{\g}{=} -  \big(\Delta^{\g} \pm s(n-s) \big).
\end{equation}
On the other hand, looking again to~\nn{degLa}, we see that $\ID$ degenerates along the conformal infinity 
$\Sigma = \mathcal{Z}(\si)$ (assumed non-empty), and there the operator  is
first order. In particular if the structure is asymptotically ASC in the 
sense that $I^2=\pm 1+ \si f$, for some smooth (weight $-1$) density $f$, then along $\Sigma$
$$
\ID = (d+2w-2) \delta_n~,
$$
on $\ce^{\Phi}[w]$, where $\delta_n$  is the conformal Robin operator,
$$
\delta_n \stackrel{g}{=}  n^a\nabla^g_a - w H^g ,
$$
of~\cite{cherrier,BrGonon} (twisted with the tractor connection);
here $n^a$ is a length $\pm 1$ normal and~$H^g$ the mean curvature, as measured
in the metric $g$.

\subsection{A boundary calculus for the degenerate Laplacian}
\label{boundcalc}
%\edz{RS: R HERE 31/5/2014 in first pass thru this section}

Let $(M,\cc)$ be a conformal structure of dimension $d\geq 3$ and of any
signature. 
Given $\si$ a section of $\ce[1]$, write $I_A $ for the
corresponding scale tractor as usual. That is $I_A= \frac{1}{d}D_A\si $.
Then $\si=X^AI_A$.

\subsubsection{The ${\mathfrak sl}(2)$-algebra}
 Suppose that $f\in \ce^\Phi[w]$, where
$\ce^\Phi$ denotes any tractor bundle. Select $g\in \cc$ for the purpose
of calculation, and write $I_A\stackrel{g}{=}(\si,~\nu_a,~\rho)$ to
simplify the notation.  Then using $\nu_a=\nabla_a\si$,  we have
\begin{equation*}
\begin{split}
\ID\big(\si f\big) =  (d+2w)\big((w+1) \si \rho f +\si \nu_a\nabla^a f+ 
f \nu_a\nu^a\big)\quad\\[2mm] \qquad\qquad  -\si\big(\si \Delta f+2\nu_a\nabla^a f +f\Delta \si +(w+1)\J\si f\big)\, ,
\end{split}
\end{equation*}
while
$$ - \si\,  \ID\,  f = -\si(d+2w-2)  \big(w\rho f +
\nu_a\nabla^a f\big) +\si^2 (\Delta f + w \J f)\, .
$$
So, by virtue of the fact that $\rho=-\frac1d(\Delta\si+\J\si)$,  we have
$$
[\ID,\si] f= (d+2w) (2\si\rho  +\nu_a\nu^a) f .
$$
Now $I^AI_A=I^2\stackrel{g}{=}2\si\rho+\nu_a\nu^a$, whence
the last display simplifies to
$$
[\ID,\si] f =(d+2w) I^2 f . 
$$ 
Denoting by $\w$ the weight operator on
tractors, we have the following. 
\begin{lemma}\label{slem}
Acting on any section of a weighted tractor bundle we have 
$$
[\ID,\si]  =  I^2 (d+2\w)  ,
$$
where $\w$ is the weight operator.
 \end{lemma}

\begin{remark}\label{remarb}
	 	A  similar  computation  to  above  shows  that, more generally,
	 	$$
	 	I\cdot  D  \big(\sigma^\alpha  f\big)  -  \sigma^\alpha   
		I\cdot  D  f  =  \sigma^{\alpha-1}\alpha\,   I^2 (d+2\w+\alpha-1)   f\,  ,
	 	$$
	 	for  any  constant  $\alpha$. \endrk
\end{remark}

 The operator $\ID$ lowers conformal weight by 1. On the other
 hand, as an operator (by tensor product) $\si$ raises conformal
 weight by 1.  We can record this by the commutator relations
$$
[\w,\ID]=-\ID\quad \mbox{and} \quad [\w, \si] =\si, 
$$
so with the Lemma we see that the operators $\si$, $\ID$,
 and $\w$, acting on weighted scalar or tractor fields, generate an
${\mathfrak sl}(2)$ Lie algebra, provided $I^2$ is nowhere vanishing. It
is convenient to fix a normalisation of the generators; we record
this and our observations as follows. 
\begin{proposition}\label{slprop} Suppose that $(M,c,\si)$ 
is such that
  $I^2$ is nowhere vanishing.   Setting $x:=
  \si$, $y:= -\frac{1}{I^2}\ID$, and $h:=d+2\w$ we obtain the
  commutation relations
$$
[h,x]=2x, \quad [h,y]=-2y, \quad [x,y]=h, 
$$
of standard  ${\mathfrak sl}(2)$-algebra generators. 

In the case of $I^2=0$ the result is an
 In\"{o}n\"{u}-Wigner contraction of the~${\mathfrak sl}(2)$-algebra: 
$$
[h,x]=2x, \quad [h,y]=-2y, \quad [x,y]=0, 
$$ 
where  $h$ and $x$ are as before, but now $y=-\ID$. 
\end{proposition}

\noindent
Subsequently $\mathfrak{g}$ will be used to denote this (${\mathfrak sl}(2)$)
Lie algebra of operators.
From Proposition~\ref{slprop}  (and  in  concordance  with  remark~\ref{remarb}) follow some useful identities in the universal
enveloping algebra $\cU(\mathfrak{g})$.
\begin{corollary}\label{corpid}
\begin{eqnarray}
&[x^k,y]=x^{k-1} k(d+2\w+k-1) = x^{k-1} k (h+k-1)&\nonumber \\[1mm]{} & {\rm  and }&\label{pid}\\[1mm]{} & \ [x,y^k]=y^{k-1}
k(d+2\w-k+1)=y^{k-1}k(h-k+1)\, .&\nonumber
\end{eqnarray}
\end{corollary}

\subsection{Tangential operators and holographic formul\ae\ }
\label{tangential ops}

Suppose that $\si\in \Gamma(\ce[1])$ is such that $I_A=\frac{1}{d}D_A
\si$ satisfies that $I^AI_A= I^2$ is nowhere zero.  As explained in
Section~\ref{scale}, the zero locus
$\mathcal{Z}(\si)$ of~$\si$ is then either empty or forms a smooth
hypersurface.

Conversely if $\Sigma$ is any smooth oriented hypersurface then, at
least in a neighbourhood of $\Sigma$, there is a smooth defining
function $s$.  Now take $\si\in \Gamma (\ce[1])$ to be the unique density which gives $s$ in the trivialisation of $\ce[1]$ determined by some $g\in \cc$. It follows then that $\Sigma =\mathcal{Z}(\si)$ and $\nabla^g \si $ is non-zero at all points of $\Sigma$. If $\nabla^g \si $ is nowhere null along~$\Sigma$ then $\Sigma$ is nondegenerate and $I^2$ is nowhere vanishing in a neighbourhood of $\Sigma$, and we are in the situation of the previous paragraph. We call such a $\si$ a {\em
defining density} for $\Sigma$ (recall our definition from Lecture \ref{ConfHyp}), and to simplify the discussion we shall take $M$ to agree with this neighbourhood of $\Sigma$. Until further notice $\si$ will mean such a section of $\ce[1]$ with $\Sigma=\mathcal{Z}(\si)$ non-empty and nondegenerate.  Note that $\Sigma$ has a conformal structure $\overline{\cc}$ induced in the obvious way from $(M,\cc)$ and is a conformal infinity for the metric $\g:= \si^{-2}\bg$ on $M\setminus \Sigma$. 

\subsubsection{Tangential operators}\label{tanS}
Suppose that $\si$ is a defining density for a hypersurface $\Sigma$.
Let $P: \ce^\Phi[w_1]\to \ce^\Phi[w_2]$ be some linear differential
 operator defined in a neighbourhood of $\Sigma$. 
 We shall say that $P$ {\em acts
  tangentially (along $\Sigma$)} if $P \circ \si =\si \circ \widetilde P $, where $\widetilde P
:\ce^\Phi[w_1-1]\to \ce^\Phi[w_2-1]$ is some other linear
 operator on the same neighbourhood.
The point is that for a tangential operator $P$ we have 
$$
P (f+\si h)= P f+ \si \widetilde P h.
$$  
Thus along $\Sigma$ the operator $P$ is insensitive to how $f$ is
extended off $\Sigma$.  It is easily seen that if $P$ is a tangential differential operator then there is a formula for $P$,
along $\Sigma$, involving only derivatives tangential to $\Sigma$, and
the converse also holds.

Using Corollary~\ref{corpid} we see at once that certain powers of
$\ID$ act tangentially on appropriately weighted tractor bundles.
We state this precisely. Suppose that $\Sigma$ is a (nondegenerate) hypersurface in a conformal manifold $(M^{n+1},\cc)$, and $\si$ a defining density for $\Sigma$. Then recall $\Sigma =\mathcal{Z}(\si)$ and $I^2$ is nowhere zero in a neighbourhood of $\Sigma$, where $I_A:=\frac{1}{n+1}D_A\si$
is the scale tractor.  The following holds.
\begin{theorem}\label{tanth}
  Let $\ce^\Phi$ be any tractor bundle and $k\in \mathbb{Z}_{\geq 1}$.
  Then, for each $k\in \mathbb{Z}_{\geq 1}$, along $\Sigma$
\begin{equation}\label{tan}
P_k: \ce^\Phi[\frac{k-n}{2}]\to \ce^\Phi[\frac{-k-n}{2}]
\quad \mbox{given by}\quad 
P_k:= \Big(-\frac{1}{I^2} \ID\Big)^k
\end{equation}
is a tangential differential operator, and so determines a canonical
differential operator $P_k: \ce^\Phi[\frac{k-n}{2}]|_\Sigma \to
\ce^\Phi[\frac{-k-n}{2}]|_\Sigma$.
\end{theorem}
\begin{proof}
  The $P_k$ are differential by construction. Thus the result is
  immediate from Corollary~\ref{corpid} with $\widetilde P =
  \Big(-\frac{1}{I^2} \ID\Big)^k$.
\end{proof}

\subsection{The extension problems and their asymptotics} \label{ext}

Henceforth we consider an almost pseudo-Riemannian structure $(M,c,\si)$ with $\si$ a defining density for a hypersurface $\Sigma$ and $I^2$ nowhere zero.  We consider the problem of solving, off  $\Sigma$ asymptotically,
$$
\ID
f=0\, ,
$$ 
for $f\in \Gamma (\ce^\Phi[w_0])$ and some given weight $w_0$. For
simplicity we henceforth calculate on the side of $\Sigma$ where $\si$ is
non-negative, so effectively this amounts to working locally along the boundary of a conformally compact manifold. 

\subsubsection{Solutions of the first kind}\label{first}
Here we treat an obvious Dirichlet-like problem where we view
$f|_\Sigma$ as the initial data.  Suppose that $f_0$ is an arbitrary
smooth extension of $f|_\Sigma$ to a section of $\ce^\Phi[w_0]$ over
$M$. We seek to solve the following problem:

\begin{problem}\label{extp}
  Given $f|_{\Sigma}$, and an arbitrary extension $f_0$ of this to
  $\ce^\Phi[w_0]$ over~$M$, find $f_i\in \ce^{\Phi}[w_0-i]$ (over
  $M$), $i=1,2,\cdots$, so that
$$
f^{(\ell)}: = f_0 + \si f_1  + \si^2f_2 +\cdots + O(\si^{\ell+1})
$$ solves $\ID f = O(\si^\ell)$, off $\Sigma$, 
for $\ell\in \mathbb{N}\cup \infty$
as high as possible.
\end{problem}
\begin{remark}
  For $i\geq 1$ we do not assume that the $f_i$ are necessarily
  non-vanishing along~$\Sigma$. Also, note that the stipulation ``given $f|_{\Sigma}$, and an arbitrary extension $f_0$ of this'' can be rephrased as ``given $f_0$ in the space of sections with a fixed restriction to ${\Sigma}$ (denoted $f|_{\Sigma}$)''. \endrk
\end{remark}

We write $h_0 =d+2w_0$ so that $h f_0=h_0 f_0$, for example.  The
existence or not of a solution at generic weights is governed by the
following result.  
\begin{lemma}\label{gjmsstyle}
  Let $f^{(\ell)}$ be a solution of Problem~\ref{extp} to order
  $\ell\in \mathbb{Z}_{\geq 0}$. Then provided $\ell\neq h_0-2= n+2w_0-1$ 
  there is an extension 
$$
f^{(\ell+1)}= f^{(\ell)}+\si^{\ell+1}
  f_{\ell+1},
$$
 unique modulo $\si^{\ell+2}$, which solves
$$
\ID f^{(\ell+1)}= 0 \quad \mbox{modulo} \quad O(\si^{\ell+1}).
$$  

If $\ell = h_0-2$ then the extension is obstructed by $P_{\ell+1} f_0|_\Sigma$, where $P_{\ell+1} = (-\frac{1}{I^2} \ID f)^{\ell+1}$ is the tangential operator on densities of weight $w_0$ given by Theorem \ref{tanth}. 
\end{lemma}
\begin{proof} Note that $\ID f=0$ is equivalent to
  $-\frac{1}{I^2} \ID f=0$ and so we can recast this as a formal
  problem using the Lie algebra $\mathfrak{g}=\langle x,y,h\rangle$ from
  Proposition~\ref{slprop}. 
Using the notation from there
$$
y f^{(\ell+1)} = y f^{(\ell)} - 
x^\ell (\ell+1)(h+\ell)f_{\ell+1}+O(x^{\ell+1}). 
$$
Now $h f_{\ell+1}=\big(h_0-2(\ell+1)\big) f_{\ell+1}$, thus 
\begin{equation}\label{key1}
y f^{(\ell+1)} = y f^{(\ell)} - 
x^\ell (\ell+1)(h_0-\ell-2)f_{\ell+1}+O(x^{\ell+1}). 
\end{equation}
By assumption $y f^{(\ell)}=O(x^\ell)$, thus if $\ell\neq h_0-2$
we can solve $ y f^{(\ell+1)} = O(x^{\ell+1})$ and this uniquely
determines $f_{\ell+1}|_\Sigma$.

On the other hand if $\ell= h_0-2$ then~\nn{key1} shows that, 
modulo $O(x^{\ell+1})$, 
$$
y f^{(\ell)}= y \big( f^{(\ell)} + x^{\ell+1}f_{\ell+1} \big),
$$
regardless of $f_{\ell+1}$. It follows that the map $f_0\mapsto x^{-\ell} y
f^{(\ell)}$ is tangential and $ x^{-\ell} y f^{(\ell)}|_\Sigma$
is the obstruction to solving $ y f^{(\ell+1)} = O(x^{\ell+1})$.
By a simple induction this is seen to be a non-zero multiple of 
$y^{\ell+1}f_0|_\Sigma$.
\end{proof}

Thus by induction we conclude the following.
\begin{proposition}\label{gformal}
  For $h_0\notin \mathbb{Z}_{\geq 2}$ Problem~\nn{extp} can be solved
  to order $\ell$~$=$~$\infty$. For $h_0\in \mathbb{Z}_{\geq 2}$ the solution is
  obstructed by $[P_{h_0-1} f]|_\Sigma$; if, for a particular $f$,
  $[P_{h_0-1} f]|_\Sigma$ $=$~$0$ then there is a family of solutions to
  order $\ell=\infty$ parametrised by sections $f_{h_0-1}\in \Gamma
  \ce^\Phi[-d-w_0+1]|_\Sigma$.
\end{proposition}

\appendix

\section{Conformal Killing vector fields and adjoint tractors}

Here we discuss conformal Killing vector fields and the prolongation
of the conformal Killing equation which is given by a connection on a
conformal tractor bundle called the adjoint tractor bundle. Although
the adjoint bundle is not a new tractor bundle \emph{per se} (being
the bundle of skew endomorphisms of $\ce^A$, naturally isomorphic to
$\ce_{[AB]}$), it is a very important tractor bundle and is worthy of
our further attention; we discuss briefly the connection of the
conformal Cartan bundle with the adjoint tractor bundle and its
calculus. We finish this section by applying the adjoint tractor
calculus and the prolongation of the conformal Killing equation to Lie
derivatives of tractors with respect to conformal Killing vector
fields.

\subsection{The conformal Cartan bundle and the adjoint tractor bundle} \label{CartanConn}

In the lectures we have been discussing tractor calculus on the so called standard tractor bundle on a conformal manifold $(M,\cc)$, which can be seen as the associated bundle $\mathcal{G}\times_P \mathbb{R}^{p+1,q+1}$ to the $P$-principal conformal Cartan bundle $\mathcal{G}\rightarrow M$. We have not discussed the construction of the conformal Cartan bundle in the lectures because it is easier to construct the standard tractor bundle directly from the underlying conformal structure. Indeed, having understood the conformal standard tractor bundle one could simply define the conformal Cartan bundle as an adapted frame bundle for the standard tractor bundle $\mathcal{T}$; from this point of view the conformal Cartan connection simply encodes the tractor parallel transport of the frames, c.f. \nn{connids}. Although for calculational purposes the standard tractor approach presented in the lectures is optimal, we have seen (e.g., in section \ref{HolonomyRed}) that the Cartan geometric point of view can afford additional insights. This is also the case when it comes to discussing the adjoint tractor bundle.

If $P\subset G=\mathrm{O}(p,q)$ is the stabiliser of a null ray in $\mathbb{R}^{p+1,q+1}$ then $P$ contains a subgroup which may be identified with $\mathrm{CO}(p,q)$. The conformal Cartan bundle of $(M,c)$ is then easily defined as the extension of the conformal orthogonal frame bundle $\mathcal{G}_0\rightarrow M$ to a principal $P$-bundle $\mathcal{G}\rightarrow M$ corresponding to the inclusion $\mathrm{CO}(p,q)\subset P$. The total space of this bundle is then $\mathcal{G}_0\times_{CO(p,q)}P$. 

The conformal Cartan connection $\omega$ is a 1-form on $\mathcal{G}$ taking values in the Lie algebra of $G$, which can be written as a (vector space) direct sum
$$
\mathbb{R}^d \oplus \mathfrak{co}(p,q) \oplus \left( \mathbb{R}^d\right)^*
$$ 
where $d=p+q$. The $\mathbb{R}^d$ component of $\omega$ is simply the trivial extension of the \emph{soldering form} of $\mathcal{G}_0$ (which is tautologically defined on any reduction of the frame bundle of a manifold) to $\mathcal{G}$; this component corresponds to the terms $-\mu_a$ and $+\bg_{ab}\rho$ in the formula \nn{cform} for the standard tractor connection. The next component arises from noting that sections of the bundle $\mathcal{G}\rightarrow \mathcal{G}_0$ are in 1-1 correspondence with affine connections on $M$ which preserve the conformal metric $\bg$ (\emph{Weyl connections}); the $\mathfrak{co}(p,q)$ component of $\omega$ is the 1-form $\boldsymbol{\gamma}$ on $\mathcal{G}$ whose pullback to $\mathcal{G}_0$ by any section is the connection 1-form for the corresponding Weyl connection. In order to obtain the formula for the tractor connection from that of the Cartan connection one must first pull everything down from $\mathcal{G}$ to $\mathcal{G}_0$ using a Weyl connection, allowing you to break tractors (in associated bundles to $\mathcal{G}$) up into weighted tensors (in associated bundles to $\mathcal{G}_0$). If one does this using the Levi-Civita $\nabla$ connection of a metric $g$ then one can see that the component $\boldsymbol{\gamma}$ of $\omega$ gives rise to the three terms involving $\nabla$ in \nn{cform}. The final component of $\omega$ gives rise to the two terms involving the Schouten tensor in \nn{cform} and is determined by the other two components of $\omega$ and an algebraic condition on the curvature 
$$
\mathrm{d}\omega + \omega \wedge \omega
$$ 
of $\omega$ (which can be identified with the tractor curvature). For more details see the first chapter of \cite{CSbook}.

The conformal Cartan connection has the following three properties:
\begin{itemize}
\item $\omega_u : T_u \mathcal{G} \rightarrow \mathcal{G}\times \mathfrak{g}$ is a linear isomorphism for each $u\in \mathcal{G}$;
\item $\omega$ is $P$-equivariant, i.e. $(r^p)^*\omega = Ad(p^{-1})\circ \omega$ for all $p\in P$ (where $r^p$ denotes the right action of $p$ on $\mathcal{G}$);
\item $\omega$ returns the generators of fundamental vector fields, i.e.
$$
\omega \left(\left.\frac{d}{dt}\right|_0u\cdot exp(tX)\right)=X
$$
for all $u\in \mathcal{G}$ and $X\in \mathfrak{p}=\mathrm{Lie}(P)$.
\end{itemize}
These properties more generally define what is called a \emph{Cartan connection} of type $(G,P)$ on a $P$-principal bundle $\mathcal{G}$. In the model case the Lie group $G$ can be seen as a $P$-principal over the model space $G/P$ and the Maurer-Cartan 1-form $\omega_{MC}$, which evaluates left invariant vector fields at the identity, is a Cartan connection which has vanishing curvature by the Maurer-Cartan structure equations
$$
\mathrm{d}\omega_{MC} + \omega_{MC}\wedge \omega_{MC} =0.
$$

The \emph{adjoint tractor bundle} $\mathcal{A}\rightarrow M$ of a conformal manifold $(M,\cc)$ is the associated bundle to the conformal Cartan bundle $\mathcal{G}$ corresponding to the adjoint representation of $G$ (restricted to $P$), i.e. $\mathcal{A}=\mathcal{G}\times_P \mathfrak{g}$. Since $\mathfrak{g}=\mathfrak{so}(p+1,q+1)$ can be identified with the skew-symmetric endomorphisms of $\mathbb{R}^{p+1,q+1}$ the adjoint tractor bundle can similarly be identified with the bundle of skew-symmetric endomorphisms of the standard tractor bundle $\mathcal{T}$ (with respect to the tractor metric). Clearly then we may identify $\mathcal{A}$ with $\mathcal{E}_{[AB]}$ by lowering a tractor index. An adjoint tractor $\mathbb{L}^A{_B}$ can be written in terms of the direct sum decomposition of $\mathcal{T}$ (and hence $\mathcal{A}$) as 
\begin{equation} \label{AdjTrac}
\left(\begin{array}{ccc} -\nu & -l_b & 0 \\
-\rho^a & \mu^a{_b} & l^a \\ 
0 & \rho_b & \nu \\
\end{array}\right)
\end{equation}
where $\mu_{ab}=\mu_{[ab]}$ and the matrix acts on standard tractors from the left, as the tractor curvature does in \nn{curvform} (the tractor curvature is in fact an adjoint tractor valued 2-form). It is easy to see that the two appearances of $l^a$ make up the ``top slot'' of $\mathbb{L}^A{_B}$, so that there is an invariant projection $\Pi$ from $\mathcal{A}$ to $TM$ that takes $\mathbb{L}^A{_B}$ to $l^a$.

By writing an adjoint tractor $\mathbb{L}^A{_B}$ in terms of the splitting tractors ($X^A$, $Z^A_a$, $Y^A$) corresponding to a choice of metric $g$ one can easily obtain the formula for the tractor connection acting on $\mathbb{L}^A{_B}$ using \nn{connids}. If $\mathbb{L}^A{_B}$ is given in matrix form (w.r.t. $g$) by \nn{AdjTrac} then we have
\begin{equation} \label{AdjTracConn}
\nabla_a \mathbb{L}^B{_C} \overset{g}{=} 
\left(\begin{array}{ccc} * & * & 0 \\
* & \bg^{bb'}(\nabla_a\mu_{b'c} - 2\V_{a[b'}l_{c]} + 2\bg_{a[b'}\rho_{c]}) & \nabla_a l^b + \mu_a{^b} + \nu \delta_a^b \\ 
0 & \nabla_a\rho_c - \V_a^b\mu_{bc} +\nu \V_{ac} & \nabla_a\nu -\V_a^b l_b -\rho_a\\
\end{array}\right)
\end{equation}
where the entries marked with a $*$ are determined by skew-symmetry.

\subsection{Prolonging the conformal Killing equation}
Note that if $\mathbb{L}$ is a parallel adjoint tractor, and $l=\Pi(\mathbb{L})$, then from the above display we must have
$$
\nabla_a l_b + \mu_{ab} + \nu \bg_{ab} = 0
$$
where $\mu_{ab}$ is skew. This implies that $\nabla_{(a} l_{b)_o}=0$, in other words $l^a$ is a \emph{conformal Killing vector field}. A conformal Killing vector field $k^a$ is a nonzero solution of the conformally invariant equation
$$
\nabla_{(a} k_{b)_o} =0,
$$ 
where $k_b=\bg_{bc}k^c$. Geometrically this equation says that the local flow of $k$ preserves any metric $g\in \cc$ up to a conformal factor, equivalently, the Lie derivative of any metric $g\in \cc$ with respect to $k$ is proportional to $g$. It is natural to ask whether there is a 1-1 correspondence between conformal Killing vector fields and (nonzero) parallel adjoint tractor fields -- the answer to this question turns out to be no, except for on the flat model (where one can use this correspondence to easily write the $d(d-1)$-dimensional space of Killing vector fields explicitly).

One can however construct a different conformally invariant connection on the adjoint tractor bundle which does prolong the conformal Killing equation, i.e. for which parallel sections are in 1-1 correspondence with solutions. One can obtain this system directly by (fixing $g$ and $\nabla=\nabla^g$ and then) writing the equation $\nabla_{(a} k_{b)_o} =0$ as
$$
\nabla_a k_b = \mu_{ab} + \nu \bg_{ab},
$$
introducing the new variables $\mu_{ab}\in \Gamma(\ce_{[ab]}[2])$ and $\nu\in \Gamma(\ce)$. Beyond this the key step is to introduce the fourth variable 
$$
\rho_a = \nabla_a\nu +\V_{ab}k^b \in \Gamma(\ce_a)
$$
(rather than $\rho_a=\nabla_a \nu$), c.f. \nn{AdjTracConn}. The remainder of the prolongation process simply involves taking covariant derivatives of the above two displays and then skew-symmetrising over certain pairs of indices in them in order to bring out curvature terms (as well as using Bianchi identities to simplify expressions); from these ``differential consequences'' of the above two displays one may derive expressions for $\nabla_a\mu_{bc}$ and $\nabla_a\rho_{b}$ which are linear in the other three variables (respectively). (The way to accomplish this process efficiently is to suppose you have a flat Levi-Civita connection first and go through the process to obtain both expressions, then go back through the same steps and take into account the non-vanishing curvature for the general case.) The result is the following system of differential equations
\begin{eqnarray*}
\nabla_a k_b &=& \nu \bg_{ab} + \mu_{ab} \\
\nabla_a \mu_{bc} &=& - 2\V_{a[b}k_{c]} - 2\bg_{a[b}\rho_{c]} +W_{dabc}k^d \\
\nabla_a \nu &=& \rho_a - \V_{ab}k^b \\
\nabla_a \rho_b &=& -\V_a^c\mu_{bc} - \V_{ab}\nu - C_{cab}k^c
\end{eqnarray*}
where $C_{abc}=2\nabla_{[a}\V_{bc]}$ is the Cotton tensor.

From the above system we can see that if we define the connection $\tilde{\nabla}$ on $\mathcal{A}$ by
\begin{equation}
\tilde{\nabla}\mathbb{L} = \nabla \mathbb{L} + i_{\Pi (\mathbb{L})}\kappal \quad \mbox{for all} \quad \mathbb{L}\in \Gamma(\mathcal{A})
\end{equation}
then there is a one to one correspondence between (nonzero) $\tilde{\nabla}$-parallel sections of $\mathcal{A}$ and conformal Killing vector fields on $(M,\cc)$. To check this we simply calculate (in a scale $g$) that
\begin{align*}
\tilde{\nabla}_a 
\left(\begin{array}{ccc} * & * & 0 \\
* & \mu^b{_c} & l^b \\ 
0 & \rho_c & \nu \\
\end{array}\right)
 &=
\left(\begin{array}{ccc} * & * & 0 \\
* & \bg^{bb'}(\nabla_a\mu_{b'c} - 2\V_{a[b'}l_{c]} + 2\bg_{a[b'}\rho_{c]}) & \nabla_a l^b + \mu_a{^b} + \nu \delta_a^b \\ 
0 & \nabla_a\rho_c - \V_a^b\mu_{bc} +\nu \V_{ac} & \nabla_a\nu -\V_a^b l_b -\rho_a\\
\end{array}\right) \\
& \quad \; \; +
\left(\begin{array}{ccc} * & * & 0 \\
* & W_{da}{^b}_c l^d & 0 \\ 
0 & -C_{dac}l^d & 0 \\
\end{array}\right)
\end{align*}
and observe that by setting the right hand side equal to zero (and substituting $k^a=-l^a$) we recover our prolonged system for the conformal Killing equation. Note that it is possible to take a a much more abstract and abstract approach to obtaining this prolonged system and the corresponding connection $\tilde{\nabla}$ on $\mathcal{A}$ (see, e.g., \cite{CapInfAut,CSbook}). From the general theory (or direct observation) we also have the invariant linear differential operator $L: TM\rightarrow \mathcal{A}$ which takes a vector field $l^a$ on $M$ to the adjoint tractor $\mathbb{L}^A{_B}$ given in a scale $g$ by \nn{AdjTrac} with $\mu_{ab}=-\nabla_{[a}l_{b]}$, $\nu=-\frac{1}{d}\nabla_a l^a$, and $\rho_a = \nabla_a \nu - \V_{ab}l^b$. Clearly $\Pi\circ L =\mathrm{id}_{TM}$ and consequently $L$ is referred to as a \emph{differential splitting operator}.

\subsection{The fundamental derivative and Lie derivatives of tractors}

Let $\mathbb{V}$ be a representation of $P$ and let $\mathcal{V}=\mathcal{G}\times_P \mathbb{V}$. If a function $\tilde{v}:\mathcal{G}\rightarrow \mathbb{V}$ satisfies
$$
\tilde{v}(u\cdot p) = p^{-1}\cdot \tilde{v}(u) \quad \mbox{for all}\quad  u\in \mathcal{G},\, p\in P  
$$
then we say that $\tilde{v}$ is $P$-\emph{equivariant}. Observe that if $\tilde{v}:\mathcal{G}\rightarrow \mathbb{V}$ is a smooth $P$-equivariant map then the map $v:M\rightarrow \mathcal{V}$ which takes $x$ to $[(u,v(u))]$ for some $u\in \mathcal{G}_x$ defines a smooth section of $\mathcal{V}\rightarrow M$ (being independent of the choice of $u\in \mathcal{G}_x$ for each $x$). It is easy to see that sections of $\mathcal{V}\rightarrow M$ are in 1-1 correspondence with such functions. (We have used this already in discussing conformal densities where $\mathbb{R}_+$ replaces $P$ and $\mathcal{Q}$ replaces $\mathcal{G}$, see section \ref{cg}.) Now observe that the conformal Cartan connection $\omega$ allows us to view smooth $P$-equivariant functions $\tilde{\mathbb{L}}:\mathcal{G}\rightarrow\mathfrak{g}$ as smooth vector fields $V_\mathbb{L}$ on $\mathcal{G}$ (since $\omega_u:T_u \mathcal{G}\rightarrow \mathfrak{g}$ is an isomorphism for each $u\in\mathcal{G}$ and $\omega$ is smooth). From the properties of the Cartan connection is not hard to see that any such vector field $V_\mathbb{L}$ must be $P$-invariant, that is, 
$$
(r^p)^* V_\mathbb{L} = V_\mathbb{L} \quad \mbox{for all} \quad p\in P
$$
where $r^p:\mathcal{G}\rightarrow \mathcal{G}$ denotes the right action of $p\in P$. Moreover, the Cartan connection gives a 1-1 correspondence between $P$-equivariant functions $\tilde{\mathbb{L}}:\mathcal{G}\rightarrow\mathfrak{g}$ and $P$-invariant vector fields on $\mathcal{G}$. Combining this with the previous observation we see that one may naturally identify $\Gamma(\mathcal{A})$ with the space $\mathfrak{X}(\mathcal{G})^P$ of $P$-invariant vector fields on $\mathcal{G}$.

The observations of the preceding paragraph allow us to define a new canonical differential operator acting on sections of any vector bundle $\mathcal{V}$ associated to the conformal Cartan bundle $\mathcal{G}$. Fix $\mathbb{L}\in \Gamma(\mathcal{A})$ and let $v\in \Gamma(\mathcal{V})= \Gamma(\mathcal{G}\times_P \mathbb{V})$, then $V_{\mathbb{L}}\tilde{v}=\mathrm{d}\tilde{v}(V_{\mathbb{L}})$ is again a $P$-equivariant function from $\mathcal{G}$ to $\mathbb{V}$ and thus defines a section $D_\mathbb{L} v$ of $\mathcal{V}$. Thus for each section $\mathbb{L}\in \Gamma(\mathcal{A})$ we have a first order differential operator $D_\mathbb{L}:\mathcal{V}\rightarrow \mathcal{V}$. It is easy to see that $D_{f\mathbb{L}}=fD_\mathbb{L}$ for any $f\in C^\infty(M)$ so that we really have a (first order) differential operator taking sections of $\mathcal{V}$ to sections of $\mathcal{A}^*\otimes \mathcal{V}$. The operator
$$
D:\mathcal{V}\rightarrow \mathcal{A}^*\otimes \mathcal{V}
$$
defined in this way is referred to as the \emph{fundamental derivative} (or \emph{fundamental D operator}). This operator was introduced by \v Cap and Gover in \cite{CapGoTAMS}.

Now if $k$ is a conformal Killing vector field on $(M,\cc)$ and $v$ is a section of $\mathcal{V}=\mathcal{G}\times_P \mathbb{V}$ then we may talk about the Lie derivative of $v$ with respect to $k$; since $\mathcal{V}\rightarrow M$ is a natural bundle in the category of conformal manifolds (with diffeomorphisms as maps) one may pull the section $v$ of $\mathcal{V}$ back by the (local) flow of $k$ and define the Lie derivative by
$$
\mathcal{L}_k v = \left.\frac{d}{dt}\right|_{t=0}(Fl^k_t)^*(v).
$$
Notice that if $\mathbb{L}$ is an section of the adjoint tractor bundle then $V_{\mathbb{L}}\tilde{v}$ (which gives the $P$-equivariant function corresponding to $D_{\mathbb{L}}v$) may be written as the Lie derivative $\mathcal{L}_{V_{\mathbb{L}}}\tilde{v}$, it should not come as a total surprise then that the Lie derivative and the fundamental derivative are connected. In fact, one has that
$$
\mathcal{L}_k v = D_{L(k)} v
$$
for any conformal Killing vector field $k$ on $(M,\cc)$ and any section $v$ of a natural bundle $\mathcal{V}=\mathcal{G}\times_P \mathbb{V}$ (this is proven in \cite{SeanMSc}). 

Using the results of \cite{CapGoTAMS} (and carefully comparing sign conventions) we have that
\begin{equation}\label{LieDensity}
D_\mathbb{L} \tau \overset{g}{=} \nabla_{l}\tau + w\nu \tau
\end{equation}
for sections $\tau$ of the density bundle $\ce[w]$ (which can be thought of as an associated bundle to $\mathcal{G}_0$ and hence to $\mathcal{G}$), where $\mathbb{L}$ is given in the scale $g$ by \nn{AdjTrac}. We also have that
$$
D_\mathbb{L} V = \nabla^\mathcal{T}_{\Pi(\mathbb{L})} V - \mathbb{L}(V)
$$
for all $\mathbb{L}\in\Gamma(\mathcal{A})$ and $V\in \Gamma(\mathcal{T})$. The operator $D_\mathbb{L}$ is easily seen to satisfy the Leibniz property and hence for any tractor field $T^{A\cdots B}{_{C\cdots D}}$ we have
\begin{align*}
D_\mathbb{L} T^{A\cdots B}{_{C\cdots D}} &=& \nabla^\mathcal{T}_{\Pi(\mathbb{L})} T^{A\cdots B}{_{C\cdots D}} - \mathbb{L}^A{_{A'}} T^{A'\cdots B}{_{C\cdots D}} - \cdots - \mathbb{L}^B{_{B'}} T^{A\cdots B'}{_{C\cdots D}}\\
& & + \mathbb{L}^{C'}{_C} T^{A\cdots B}{_{C'\cdots D}} + \cdots +\mathbb{L}^{D'}{_D} T^{A\cdots B}{_{C\cdots D'}}.
\end{align*}

From all of this we can finally write down an explicit formula for the Lie derivative of a standard tractor field $V^A$ in terms of ``slots'': if  $k$ is a conformal Killing vector field on $(M,\cc)$ and $V^A \overset{g}{=} (\sigma, \mu^a, \rho)$ then
\begin{eqnarray*}
\mathcal{L}_k V^A &=& k^b \nabla_b V^B + \mathbb{K}^A{_B} V^B \\ &\overset{g}{=}&
k^b\left(\begin{array}{c} 
\nabla_b \si - \mu_b  \\
\nabla_b\mu^a +\rho\delta^a_b +\si\V^a_b \\ 
\nabla_b \rho - \V_{ab}\mu^a  \\
\end{array}\right) +
\left(\begin{array}{ccc} 
 -\nu   &  k_b      &   0  \\
-\rho^a & \mu^a{_b} & -k^a \\ 
   0    &  \rho_b   &  \nu \\
\end{array}\right) 
\left(\begin{array}{c} 
\si  \\
\mu^b  \\ 
\rho \\
\end{array}\right)\\
&\overset{g}{=}&
\left(\begin{array}{c} 
k^b\nabla_b \si - \nu\si  \\
k^b\nabla_b\mu^a + \mu^a{_b}\mu^b - \si \nabla^a\nu  \\ 
k^b\nabla_b \rho + \nu\rho + \mu^a\nabla_a\nu  \\
\end{array}\right)
=
\left(\begin{array}{c} 
\mathcal{L}_k \si  \\
\mathcal{L}_k \mu^a - \si \nabla^a\nu  \\ 
\mathcal{L}_k \rho + \mu^a\nabla_a\nu  \\
\end{array}\right)
\end{eqnarray*}
where $\mathbb{K}=L(-k)$ and we have used that $\rho_a = \nabla_a \nu + \V_{ab}k^b$ and that $\mathcal{L}_k=D_{-\mathbb{K}}$ on densities so that by \nn{LieDensity} we have $\mathcal{L}_k\si = \nabla_k \si - \nu\sigma$, $\mathcal{L}_k\rho = \nabla_k \rho + \nu\rho$, and
\begin{eqnarray*}
\mathcal{L}_k \mu^a &=& (k^b\nabla_b \mu^a - \mu^b \nabla_b k^a) + \nu\mu^a \\
&=& k^b\nabla_b\mu^a - \mu^b(\mu_b{^a}+\nu\delta_b^a) + \nu\mu^a \\
&=& k^b\nabla_b \mu^a + \mu^a{_b}\mu^b
\end{eqnarray*}
since $\mu^a$ has conformal weight $-1$. 

Note that one can also, of course, calculate the expression for $\mathcal{L}_k$ on densities and on standard tractor fields directly from the definition (which was done in \cite{SeanMSc}) by looking at what you get when you pull back densities and standard tractors by the local flow of $k$. As a check of our above formula for $\mathcal{L}_k V^A$ we observe that in the case where $k$ is a Killing vector field for $g$ then the flow of $k$ preserves $g$ and hence also preserves the splitting tractors $X^A$, $Z^A_a$ and $Y^A$ so that $\mathcal{L}_k X^A=0$, $\mathcal{L}_k Z^A_a=0$, and $\mathcal{L}_k Y^A=0$; thus by the Leibniz property and linearity one immediately has that if $V^A=\si X^A + \mu^aZ^A_a + \rho Y^A$ then $\mathcal{L}_k V^A=(\mathcal{L}_k\si) X^A + (\mathcal{L}_k\mu^a)Z^A_a + (\mathcal{L}_k\rho) Y^A$ which is consistent with our above formula for $\mathcal{L}_k V^A$ since $\mathcal{L}_k g=0$ forces $\nu=\frac{1}{d}\nabla^g_ak^a$ to be zero.

\begin{remark}
On vector fields the fundamental derivative acts according to
$$
D_\mathbb{L} v^b = l^a\nabla_a v^b + (\mu_a{^b} + \nu \delta_a^b)v^a
$$
where $\mathbb{L}$ is given in the scale $g$ by \ref{AdjTrac} and $\nabla=\nabla^g$. Thus if $k$ is a conformal Killing vector field then applying $\mathcal{L}_k=D_{L(k)}$ on vector fields simply returns the usual formula for the Lie derivative in terms of a torsion free (in this case Levi-Civita) connection:
$$
\mathcal{L}_k v^b = k^a \nabla_a v^b - (\nabla_a k^b)v^a.
$$ 
Similarly, using the Leibniz property of the fundamental derivative, we have that
\begin{align*}
D_\mathbb{L} T^{b \cdots c}{_{d\cdots e}} &=& l^a\nabla_a T^{b \cdots c}{_{d\cdots e}} 
- (\mu_a{^b} + \nu \delta_a^b)T^{a \cdots c}{_{d\cdots e}} - \cdots - (\mu_a{^c} + \nu \delta_a^c)T^{b \cdots a}{_{d\cdots e}} \\
&& + (\mu_d{^a} + \nu \delta_d^a)T^{b \cdots c}{_{a\cdots e}}+\cdots + (\mu_e{^a} + \nu \delta_e^a)T^{b \cdots c}{_{d\cdots a}},
\end{align*}
and again applying $\mathcal{L}_k=D_{L(k)}$ for a conformal killing vector field $k$ simply yields the standard formula for the Lie derivative. These observations further demonstrate the consistency of our claims. What's more, we may now calculate the fundamental derivative of any weighted tensor-tractor field $T^{b \cdots c}{_{d\cdots e}}{^{B\cdots C}}{_{D\cdots E}}$. For instance, by writing $\bg_{ab}$ as $\sigma^2g_{ab}$ and using the Leibniz property one may easily show that
$$
D_\mathbb{L} \bg_{ab} = 0
$$
for all $\mathbb{L}\in \Gamma(\mathcal{A})$. \endrk
\end{remark}

\subsubsection{Static and stationary spacetimes} As an application of the above, we observe that if $g\in\cc$ is an Einstein metric with corresponding scale tractor $I$ and $k$ is a Killing vector field for $g$ with $\mathbb{K}=L(-k)$ then $\mathbb{K}(I)=0$. This follows from the fact $\mathcal{L}_k I=0$ since the flow of $k$ preserves $g$ (and hence also $I$), and from the above formula for the Lie derivative of $I$,
$$
\mathcal{L}_k I^A = \nabla_k I^A + \mathbb{K}^A{_B}I^B = \mathbb{K}^A{_B}I^B,
$$ 
since $I$ is parallel. Moreover, if $k$ is \emph{hypersurface orthogonal} (i.e. its orthogonal distribution is integrable) then it is easy to see that $\mathbb{K}_{[AB}\mathbb{K}_{C]D}X^D=0$ so that $K_{AB}$ is simple and can be written as $2v_{[A} K_{B]}$ where $K_B=X^A\mathbb{K}_{AB}$ and $v_A K^A =0$; on top of this, from $\mathbb{K}_{AB}I^B=0$ we obtain $K_A I^A=0$ and hence also $v_A I^A =0$. These observations form the starting point for the development of conformal tractor calculus adapted to static and stationary spacetimes. In particular, we note that in the case where $(M,g,k)$ is a static spacetime then $K_A = u N_A$ where $N_A$ is the normal tractor to each of the spacelike hypersurfaces in the foliation given by $k^\perp \subset TM$ and $u$ is the so called \emph{static potential} (if we trivialise the density bundles using $g$); thus $K_A I^A=0$ and $K_A v^A =0$ imply that both $I^A$ and $v^A$ lie in the intrinsic tractor bundle of the foliating spacelike hypersurfaces and one can carry out dimensional reduction using tractors. One can in fact still carry out dimensional reduction in the stationary case by identifying tractors which are Lie dragged by $k$ and orthogonal to $K_B=X^A\mathbb{K}_{AB}$ with conformal tractors on the manifold of integral curves of $k$ (see \cite{SeanMSc} for further details in both cases).

\end{document}